\documentclass[a4paper, 12pt, reqno]{amsart}

\usepackage{amsmath,amssymb,amsthm}
\usepackage{enumerate}
\usepackage{xcolor}
\usepackage[mathscr]{euscript}
\usepackage{url}
\usepackage[all]{xy}
\usepackage{booktabs}
\usepackage{graphicx}
\usepackage{epstopdf}
\usepackage{algorithmic}
\usepackage[caption=false]{subfig}
\usepackage{mathtools}

\usepackage[ruled,vlined]{algorithm2e}

\usepackage{hyperref}
\definecolor{cobalt}{rgb}{0.0, 0.28, 0.67}
\definecolor{darkblue}{rgb}{0.0, 0.0, 0.55}

\hypersetup
{
bookmarksopen=true,
pdftitle={Mon sujet de th¨¨se complet},
pdfauthor={Li LIU},
pdfsubject={Rapport de th¨¨se},
pdfmenubar=true,
pdfhighlight=/O,
colorlinks=true,
pdfpagemode=None,
pdfpagelayout=SinglePage,
pdffitwindow=true,
linkcolor=cobalt,
citecolor=cobalt,
urlcolor=cobalt
}

\usepackage[capitalize,nameinlink,noabbrev]{cleveref}

\usepackage[scale=0.75]{geometry}
\usepackage{tikz}
\usetikzlibrary{shapes}
\usetikzlibrary{arrows}
\usetikzlibrary{matrix}

\theoremstyle{definition}

\newtheorem{theorem}{Theorem}[section]
\newtheorem{corollary}[theorem]{Corollary}
\newtheorem{lemma}[theorem]{Lemma}
\newtheorem{definition}[theorem]{Definition}
\newtheorem{proposition}[theorem]{Proposition}
\newtheorem{remark}[theorem]{Remark}

\newtheorem{example}[theorem]{Example}

\definecolor{darkred}{rgb}{0.7,0,0}
\definecolor{darkgreen}{rgb}{0,0.46,0}
\definecolor{purple}{rgb}{0.6,0,0.5}

\newcommand{\tenss}[1]{\boldsymbol{\mathcal{#1}}}
\newcommand{\tenselem}[1]{\mathcal{#1}}

\newcommand{\matr}[1]{\boldsymbol{#1}}

\newcommand{\vect}[1]{\boldsymbol{#1}}

\newcommand{\set}[1]{\mathscr{#1}}
\newcommand{\T}{{\sf T}}        
\renewcommand{\H}{{\sf H}}      

\makeatletter
\newcommand{\rank}[1]{\mathop{\operator@font rank}(#1)}
\newcommand{\colrank}[1]{\mathop{\operator@font colrank}\{#1\}}
\newcommand{\krank}[1]{\mathop{\operator@font krank}\{#1\}}
\newcommand{\trace}[1]{\mathop{\operator@font tr}(#1)}
\newcommand{\symmm}[1]{\mathop{\operator@font sym}\left(#1\right)}
\newcommand{\skeww}[1]{\mathop{\operator@font skew}(#1)}
\newcommand{\Diag}[1]{\mathop{\operator@font Diag}\{#1\}}    
\newcommand{\diag}[1]{\mathop{\operator@font diag}\{#1\}}    
\newcommand{\Span}[1]{\mathop{\operator@font Span}\{#1\}}    
\newcommand{\argmin}{\mathop{\operator@font argmin}}

\newcommand{\offdiag}[1]{\mathop{\operator@font offdiag}\{#1\}}    
\newcommand{\Proj}[2]{\mathop{\operator@font Proj_{#1}}{#2}}
\newcommand{\ProjGrad}[2]{\mathop{{\operator@font grad} }#1(#2)}
\newcommand{\Hess}[2]{\mathrm{Hess_{}}{#1}(#2)}
\newcommand{\HessAppl}[3]{\mathrm{Hess_{}}{#1}(#2) [#3]}
\newcommand{\expp}[1]{\mathop{\operator@font exp}\left(#1\right)}
\makeatother

\newcommand{\eqdef}{\stackrel{\sf def}{=}}
\newcommand{\RR}{\mathbb{R}}
\newcommand{\CC}{\mathbb{C}}

\newcommand{\R}{\Re}
\newcommand{\I}{\Im}

\newcommand{\RInner}[2]{\left\langle #1, #2\right\rangle_{\Re}} 

\newcommand{\ui}{i}

\newcommand{\ON}[1]{\set{O}_{#1}}

\newcommand{\UN}[1]{\set{U}_{#1}}

\newcommand{\contr}[1]{\mathop{\bullet_{#1}}}   

\begin{document}

\title[Polar decomposition based algorithms]
{Polar decomposition based algorithms on the product of Stiefel manifolds with applications in tensor approximation}

\author[Jianze Li]{Jianze Li$^{\dagger}$}
\thanks{$\dagger$ Shenzhen Research Institute of Big Data, The Chinese University of Hong Kong, Shenzhen, China (lijianze@gmail.com)}

\author[Shuzhong Zhang]{Shuzhong Zhang$^{\ddagger}$}
\thanks{$\ddagger$ Department of Industrial and Systems Engineering, University of Minnesota, Minneapolis, MN 55455, USA (zhangs@umn.edu), and Institute for Data and Decision Analytics, The Chinese University of Hong Kong, Shenzhen, China (zhangs@cuhk.edu.cn)}

\date{\today}


\thanks{This work was supported in part by the National Natural Science Foundation of China (11601371).}


\begin{abstract}
In this paper, 
we propose a general algorithmic framework to solve a class of optimization problems on the product of complex Stiefel manifolds based on the matrix polar decomposition. 
We establish the weak convergence, global convergence and linear convergence rate of this general algorithmic approach using the \L{}ojasiewicz gradient inequality and the Morse-Bott property. This general algorithm and its convergence results are applied to the simultaneous approximate tensor diagonalization and simultaneous approximate tensor compression,
which include as special cases the low rank orthogonal approximation, best rank-1 approximation and low multilinear rank approximation for higher order complex tensors. 
We also present a symmetric variant of this general algorithm to solve a symmetric variant of this class of optimization models, which essentially optimizes over a single Stiefel manifold. We establish its weak convergence, global convergence and linear convergence rate in a similar way. This symmetric variant and its convergence results are applied to the simultaneous approximate symmetric tensor diagonalization, which includes as special cases the low rank symmetric orthogonal approximation and best symmetric rank-1 approximation for higher order complex symmetric tensors.
It turns out that well-known algorithms such as LROAT, S-LROAT, HOPM, S-HOPM are all special cases of this general algorithmic framework and its symmetric variant, and our convergence results subsume the results found in the literature designed for those special cases. 
All the algorithms and convergence results in this paper also apply to the real case. 
\end{abstract}

\maketitle

\section{Introduction}

Theory and algorithms for optimization over manifolds have been developed and applied widely because of its practical importance; see \cite{absil2009optimization,JMLR:v15:boumal14a,hu2019brief,smith1994optimization}.
In particular, many different algorithms, including the Newton-type methods \cite{adler2002newton}, trust-region algorithm \cite{absil2007trust} and \emph{alternating direction method of multipliers} (ADMM) \cite{kovnatsky2016madmm}, have been proposed for optimization over a single Stiefel manifold or the product of Stiefel manifolds, which capture the orthogonality type of constraints \cite{absil2009optimization,edelman1998geometry,journee2010generalized,sato2013riemannian,wen2013feasible}.

\subsection{Problem formulation}
For a matrix $\matr{X}\in\CC^{n\times r}$, we denote by $\matr{X}^{\T}$, $\matr{X}^{*}$ and $\matr{X}^{\H}$ its \emph{transpose, conjugate} and \emph{conjugate transpose,} respectively.   
Let 
$\text{St}(r,n,\CC) =\{\matr{X}\in\CC^{n\times r}: \matr{X}^{\H}\matr{X}=\matr{I}_r\}$ 
be the complex Stiefel manifold with $1\leq r\leq n$. 
In this paper, we mainly study the \emph{optimization problem on the product of complex Stiefel manifolds}, which is to maximize a smooth function 
\begin{equation}\label{definition-f}
f:\ \Omega \longrightarrow \RR^{+},
\end{equation}
where
\begin{equation}\label{eq:Omega}
\Omega\eqdef\text{St}(r_1,n_1,\CC)\times\text{St}(r_2,n_2,\CC)\times\cdots\times\text{St}(r_d,n_d,\CC)
\end{equation}
is the product of $d$ complex Stiefel manifolds ($d>1$). 
Let 
$$\Omega^{(i)}\eqdef\text{St}(r_1,n_1,\CC)\times\cdots\times\text{St}(r_{i-1},n_{i-1},\CC)\times\text{St}(r_{i+1},n_{i+1},\CC)\times\cdots\times\text{St}(r_d,n_d,\CC)$$ 
be the product of $d-1$ complex Stiefel manifolds and $\nu^{(i)}\in\Omega^{(i)}$.
Let 
$\matr{U}\in\text{St}(r_i,n_i,\CC)$. 
Define
\begin{align}
h_{(i)}:\ \text{St}(r_i,n_i,\CC) &\longrightarrow \RR^{+}, \notag\\
\matr{U} &\longmapsto f(\nu^{(i)},\matr{U})
\eqdef f(\matr{U}^{(1)}, \cdots, \matr{U}^{(i-1)}, \matr{U}, \matr{U}^{(i+1)}, \cdots, \matr{U}^{(d)}).\label{definition-h}
\end{align}
For simplicity, we also denote by $h_{(i)}$ the natural extension of itself to $\CC^{n_i\times r_i}$. 
In this paper, we say that the objective function \eqref{definition-f} is 
\emph{block multiconvex} \cite{li2015convergence,shen2017disciplined,xu2013block} if the restricted function $h_{(i)}$ is convex for any fixed $\nu^{(i)}\in\Omega^{(i)}$ and $1\leq i\leq d$.

\subsection{Symmetric variant}
Let $\omega\eqdef (\matr{U}^{(1)}, \matr{U}^{(2)}, \cdots, \matr{U}^{(d)})\in\Omega$. 
If $r_i = r$ and $n_i = n$ for all $1\leq i\leq d$, 
then $\Omega$ in \eqref{eq:Omega} becomes
\begin{equation}\label{eq:Omega_s}
\Omega_{s}\eqdef\text{St}(r,n,\CC)\times\text{St}(r,n,\CC)\times\cdots\times\text{St}(r,n,\CC).
\end{equation}
Assume that 
the objective function $f$ is \emph{symmetric} on $\Omega_{s}$, 
that is, 
$f(\omega) = f(\pi(\omega))$
for any $\omega\in\Omega_{s}$ and permutation $\pi$. 
In this paper, we also study the \emph{symmetric variant} of model \eqref{definition-f}, which is to maximize the smooth function
\begin{equation}\label{cost-function-genral-h}
g:\ \text{St}(r,n,\CC) \longrightarrow \RR^{+}, 
\ \matr{U} \longmapsto f(\matr{U}, \matr{U},\cdots, \matr{U}).
\end{equation}
For simplicity, we also denote by $g$ the natural extension of itself to $(\CC^{n\times r})^d$. 

\subsection{Notations}
Let $\CC^{n_1\times n_2\times\cdots\times n_d}\eqdef\CC^{n_1}\otimes\CC^{n_2}\otimes\cdots\otimes\CC^{n_d}$ be the linear space of $d$-th order complex tensors. 
Let $\text{symm}(\CC^{n\times n\times\cdots\times n})$ be the set of \emph{symmetric} ones in $\CC^{n\times n\times\cdots\times n}$,
whose entries do not change under any permutation of indices \cite{Comon08:symmetric,qi2017tensor}.
For a tensor $\tenss{A}\in\CC^{n_1\times n_2\times\cdots\times n_d}$ and a matrix $\matr{X}\in\CC^{m\times n_i}$, 
we adopt the \emph{$i$-mode product} defined as 
$(\tenss{A}\contr{i}\matr{X})_{p_1\cdots q\cdots p_d}\eqdef\sum_{p_i} \tenselem{A}_{p_1\cdots p_i\cdots p_d} X_{q p_i}.$ 
The \emph{$i$-mode unfolding} of $\tenss{A}$ is denoted by $\tenss{A}_{(i)}$, which is the matrix obtained by reordering the \emph{$i$-mode fibers}\footnote{The fiber of a tensor is defined by fixing every index but one.} of $\tenss{A}$ in a fixed way \cite{kolda2009tensor}. 
The \emph{diagonal} of a tensor 
$\tenss{A}\in\CC^{n_1\times n_2\times\cdots\times n_d}$ is defined as 
$\diag{\tenss{A}}\eqdef (\tenselem{A}_{11\cdots 1}, \cdots, \tenselem{A}_{nn\cdots n})^{\T}$, where $n=\min(n_1,\cdots,n_d)$. 
We say that $\tenss{A}$ is \emph{diagonal}, if its diagonal elements are the only nonzero elements. 
Let $\matr{I}_{n\times r}$ be the diagonal matrix in $\RR^{n\times r}$ satisfying that the diagonal elements are all equal to 1. 
Let $\matr{I}^{\perp}_{n\times r}$ be the matrix in $\RR^{n\times(n-r)}$ such that $[\matr{I}_{n\times r}\ \matr{I}^{\perp}_{n\times r}]=\matr{I}_{n}$. 
We denote by $\|\cdot\|$ the Frobenius norm of a tensor or matrix, or the Euclidean norm of a vector.
Let $\mathbb{S}_{n-1}\subseteq\CC^{n}$ be the unit sphere, and denote $\mathbb{S}=\mathbb{S}_0$ for simplicity. 
Let $\UN{n}=\text{St}(n,n,\CC)$ be the unitary group. 
We denote by $\vect{e}_{p}$ the $p$-th standard unit vector, and its dimension depends on the context.

\subsection{Objective functions of interest}
Let $\{\tenss{A}^{(\ell)}\}_{1\leq\ell\leq L}\subseteq\CC^{n_1\times n_2\times\cdots n_d}$ be a set of complex tensors, and
$\{\tenss{S}^{(\ell)}\}_{1\leq\ell\leq L}\subseteq\text{symm}(\CC^{n\times n\times\cdots\times n})$ be a set of $d$-th order complex symmetric tensors. 
Let $\alpha_\ell \in \RR^{+}$ for $1\leq\ell\leq L$.
Denote $(\cdot)^{\dagger}=(\cdot)^{\H}$ or $(\cdot)^{\T}$.
In this paper, we mainly focus on the following three tensor related objective functions:
\begin{itemize}
	\item[$\bullet$] \emph{Simultaneous approximate tensor diagonalization:} 
	\begin{align}\label{cost-fn-general-1}
	f(\omega) &= \sum\limits_{\ell=1}^{L} \alpha_\ell\|\diag{\tenss{W}^{(\ell)}}\|^2,\ \ \matr{U}^{(i)}\in\text{St}(r,n_i,\CC),\ \ 1\leq r\leq n_i,\ 1\leq i\leq d,\\
	\tenss{W}^{(\ell)}&=\tenss{A}^{(\ell)} \contr{1} (\matr{U}^{(1)})^{\dagger}\cdots \contr{d} (\matr{U}^{(d)})^{\dagger},\ 1\leq \ell\leq L;\notag
	\end{align}
	\item[$\bullet$] \emph{Simultaneous approximate symmetric tensor diagonalization:} 
	\begin{align}\label{cost-fn-general-1-s}
	g(\matr{U}) &= \sum\limits_{\ell=1}^{L} \alpha_\ell\|\diag{\tenss{W}^{(\ell)}}\|^2,\ \ \matr{U}\in\text{St}(r,n,\CC),\ \ 1\leq r\leq n,\\
	\tenss{W}^{(\ell)}&=\tenss{S}^{(\ell)} \contr{1} \matr{U}^\dagger\cdots \contr{d} \matr{U}^\dagger,\ 1\leq \ell\leq L;\notag
	\end{align}
	\item[$\bullet$] \emph{Simultaneous approximate tensor compression:} 
	\begin{align}
	f(\omega) &= \sum\limits_{\ell=1}^{L} \alpha_\ell\|\tenss{W}^{(\ell)}\|^2,\ \ \matr{U}^{(i)}\in\text{St}(r_i,n_i,\CC),\ \ 1\leq r_i\leq n_i,\ 1\leq i\leq d,\label{cost-fn-general-2}\\
	\tenss{W}^{(\ell)}&=\tenss{A}^{(\ell)} \contr{1} (\matr{U}^{(1)})^{\dagger} \cdots \contr{d} (\matr{U}^{(d)})^{\dagger},\ 1\leq \ell\leq L.\notag
	\end{align}
\end{itemize}

It can be seen that objective functions \eqref{cost-fn-general-1} and \eqref{cost-fn-general-2} are both of form \eqref{definition-f},
and objective function \eqref{cost-fn-general-1-s} is the symmetric variant of \eqref{cost-fn-general-1}. 
It will be shown in \Cref{subsec:eucli_prob_01} and \Cref{subsec:eucli_prob_03} 
that objective functions \eqref{cost-fn-general-1} and \eqref{cost-fn-general-2} are both block multiconvex, and in 
\Cref{lemma-convex-tenso-form}
that the objective function \eqref{cost-fn-general-1-s} is convex in many cases.

\subsection{Tensor approximations}
The above three objective functions include the following well-known approximation problems for higher order complex tensors \cite{Cichocki15:review,comon2014tensors,kolda2009tensor,sidiropoulos2017tensor} as special cases:
\begin{itemize}
	\item[$\bullet$] the objective function \eqref{cost-fn-general-1}:
	\begin{itemize}
		\item the \emph{low rank orthogonal approximation} \cite{chen2009tensor,kolda2001orthogonal}: $L=1$;
		\item the \emph{best rank-1 approximation} \cite{de1997signal,Lathauwer00:rank-1approximation}: $L=1$, $r=1$;
	\end{itemize}
	\item[$\bullet$] the objective function \eqref{cost-fn-general-1-s}:
	\begin{itemize}
		\item the \emph{low rank symmetric orthogonal approximation} \cite{chen2009tensor,pan2018symmetric,LUC2019}: $L=1$.
		\item the \emph{best symmetric rank-1 approximation} \cite{Lathauwer00:rank-1approximation,kofidis2002best}: $L=1$, $r=1$;
	\end{itemize}
	\item[$\bullet$] the objective function \eqref{cost-fn-general-2}:
	\begin{itemize}
		\item the \emph{low multilinear rank approximation} \cite{de1997signal,Lathauwer00:rank-1approximation}: $L=1$.
		\item the \emph{best rank-1 approximation} \cite{de1997signal,Lathauwer00:rank-1approximation}: $L=1$, $r_i=1$ $(1\leq i\leq d)$.
	\end{itemize}
\end{itemize}

We seperately consider the best rank-1 approximation and best symmetric rank-1 approximation here (also in \Cref{sec-best-rank-1-appro-11} and \Cref{sub-pro-sec-best-rank-1-appro-2}), because of their own interest and importance, although they are also special cases of the low rank orthogonal approximation, low multilinear rank approximation and their symmetric variants, respectively. 

These approximations have been widely used in various fields,
including \emph{signal processing} \cite{comon2014tensors,Como10:book,hu2017attribute,karami2012compression,li2019point}, \emph{numerical linear algebra} \cite{qi2017tensor,qi2009z} and \emph{data analysis} 
\cite{Anan14:latent,Cichocki15:review,kolda2009tensor,sidiropoulos2017tensor}. 
In particular, the best symmetric rank-1 approximation of a real symmetric tensor is corresponding to finding the largest Z-eigenvalue \cite{qi2009z}.
The low multilinear rank approximation is equivalent to the well-known \emph{Tucker decomposition} \cite{Lathauwer00:rank-1approximation,kolda2009tensor}, which has been a popular method for data reduction in signal processing and machine learning. 
When the rank is equal to the dimension (\emph{i.e.}, $r=n$),
problem \eqref{cost-fn-general-1-s} boils down to the \emph{simultaneous approximate symmetric tensor diagonalization} \cite{Como94:ifac,LUC2017globally,LUC2018,LUC2019-SAM,ULC2019} of complex symmetric tensors by unitary transformations, which is central in \emph{Independent Component Analysis} (ICA) \cite{Como92:elsevier,Como94:sp,de1997signal}. 
Therefore, it is desirable to develop a general algorithmic scheme to solve model \eqref{definition-f} and its symmetric variant \eqref{cost-function-genral-h}.

\subsection{Block coordinate descent}
In the literature, the following general block optimization model has been well studied,
which is to maximize 
a smooth function 
\begin{equation}\label{definition-f-general}
f:\ \mathcal{X}_1\times\mathcal{X}_2\times\cdots\times\mathcal{X}_d \longrightarrow \RR^{+},
\end{equation}
where $\mathcal{X}_i\subseteq\RR^{n_i}$ is a closed set for $1\leq i\leq d$. 
A popular approach to solve model \eqref{definition-f-general} is known as 
\emph{block coordinate descent} (BCD) \cite{bertsekas1997nonlinear}, with different ways to choose blocks for optimization. 
One natural choice is the \emph{cyclic} rule \cite{tseng2001convergence,wright2012accelerated}, which is also known as
the \emph{block nonlinear Gauss–Seidel} method. 
In this cyclic approach, the regularized optimization of block multiconvex function was studied in \cite{xu2013block}, with three different update schemes. 
Along a similar line, the so-called 
\emph{maximum block improvement} (MBI) method was proposed in \cite{chen2012maximum,li2015convergence} to optimize the objective function \eqref{definition-f-general}, which updates the block of variables corresponding to the maximally improving block at each iteration. Although MBI is more expensive than other methods, it was proved to have better theoretical convergence properties. 

\subsection{Contributions} 
In this paper, we mainly develop a general algorithmic scheme to solve model \eqref{definition-f} and its symmetric variant \eqref{cost-function-genral-h} based on the matrix polar decomposition, and establish the \emph{weak convergence}\footnote{Every accumulation point is a stationary point.},  \emph{global convergence}\footnote{For any starting point, the iterates converge as a whole sequence.} and linear convergence rate based on the \emph{\L{}ojasiewicz gradient inequality} and the \emph{Morse-Bott}\footnote{See \Cref{subsec:Morse} for a definition.} property. 
The main contributions of this paper can be summarized as follows:
\begin{itemize}
	\item[$\bullet$] Algorithm APDOI to solve model \eqref{definition-f}:
	\begin{itemize}
		\item Based on the matrix polar decomposition, 
		we propose an approach to be called APDOI (\emph{alternating polar decomposition based orthogonal iteration}) (\Cref{al-general-polar}) to solve model \eqref{definition-f}, as well as its \emph{shifted} version APDOI-S (\Cref{al-general-polar-S}). 
		As in the BCD method, our algorithms tackle one block at each time while fixing the remaining blocks, and choose the blocks in a cyclic manner. However, at each iteration, we update the block based on the matrix polar decomposition instead of optimization. This approach was motivated by the \emph{low rank orthogonal approximation of tensors} (LROAT) algorithm in \cite{chen2009tensor}, and can be also seen as a block coordinate version of the \emph{generalized power method} in \cite{journee2010generalized}. 
		\item We establish the weak convergence and global convergence of APDOI and its shifted version APDOI-S based on the \L{}ojasiewicz gradient inequality, when the objective function \eqref{definition-f} is
		block multiconvex. 
		The APDOI-S is proved to have better theoretical convergence properties than APDOI. 
		\item We apply APDOI and establish its convergence properties to objective functions \eqref{cost-fn-general-1} and \eqref{cost-fn-general-2}, which include the special cases: the low rank orthogonal approximation (\Cref{sub-pro-sec-best-rank-ort-appro-1}); the best rank-1 approximation (\Cref{sec-best-rank-1-appro-11});  the low multilinear rank approximation (\Cref{sub-pro-sec-best-rank-mul-appro}) for higher order complex tensors.
		It turns out that the well-known methods such as LROAT \cite{chen2009tensor} (for low rank orthogonal approximation) and \emph{higher order power method} (HOPM) \cite{DeLathauwerLieven1995,Lathauwer00:rank-1approximation} (for best rank-1 approximation) are both special cases of APDOI.
		Our convergence results subsume the results found in the literature designed for those special cases. 
		\item Algorithm APDOI for the low multilinear rank approximation will be called LMPD (\emph{low multilinear rank approximation based on polar decomposition}), 
		and its shifted version will be called LMPD-S.  
		Experiments are conducted showing that LMPD and LMPD-S have comparable speed of convergence as compared with the well-known \emph{higher order orthogonal iteration} (HOOI)  \cite{Lathauwer00:rank-1approximation} algorithm, and LMPD-S has an even much better convergence performance.
	\end{itemize}
	\item[$\bullet$] Algorithm PDOI to solve model \eqref{cost-function-genral-h}:
	\begin{itemize}
		\item As a symmetric variant of APDOI, 
		we propose a PDOI (\emph{polar decomposition based orthogonal iteration}) approach (\Cref{al-general-polar-general}) to solve model \eqref{cost-function-genral-h}, as well as its shifted version PDOI-S (\Cref{al-general-polar-general-s}). 
		\item We establish their weak convergence and global convergence in a similar way, when the objective function \eqref{cost-function-genral-h} is 
		convex. Algorithm PDOI-S is proved to have better theoretical convergence properties than PDOI. 
		\item Algorithm PDOI and its convergence results are applied to the objective function \eqref{cost-fn-general-1-s}, which includes the special cases: the low rank symmetric orthogonal approximation (\Cref{sub-pro-sec-best-rank-ort-appro-2}); the best symmetric rank-1 approximation (\Cref{sub-pro-sec-best-rank-1-appro-2}) for higher order complex symmetric tensors. 
		It turns out that the well-known algorithms such as S-LROAT \cite{chen2009tensor} (for low rank symmetric orthogonal approximation) and S-HOPM \cite{Lathauwer00:rank-1approximation} (for best symmetric rank-1 approximation) are both special cases of PDOI.
		Our convergence results also subsume the results found in the literature designed for those special cases. 
	\end{itemize}
	\item[$\bullet$] Linear convergence rate: 
	\begin{itemize}
		\item We first show that all the stationary points are not nondegenerate, if the objective function 
		is \emph{scale (or unitarily) invariant}\footnote{See \Cref{subsec:Rie_Hess} for a definition.}. 
		Then, based on the Morse-Bott property, we prove that the convergence rate is linear, if the limit is a \emph{scale (or unitarily) semi-nondegenerate point}\footnote{See \Cref{subsec:Rie_Hess} for a definition.}. 
		For the objective functions \eqref{cost-fn-general-1} and \eqref{cost-fn-general-1-s}, which are scale invariant,
		we show that there exist scale semi-nondegenerate points.
		For the objective function \eqref{cost-fn-general-2}, which is unitarily invariant,
		we show that there exist unitarily semi-nondegenerate points. 
		In fact, the Morse-Bott property was earlier used in \cite{ULC2019} to prove the linear convergence of Jacobi-type algorithm for simultaneous approximate diagonalization of complex tensors, which is defined on a single unitary group. In this paper, we extend this method to the product of Stiefel manifolds, introduce the unitarily semi-nondegenerate point and calculate the rank of Riemannian Hessian in a new way. 
	\end{itemize}
\end{itemize}

\begin{remark}
	(i) This paper is based on the complex Stiefel manifolds and complex tensors.
	In fact, all the algorithms and convergence results in this paper also apply to the real case. \\
	(ii) The LROAT, S-LROAT algorithms in \cite{chen2009tensor} and HOPM, S-HOPM, HOOI algorithms in \cite{Lathauwer00:rank-1approximation} were all developed for low rank approximations of real tensors.
	In this paper, for convenience, we also denote them to be the direct extensions of these algorithms 
	for complex tensors. 
\end{remark}

\subsection{Organization}

The paper is organized as follows. 
In \Cref{sec:geometri_Stiefel}, we recall the basics of first order and second order geometries on the Stiefel manifolds, as well as the \L{}ojasiewicz gradient inequality and the Morse-Bott property.
In \Cref{sec-opt-manifold-product}, we propose Algorithm APDOI and its shifted version APDOI-S to solve model \eqref{definition-f}, and establish their weak convergence, global convergence and linear convergence rate.
In \Cref{sec-opt-manifold}, as a symmetric variant of APDOI, we propose Algorithm PDOI and its shifted version PDOI-S to solve model \eqref{cost-function-genral-h}, and establish their weak convergence, global convergence and linear convergence rate. 
In \Cref{sec:objec_01,sec:objec_01-s,sec:objec_02}, we study the Riemannian gradient and Riemannian Hessian of objective functions \eqref{cost-fn-general-1}, \eqref{cost-fn-general-1-s} and \eqref{cost-fn-general-2}, respectively, and then apply the convergence results obtained in  \Cref{sec-opt-manifold-product} and \Cref{sec-opt-manifold} to the low rank orthogonal approximation, best rank-1 approximation, low multilinear rank approximation for higher order complex tensors and their symmetric variants.

\section{Geometries on the Stiefel manifolds}\label{sec:geometri_Stiefel}

\subsection{Riemannian gradient}
For $\matr{X} \in \CC^{n\times r}$, we write $\matr{X} =\matr{X}^{\R} + \ui \matr{X}^{\I}$ for the real and imaginary parts. 
For  $\matr{X},\matr{Y}\in \CC^{n\times r}$, we introduce the following real-valued inner product
\begin{equation}\label{eq:RInner}
\RInner{\matr{X}}{\matr{Y}} \eqdef 
\langle\matr{X}^{\R},\matr{Y}^{\R} \rangle + \langle\matr{X}^{\I},\matr{Y}^{\I} \rangle  =
\Re \left({\trace{\matr{X}^{\H}\matr{Y}}}\right),
\end{equation}
which makes $\CC^{n\times r}$ a real Euclidean space  of dimension $2nr$.
Let $h:\text{St}(r,n,\CC)\rightarrow\RR$ be a differentiable function and $\matr{U}\in\text{St}(r,n,\CC)$. 
We denote by  $\frac{\partial h}{\partial\matr{U}^\R},\frac{\partial h}{\partial\matr{U}^\I} \in \RR^{n\times r}$ the matrix Euclidean derivatives of $h$ with respect to real and imaginary parts of $\matr{U}$. 
The \emph{Wirtinger derivatives} \cite{abrudan2008steepest,brandwood1983complex,krantz2001function} are defined as
\begin{equation*}
\frac{\partial h}{\partial\matr{U}^{*}} \eqdef \frac{1}{2}\left(\frac{\partial h}{\partial\matr{U}^\R}+ \ui\frac{\partial h}{\partial\matr{U}^\I}\right), \quad 
\frac{\partial h}{\partial\matr{U}} \eqdef \frac{1}{2}\left(\frac{\partial h}{\partial\matr{U}^\R}- \ui\frac{\partial h}{\partial\matr{U}^\I}\right). 
\end{equation*}
Then the Euclidean gradient of $h$ with respect to the inner product \eqref{eq:RInner} becomes
\begin{equation}\label{eq:matr_Euci_grad}
\nabla h(\matr{U}) = \frac{\partial h}{\partial\matr{U}^\R}+ \ui\frac{\partial h}{\partial\matr{U}^\I} = 2 \frac{\partial h}{\partial\matr{U}^{*}}.
\end{equation}
For real matrices $\matr{X},\matr{Y}\in\RR^{n\times r}$, we see that \eqref{eq:RInner} becomes the standard inner product, and \eqref{eq:matr_Euci_grad} becomes the standard Euclidean gradient. 

For $\matr{X}\in\CC^{r\times r}$, we denote that 
\begin{equation*}
\symmm{\matr{X}}\eqdef\frac{1}{2}\left(\matr{X}+\matr{X}^{\H}\right),\ \  \skeww{\matr{X}}\eqdef\frac{1}{2}\left(\matr{X}-\matr{X}^{\H}\right).
\end{equation*} 
Let ${\rm\bf T}_{\matr{U}} \text{St}(r,n,\CC)$ be the \emph{tangent space} to $\text{St}(r,n,\CC)$ at a point $\matr{U}$. 
By \cite[Def. 6]{manton2001modified}, we know that 
\begin{align}
{\rm\bf T}_{\matr{U}} \text{St}(r,n,\CC) 
= \{\matr{Z}\in\CC^{n\times r}: \matr{Z}=\matr{U}\matr{A}+\matr{U}_{\perp}\matr{B},
\matr{A}\in\CC^{r\times r}, \matr{A}^{\H}+\matr{A}=0, \matr{B}\in\CC^{(n-r)\times r}\},\label{eq:tanget_spac}
\end{align} 
which is a $(2nr-r^2)$-dimensional vector space. 
The orthogonal projection of $\xi\in\CC^{n\times r}$ to ${\rm\bf T}_{\matr{U}} \text{St}(r,n,\CC)$ is 
\begin{equation}\label{eq:proj_Stiefel}
{\rm Proj}_{\matr{U}} \xi = (\matr{I}_{n}-\matr{U}\matr{U}^{\H})\xi + \matr{U}\skeww{\matr{U}^{\H}\xi} 
=\xi - \matr{U}\symmm{\matr{U}^{\H}\xi}.
\end{equation} 
We denote ${\rm Proj}^{\bot}_{\matr{U}} \xi =\xi - {\rm Proj}_{\matr{U}} \xi $, which is in fact the orthogonal projection of $\xi $ to the \emph{normal space}\footnote{This is the orthogonal complement of the tangent space; see \cite[Sec. 3.6.1]{absil2009optimization} for more details.}. 
Note that $\text{St}(r,n,\CC)$ is an embedded submanifold of the Euclidean space $\CC^{n\times r}$. 
By \eqref{eq:proj_Stiefel}, we have the \emph{Riemannian gradient}\footnote{See \cite[Eq. (3.31), (3.37)]{absil2009optimization} for a detailed definition.} of $h$ at $\matr{U}$ as:
\begin{align}\label{eq:Rie_grad}
\ProjGrad{h}{\matr{U}} = {\rm Proj}_{\matr{U}} \nabla h(\matr{U})
= \nabla h(\matr{U}) -\matr{U}\symmm{\matr{U}^{\H}\nabla h(\matr{U})}.
\end{align}
For the objective function \eqref{definition-f},
which is defined on the product of Stiefel manifolds, its Riemannian gradient can be computed as
\begin{align}
\ProjGrad{f}{\omega} = (\ProjGrad{h_{(1)}}{\matr{U}^{(1)}},\ProjGrad{h_{(2)}}{\matr{U}^{(2)}},\cdots,\ProjGrad{h_{(d)}}{\matr{U}^{(d)}}).
\end{align}

\begin{remark}\label{remar:symme_grad}
	Suppose that the objective function \eqref{definition-f} is symmetric on $\Omega_{s}$. 
	For any $\matr{U}_{0}\in\text{St}(r,n,\CC)$, 
	we define 
	\begin{equation*}
	h:\ \text{St}(r,n,\CC) \longrightarrow \RR^{+},\ 
	\matr{U} \longmapsto f(\matr{U}_{0}, \cdots, \matr{U}_{0}, \matr{U}).
	\end{equation*}
	Then we get that 
	\begin{equation}\label{eq-grad-rela-g-h}
	\nabla g(\matr{U}_{0}) = d\nabla h(\matr{U}_{0}).
	\end{equation} 
	It follows that the Riemannian gradients in \eqref{eq:Rie_grad} also satisfy
	$\ProjGrad{g}{\matr{U}_0} = d\ProjGrad{h}{\matr{U}_0}.$
\end{remark}

\subsection{Riemannian Hessian}\label{subsec:Rie_Hess}

Let $h:\text{St}(r,n,\CC)\rightarrow\RR$ be a differentiable function and $\matr{U}\in\text{St}(r,n,\CC)$. 
By \cite[Ex. 5.4.2]{absil2009optimization},
the exponential map at $\matr{U}$ is defined as
\begin{align*}
{\rm Exp}_{\matr{U}}: {\rm\bf T}_{\matr{U}} \text{St}(r,n,\CC)&\longrightarrow \text{St}(r,n,\CC)\\
\matr{Z}&\longmapsto [\matr{U}, \matr{Z}]
\expp{
	\begin{bmatrix}
	\matr{U}^{H}\matr{Z} &-\matr{Z}^{\H}\matr{Z}\\
	\matr{I}_{r} &\matr{U}^{H}\matr{Z}
	\end{bmatrix}}
\begin{bmatrix}
\expp{-\matr{U}^{H}\matr{Z}} \\
\matr{0}_{r\times r}
\end{bmatrix}.
\end{align*}
Then, based on \cite[Prop. 5.5.4]{absil2009optimization}, 
the \emph{Riemannian Hessian} $\Hess{h}{\matr{U}}$ can be seen as a linear map ${\rm\bf T}_{\matr{U}} \text{St}(r,n,\CC) \to {\rm\bf T}_{\matr{U}} \text{St}(r,n,\CC)$  defined by
\[
\Hess{h}{\matr{U}} = \nabla^2(h \circ\text{Exp}_{\matr{U}})(\matr{0}_{\matr{U}}),
\]
where $\matr{0}_{\matr{U}}$ is the origin in the tangent space and $\nabla^2$ is the Euclidean Hessian. 
In fact, by \cite[Eq. (8)--(10)]{Absil2013}, the Riemannian Hessian is  a sum of  the projection of the Euclidean Hessian on the tangent space
and a second term given by the Weingarten operator 
\begin{equation}\label{eq:Hessian_reduction}
\HessAppl{h}{\matr{U}}{\matr{Z}} = {\rm Proj}_{\matr{U}} \nabla^2 h(\matr{U}) [\matr{Z} ]  +
\mathfrak{A}_{\matr{U}}(\matr{Z}, {\rm Proj}^{\bot}_{\matr{U}} \nabla h(\matr{U})), 
\end{equation}
where the \emph{Weingarten operator}\footnote{This is similar to the case of real Stiefel manifold in \cite{Absil2013}.} for $\text{St}(r,n,\CC)$ is given by
\begin{equation*}
\mathfrak{A}_{\matr{U}}(\matr{Z},\matr{V}) = -\matr{Z}\matr{U}^{\H}\matr{V} -
\matr{U} \symmm{\matr{Z}^{\H} \matr{V}}.
\end{equation*}
For the objective function \eqref{definition-f}, its Riemannian Hessian can be computed as
\begin{align}
{\rm Hess}f(\omega) = ({\rm Hess}h_{(1)}(\matr{U}^{(1)}),{\rm Hess}h_{(2)}(\matr{U}^{(2)}),\cdots,{\rm Hess}h_{(d)}(\matr{U}^{(d)})).
\end{align}

Note that $\mathbb{S}\subseteq\CC$ is the set of unimodular complex numbers in this paper. 
We say that $h$ is \emph{scale invariant}, if  
$h(\matr{U}) = h(\matr{U}\matr{R})$ for all
\begin{equation}\label{eq:invariance_scaling-0}
\matr{R}\in\set{DU}_{r} \eqdef\left\{
\begin{bmatrix}z_1 & & 0 \\ & \ddots & \\ 0 & & z_r \end{bmatrix}: z_p\in\mathbb{S},\ 1\leq p\leq r\right\}\subseteq\UN{r}.
\end{equation} 
We say that $h$ is \emph{unitarily invariant}, if  
$h(\matr{U}) = h(\matr{U}\matr{R})$ for all
$\matr{R}\in\UN{r}$. 
We say that $f$ in \eqref{definition-f} is \emph{scale (or unitarily) invariant}, if $h_{(i)}$ is scale (or unitarily) invariant for all $1\leq i\leq d$. 
It is clear that the objective functions \eqref{cost-fn-general-1} and \eqref{cost-fn-general-1-s} are both scale invariant, and the objective function \eqref{cost-fn-general-2} is unitarily invariant (also scale invariant). 

A point $\matr{U}\in\text{St}(r,n,\CC)$ is called a \emph{stationary point}, if $\ProjGrad{h}{\matr{U}} = 0$.
A stationary point is called \emph{nondegenerate} if $\Hess{h}{\matr{U}}$ is nonsingular on $\mathbf{T}_{\matr{U}}\text{St}(r,n,\CC)$.
Now we show that, if $h$ is scale (or unitarily) invariant, then all the stationary points of $h$ are not nondegenerate.  
This result is an easy extension of \cite[Lem. 3.7]{ULC2019}. 
Note that $\vect{e}_{p}$ is the $p$-th standard unit vector in this paper.

\begin{lemma}\label{lem:f_invariant_Hessian}
	Suppose that $h: \text{St}(r,n,\CC) \to \RR$ is scale invariant and $\matr{U}$ is a stationary point. 
	Let $\matr{Z}_p = \matr{U} \matr{\Omega}_p \in \mathbf{T}_{\matr{U}} \text{St}(r,n,\CC)$, where $\matr{\Omega}_p = i \vect{e}_p \vect{e}_p^{\T}$ for $1\leq p\leq r$. 
	Then  $\HessAppl{h}{\matr{U}}{\matr{Z}_p} = \vect{0}$ for $1\leq p\leq r$. 
	In particular, 
	we have $\rank{\Hess{h}{\matr{U}}} \le 2nr-r^2-r$. 
\end{lemma}

\begin{proof}
	Let $\gamma:t\mapsto \text{Exp}_{\matr{U}}(t\matr{Z}_p)$ be a curve. 
	It can be calculated that $\gamma(t)=\matr{U}\expp{t\matr{\Omega}_p}$. 
	Then $\gamma(0)=\matr{U}$ and $\gamma'(0)=\matr{Z}_p$.
	By \cite[Def. 5.5.1]{absil2009optimization} and \cite[Eq. (5.17)]{absil2009optimization}, we see that 
	\begin{align*}
	\HessAppl{h}{\matr{U}}{\matr{Z}_p}&=
	\nabla_{\matr{Z}_p}\textrm{grad}h=
	{\rm Proj}_{\matr{U}}\left({\bf D}\ProjGrad{h}{\matr{U}}[\matr{Z}_p]\right)
	= {\rm Proj}_{\matr{U}}\left(\frac{d}{d t}{\rm grad}h(\gamma(t))|_{t=0}\right)\\
	&= {\rm Proj}_{\matr{U}}\left(\frac{d}{d t}{\rm grad}h(\matr{U})\expp{t\matr{\Omega}_p}|_{t=0}\right)
	=0.
	\end{align*}
	The proof is complete.  
\end{proof}

The next result for unitarily invariant functions can be proved similarly. 

\begin{lemma}\label{lem:f_invariant_Hessian-2}
	Suppose that $h: \text{St}(r,n,\CC) \to \RR$ is unitarily invariant and $\matr{U}$ is a stationary point.
	Let $\matr{\Omega}_{p,q} = i (\vect{e}_p \vect{e}^{\T}_q+\vect{e}_q \vect{e}^{\T}_p)$ and $\matr{\Omega}_{p,q}^{'} = (\vect{e}_p \vect{e}_q^{\T}-\vect{e}_q \vect{e}_p^{\T})$ for $1\leq p<q\leq r$.
	Let $\matr{Z}_{p,q} = \matr{U} \matr{\Omega}_{p,q}$ and $\matr{Z}^{'}_{p,q} = \matr{U} \matr{\Omega}_{p,q}^{'}$. 
	Then  $\HessAppl{h}{\matr{U}}{\matr{Z}_{p,q}} = \vect{0}$ and $\HessAppl{h}{\matr{U}}{\matr{Z}^{'}_{p,q}} = \vect{0}$ for $1\leq p<q\leq r$.  
	In particular, 
	we have $\rank{\Hess{h}{\matr{U}}} \le 2r(n-r)$. 
\end{lemma}


By \Cref{lem:f_invariant_Hessian} and \Cref{lem:f_invariant_Hessian-2}, we know that there exists no nondegenerate stationary point for objective functions \eqref{cost-fn-general-1}, \eqref{cost-fn-general-1-s} and \eqref{cost-fn-general-2}, as they are all scale invariant or unitarily invariant. 
In this paper, we say that a stationary point $\matr{U}$ is a \emph{scale (or unitarily) semi-nondegenerate point} of $h$, if $h$ is scale (or unitarily) invariant and $\rank{\Hess{h}{\matr{U}}} = 2nr-r^2-r$ (or $2r(n-r)$). 
We say that a stationary point $\omega$ is a \emph{scale (or unitarily) semi-nondegenerate point} of $f$ in \eqref{definition-f}, if $\matr{U}^{(i)}$ is a scale (or unitarily) semi-nondegenerate point of the restricted function $h_{(i)}$ for all $1\leq i\leq d$. 
It will be seen in \Cref{prop:hess_semi_stric_01} and  \Cref{prop:hess_semi_stric_02-0} that 
there exist scale semi-nondegenerate points for objective functions \eqref{cost-fn-general-1} and \eqref{cost-fn-general-1-s}, and in \Cref{prop:countexam_02} that there exist unitarily semi-nondegenerate points for objective function \eqref{cost-fn-general-2}.  
Now we end this subsection with a lemma about the calculation of Euclidean Hessian, which can be proved by \cite[Eq. (7), (10), (33)]{van1994complex}.

\begin{lemma}
	Let $h:\CC^n\rightarrow\RR$ be a smooth function. 
	Let $\vect{u}\in\CC^n$ be a complex vector variable. 
	Define 
	\begin{align*}
	\nabla^2_{c} h(\vect{u})\eqdef\frac{\partial^2 h}{\partial\vect{u}^{*}\partial\vect{u}^{\T}}\ \ \textrm{and}\ \  
	\nabla^2_{r} h(\vect{u})\eqdef\frac{\partial^2 h}{\partial\vect{u}\partial\vect{u}^{\T}}.
	\end{align*} 
	Then, for $\vect{z}\in\CC^n$, we have  
	\begin{equation}\label{eq:Hess_computation}
	\nabla^2 h(\vect{u})[\vect{z}] = 2\left(\nabla^2_{c} h(\vect{u})\vect{z} + (\nabla^2_{r} h(\vect{u})\vect{z})^{*}\right).
	\end{equation} 
\end{lemma}

\subsection{\L{}ojasiewicz gradient inequality}
In this subsection, we present some results about the \L{}ojasiewicz gradient inequality \cite{loja1965ensembles,lojasiewicz1993geometrie,AbsMA05:sjo,Usch15:pjo}.
These results were used in \cite{LUC2018,ULC2019} to prove the global convergence of Jacobi-G algorithms on the orthogonal and unitary groups.

\begin{definition} [{\cite[Def. 2.1]{SU15:pro}}]\label{def:Lojasiewicz}
	Let $\mathcal{M} \subseteq \RR^n$ be a Riemannian submanifold,
	and $f: \mathcal{M} \to \RR$ be a differentiable function.
	The function $f: \mathcal{M} \to \RR$ is said to satisfy a \emph{\L{}ojasiewicz gradient inequality} at $\vect{x} \in \mathcal{M}$, if there exist
	$\sigma>0$, $\zeta\in (0,\frac{1}{2}]$ and a neighborhood $\mathcal{U}$ in $\mathcal{M}$ of $\vect{x}$ such that for all $\vect{y}\in\mathcal{U}$, it follows that 
	\begin{equation}\label{eq:Lojasiewicz}
	|{f}(\vect{y})-{f}(\vect{x})|^{1-\zeta}\leq \sigma\|\ProjGrad{f}{\vect{y}}\|.
	\end{equation}
\end{definition}

\begin{lemma}[{\cite[Prop. 2.2]{SU15:pro}}]\label{lemma-SU15}
	Let $\mathcal{M}\subseteq\RR^n$ be an analytic submanifold\footnote{See {\cite[Def. 2.7.1]{krantz2002primer}} or \cite[Def. 5.1]{LUC2018} for a definition of an analytic submanifold.} and $f: \mathcal{M} \to \RR$ be a real analytic function.
	Then for any $\vect{x}\in \mathcal{M}$, $f$ satisfies a \L{}ojasiewicz gradient inequality \eqref{eq:Lojasiewicz} in the $\delta$-neighborhood of $\vect{x}$, for 
	some\footnote{The values of $\delta,\sigma,\zeta$ depend on the specific point in question.} $\delta,\sigma>0$ and $\zeta\in (0,\frac{1}{2}]$.	
\end{lemma}


\begin{theorem}[{\cite[Thm.  2.3]{SU15:pro}}]\label{theorem-SU15}
	Let $\mathcal{M}\subseteq\RR^n$ be an analytic submanifold
	and 
	$\{\vect{x}_k\}_{k\geq 1}\subseteq\mathcal{M}$.
	Suppose that $f$ is real analytic and, for large enough $k$,\\
	(i) there exists $\sigma>0$ such that
	\begin{equation*}\label{eq:sufficient_descent}
	|{f}(\vect{x}_{k+1})-{f}(\vect{x}_k)|\geq \sigma\|\ProjGrad{f}{\vect{x}_k}\|\|\vect{x}_{k+1}-\vect{x}_{k}\|;
	\end{equation*}
	(ii) $\ProjGrad{f}{\vect{x}_k}=0$ implies that $\vect{x}_{k+1}=\vect{x}_{k}$.\\
	Then any accumulation point $\vect{x}_*$ of $\{\vect{x}_k\}_{k\geq 1}$ must be the only limit point. 
	If, in addition, for some $\kappa > 0$ and for large enough  $k$  it holds that 
	\begin{equation}\label{eq:safeguard}
	\|\vect{x}_{k+1} - \vect{x}_{k}\| \ge \kappa \|\ProjGrad{f}{\vect{x}_k}\|,
	\end{equation}
	then the following convergence rates apply
	\begin{equation}\label{eq:convergence_speed}
	\|\vect{x}_{k} - \vect{x}^*\| \le \begin{cases}
	Ce^{-ck}, & \text{ if } \zeta = \frac{1}{2}; \\
	Ck^{-\frac{\zeta}{1-2\zeta}}, & \text{ if } 0 < \zeta < \frac{1}{2},
	\end{cases}
	\end{equation}
	where $\zeta$ is the parameter in \eqref{eq:Lojasiewicz} at the limit point $\vect{x}_*$, and $C>0$, $c>0$ are some constants. 
\end{theorem}

\begin{remark}\label{remark-condition-change}
	It can be verified, after checking the proof in \cite[Thm.  3.2]{AbsMA05:sjo} and \cite[Thm.  2.3]{SU15:pro}, that condition (ii) in \Cref{theorem-SU15} can be replaced by that $f(\vect{x}_{k+1})=f(\vect{x}_{k})$ implies $\vect{x}_{k+1}=\vect{x}_{k}$. 
\end{remark}

\subsection{The Morse-Bott property}\label{subsec:Morse}

If $\zeta = \frac{1}{2}$ in \eqref{eq:Lojasiewicz},
then the \L{}ojasiewicz gradient inequality is called the \emph{Polyak-\L{}ojasiewicz inequality} \cite{Poly63:gradient}.
It can be seen in \Cref{theorem-SU15} that 
the convergence rate is linear if $\zeta = \frac{1}{2}$ at the limit point $\vect{x}_*$. 
Now we first recall the Morse-Bott property \cite{bott1954,feehan2018optimal,ULC2019}.

\begin{definition}[{\cite[Def.  1.5]{feehan2018optimal}}]\label{def-local-morse-bott}
	Let $\mathcal{M}$ be a $C^\infty$ submanifold and $f: \mathcal{M} \to \RR$ be a $C^{2}$ function. Denote the set of stationary points as 
	$${\rm Crit} f = \{\vect{x} \in \mathcal{M} : \ProjGrad{f}{\vect{x}} = 0\}.$$
	The function $f$ is said to be \emph{Morse-Bott} at $\vect{x}_0\in\mathcal{M}$ if there exists an open neighborhood $\set{U} \subseteq \mathcal{M}$ of $\vect{x}_0$ such that
	\begin{enumerate}[(i)]
		\item $\set{C} = \set{U} \cap {\rm Crit}f$ is a relatively open, smooth submanifold of $\mathcal{M}$;
		\item $\mathbf{T}_{\vect{x}_0}\set{C} = {\rm Ker}\ \Hess{f}{\vect{x}_0} $.
	\end{enumerate}
\end{definition}

\begin{remark}[{\cite[Rem. 6.6]{ULC2019}}]\label{remMorseBott}
	(i) If $\vect{x}_0\in\mathcal{M}$ is a nondegenerate stationary point, then $f$ is Morse-Bott at $\vect{x}_0$, since $\{\vect{x}_0\}$ is a zero-dimensional manifold in this case.\\
	(ii) If $\vect{x}_0\in\mathcal{M}$ is a degenerate stationary point, then condition (ii) in \Cref{def-local-morse-bott} can be rephrased\footnote{This is due to the fact that $\mathbf{T}_{\vect{x}_0}\mathcal{C} \subseteq {\rm Ker}\ \Hess{f}{\vect{x}_0}$.} as 
	\begin{equation}\label{eq:Hessian_rank_condition}
	\rank{\Hess{f}{\vect{x}_0}} = \dim \mathcal{M} - \dim \set{C}.
	\end{equation}
\end{remark}

For the functions that satisfy the Morse-Bott property, it was recently shown that the Polyak-\L{}ojasiewicz inequality holds true.
\begin{theorem}[{\cite[Thm.  3, Cor. 5]{feehan2018optimal}}]\label{thm:PLMorseBottEuclidean}
	If $\set{U}\subseteq\RR^{n}$ is an open subset and $f: \set{U} \to \RR$ is Morse-Bott at a stationary point $x$, then there exist $\delta, \sigma>0$ such that
	\begin{equation*}
	|{f}(y)-{f}(x)| \leq \sigma \|\nabla{f}(y)\|^2
	\end{equation*}
	for any $y\in \set{U}$ satisfying $\|y-x\| \le \delta$.
\end{theorem}

This result can be easily extended to the smooth manifold $\mathcal{M}$.

\begin{proposition}[{\cite[Prop. 6.8]{ULC2019}}]\label{cor:PLMorseBottManifold}
	If $\set{U}\subseteq\mathcal{M}$ is an open subset and a $C^2$ function $f: \set{U} \to \RR$ is Morse-Bott at a stationary point $x$, then there exist an open neighborhood $\set{V} \subseteq \set{U}$ of $x$ and $\sigma>0$ such that for all $y \in \set{V}$ it holds that
	\begin{equation*}
	|{f}(y)-{f}(x)| \leq \sigma \|\ProjGrad{f}{y}\|^2.
	\end{equation*}
\end{proposition}

\section{Algorithm APDOI on the product of Stiefel manifolds}\label{sec-opt-manifold-product}

In this section, based on the matrix polar decomposition, we propose a general algorithm and its shifted version to solve model \eqref{definition-f},
and establish their convergence. 
These two algorithms and convergence results will be applied to the objective functions \eqref{cost-fn-general-1} and \eqref{cost-fn-general-2} in \Cref{sec:objec_01} and \Cref{sec:objec_02}, respectively.

\subsection{Algorithm APDOI}
By \eqref{eq:Rie_grad}, if $\matr{U}^{(i)}_{*}$ is a maximal point of $h_{(i)}$,
then it should satisfy 
\begin{equation}\label{eq-gradient-polar-decom}
\nabla h_{(i)}(\matr{U}^{(i)}_{*}) = \matr{U}^{(i)}_{*}\symmm{(\matr{U}^{(i)}_{*})^{\H}\nabla h_{(i)}(\matr{U}^{(i)}_{*})}, 
\end{equation}
which is close to a \emph{polar decomposition} \cite{golub2012matrix,higham2008functions,horn2012matrix} form of $\nabla h_{(i)}(\matr{U}^{(i)}_{*})$.

\begin{lemma}[{\cite[Thm.  9.4.1]{golub2012matrix}, \cite[Thm.  8.1]{higham2008functions}, \cite[Thm.  7.3.1]{horn2012matrix}}]\label{lemma-polar-decom}
	Let $\matr{X}\in\CC^{n\times r}$ with $1\leq r\leq n$.
	There exist $\matr{U}\in\text{St}(r,n,\CC)$ and a unique Hermitian positive semidefinite matrix $\matr{P}\in \CC^{r\times r}$ such that $\matr{X}$ has the \emph{polar decomposition} $\matr{X}=\matr{U}\matr{P}$.	
	We say that $\matr{U}$ is the \emph{orthogonal polar factor} and $\matr{P}$ is the \emph{Hermitian polar factor}. 
	Moreover,\\
	(i) for any $\matr{U}'\in\text{St}(r,n,\CC)$, we have \cite[pp. 217]{higham2008functions}
	\begin{equation*}
	\langle\matr{U}, \matr{X}\rangle_{\R}\geq\langle\matr{U}', \matr{X}\rangle_{\R};
	\end{equation*}
	(ii) $\matr{U}$ is the best orthogonal approximation \cite[Thm.  8.4]{higham2008functions} to $\matr{X}$,
	that is, for any $\matr{U}'\in\text{St}(r,n,\CC)$, we have
	\begin{equation*}
	\|\matr{X}-\matr{U}\| \leq \|\matr{X}-\matr{U}'\|; 
	\end{equation*}
	(iii) if $\text{rank}(\matr{X})=r$,
	then $\matr{P}$ is positive definite and $\matr{U}$ is unique \cite[Thm.  8.1]{higham2008functions}.
\end{lemma}

Let $k\geq 1$ and $1\leq i\leq d$. 
We denote that
\begin{align*}
&\nu^{(k,i)} \eqdef (\matr{U}_{k}^{(1)},\cdots, \matr{U}_{k}^{(i-1)}, \matr{U}_{k-1}^{(i+1)}, \cdots, \matr{U}_{k-1}^{(d)})\in\Omega^{(i)},\\
&\omega^{(k,i)} \eqdef (\matr{U}_{k}^{(1)},\cdots,\matr{U}_{k}^{(i-1)},\matr{U}_{k}^{(i)},\matr{U}_{k-1}^{(i+1)},\cdots,\matr{U}_{k-1}^{(d)})\in\Omega,
\end{align*}
and $\omega^{(k)}\eqdef\omega^{(k,d)}=\omega^{(k+1,0)}$. 
Define 
\begin{align*}\label{definition-hki}
h_{(k,i)}:\ \text{St}(r_i,n_i,\CC) &\longrightarrow \RR^{+}, \notag\\
\matr{U} &\longmapsto f(\nu^{(k,i)},\matr{U})
\eqdef f(\matr{U}_{k}^{(1)}, \cdots, \matr{U}_{k}^{(i-1)}, \matr{U}, \matr{U}_{k-1}^{(i+1)}, \cdots, \matr{U}_{k-1}^{(d)}).
\end{align*}
Inspired by the decomposition form in \eqref{eq-gradient-polar-decom}, we propose the following \emph{alternating polar decomposition based orthogonal iteration} (APDOI) algorithm to solve model \eqref{definition-f}.
We assume that the starting point $\omega^{(0)}$ satisfies\footnote{This condition will be used in \Cref{sec:condition_conver}.} that $f(\omega^{(0)})>0$.

\begin{algorithm}
	\caption{APDOI}\label{al-general-polar}
	\begin{algorithmic}[1]
		\STATE{{\bf Input:} starting point $\omega^{(0)}$.}
		\STATE{{\bf Output:} $\omega^{(k,i)}$, $k\geq 1$, $1\leq i\leq d$.}
		\FOR{$k=1,2,\cdots,$}
		\FOR{$i=1,2,\cdots,d$}
		\STATE Compute $\nabla h_{(k,i)}(\matr{U}^{(i)}_{k-1})$;
		\STATE Compute the polar decomposition of $\nabla h_{(k,i)}(\matr{U}^{(i)}_{k-1})$;
		\STATE Update $\matr{U}_{k}^{(i)}$ to be the orthogonal polar factor.
		\ENDFOR
		\ENDFOR
	\end{algorithmic}
\end{algorithm}

\begin{remark} \label{remark-smaller-delta}
	Note that the objective function \eqref{definition-f} is smooth and $\Omega$ is compact. 
	Let $\Delta_{1}\eqdef\max_{\omega\in\Omega} \|\nabla f(\omega)\|$.
	Then $\|\nabla h_{(k,i)}(\matr{U}^{(i)}_{k-1})\|\leq\Delta_{1}$ always holds in \Cref{al-general-polar}. 
\end{remark}

\begin{lemma}\label{lemma-equiva-no-move-general}
	Let $h:\text{St}(r,n,\CC)\rightarrow\RR$ be a differentiable function and $\matr{U}\in\text{St}(r,n,\CC)$.
	Suppose that $\matr{U}_{*}$ is the orthogonal polar factor of $\nabla h(\matr{U})$.\\
	(i) If $\matr{U}_{*}=\matr{U}$, then $\ProjGrad{h}{\matr{U}}=0$.\\
	(ii) If $\ProjGrad{h}{\matr{U}}=0$, then $\matr{U}^{\H}\nabla h(\matr{U})$ is Hermitian. Furthermore, if $\nabla h(\matr{U})$ is of full rank and $\matr{U}^{\H}\nabla h(\matr{U})$ is positive semidefinite,
	then $\matr{U}_{*}=\matr{U}$. 
\end{lemma}
\begin{proof}
	(i) Since $\matr{U}_{*}=\matr{U}$, we see that 
	$\nabla h(\matr{U}) = \matr{U}\matr{P}$ 
	is a polar decomposition.
	It follows that $\matr{U}^{\H}\nabla h(\matr{U})=\matr{P}$ is Hermitian,
	and thus 
	\begin{equation*}
	\nabla h(\matr{U}) = \matr{U}\matr{P} = \matr{U}\symmm{\matr{U}^{\H}\nabla h(\matr{U})}.
	\end{equation*}
	Therefore, by equation \eqref{eq:Rie_grad},
	we have $\ProjGrad{h}{\matr{U}} =0$.\\
	(ii) By equation \eqref{eq:Rie_grad}, we see that 
	$\nabla h(\matr{U}) = \matr{U}\symmm{\matr{U}^{\H}\nabla h(\matr{U})}$ and thus 
	$$2\matr{U}^{\H}\nabla h(\matr{U}) =\matr{U}^{\H}\nabla h(\matr{U}) + \nabla h(\matr{U})^{\H}\matr{U},$$
	which means that $\matr{U}^{\H}\nabla h(\matr{U})$ is Herimitian. 
	Note that $\matr{U}^{\H}\nabla h(\matr{U})$ is positive semidefinite and $\nabla h(\matr{U})$ is of full rank.
	Then $\nabla h(\matr{U}) = \matr{U}\matr{U}^{\H}\nabla h(\matr{U})$ is the unique polar decomposition of $\nabla h(\matr{U})$ by \Cref{lemma-polar-decom}(iii).
	Therefore, we have $\matr{U}_{*}=\matr{U}$. 
	The proof is complete. 
\end{proof}

\begin{remark}
	By \Cref{lemma-equiva-no-move-general},
	we see that, in \Cref{al-general-polar},
	if the Euclidean gradient $\nabla h_{(k,i)}(\matr{U}^{(i)}_{k-1})$ is of full rank and $\symmm{(\matr{U}^{(i)}_{k-1})^{\H}\nabla h_{(k,i)}(\matr{U}^{(i)}_{k-1})}$ is positive semidefinite, 
	then $\matr{U}^{(i)}_{k}=\matr{U}^{(i)}_{k-1}$ if and only if 
	$\ProjGrad{h_{(k,i)}}{\matr{U}^{(i)}_{k-1}}=0$.
\end{remark}

\subsection{Convergence analysis}\label{subsec-cost-increase-general-product}
Let $\sigma_{\rm min}$ be the minimal singular value of a matrix. 
For \Cref{al-general-polar}, we mainly prove the following results about its weak convergence, global convergence and convergence rate. 

\begin{theorem}\label{theorem-weak-conver-product}
	Suppose that the objective function \eqref{definition-f} is 
	block multiconvex. 
	If there exists $\delta>0$ such that 
	\begin{equation}\label{main-condition-convergence}
	\sigma_{\rm min}(\nabla h_{(k,i)}(\matr{U}_{k-1}^{(i)}))>\delta
	\end{equation}
	for all $1\leq i\leq d$ and $k\geq 1$ 
	in \Cref{al-general-polar}, then every accumulation point of the iterates $\omega^{(k,i)}$ is a stationary point. 
\end{theorem}

\begin{theorem}\label{theor-isolate-limit}
	Suppose that the objective function \eqref{definition-f} is 
	block multiconvex and real analytic. 
	If, in \Cref{al-general-polar} there exists $\delta>0$ such that condition \eqref{main-condition-convergence}
	always holds,
	then the iterates $\omega^{(k,i)}$ converge to a stationary point $\omega_{*}$.
	If $f$ is scale (or unitarily) invariant and $\omega_{*}$ is a scale (or unitarily) semi-nondegenerate point,
	then the convergence rate is linear. 
\end{theorem}

Let $\nabla_{i}f(\omega)\in\CC^{n_i\times r_i}$ be the partial derivative of the objective function \eqref{definition-f} with respect to the $i$-th block matrix $\matr{U}^{(i)}$ at $\omega$ point. 
Note that \eqref{definition-f} is smooth and $\Omega$ is compact. There exists $\rho>0$ such that
\begin{equation}\label{eq-L}
\|\nabla_{i}f(\omega)-\nabla_{i}f(\omega')\|\leq \rho\|\omega-\omega'\|
\end{equation}
for any $\omega,\omega'\in\Omega$ and $1\leq i\leq d$.
Define 
\begin{align*}\label{definition-hki-1}
p_{(k,i)}:\ \text{St}(r_i,n_i,\CC) &\longrightarrow \RR^{+}, \notag\\
\matr{U} &\longmapsto f(\matr{U}_{k-1}^{(1)}, \cdots, \matr{U}_{k-1}^{(i-1)}, \matr{U}, \matr{U}_{k-1}^{(i+1)}, \cdots, \matr{U}_{k-1}^{(d)}).
\end{align*}
Let $r_{\rm max}\eqdef\max(r_1,\cdots,r_d)$. 
Now we need some lemmas before proving \Cref{theorem-weak-conver-product} and \Cref{theor-isolate-limit}. 

\begin{lemma}\label{lemma-incresease-cost-general}
	Suppose that the objective function \eqref{definition-f} is block multiconvex. 
	Then, the objective value \eqref{definition-f} monotonically increases and converges in \Cref{al-general-polar}.
\end{lemma}

\begin{proof}
	By the \emph{gradient inequality} \cite[Eq. (3.2)]{boyd2004convex} and \Cref{lemma-polar-decom}(i),
	we have 
	\begin{equation}\label{eq-cost-increase-general-1-symme-general}
	h_{(k,i)}(\matr{U}_{k}^{(i)}) - h_{(k,i)}(\matr{U}_{k-1}^{(i)}) 
	\geq \langle \matr{U}_{k}^{(i)}-\matr{U}_{k-1}^{(i)}, \nabla h_{(k,i)}(\matr{U}_{k-1}^{(i)})\rangle_{\R}\geq 0.
	\end{equation}
	Note that the objective function \eqref{definition-f} is continuous, and thus upper bounded. 
	The proof is complete. 
\end{proof}

\begin{lemma}\label{lemma-inquality-01-general}
	Let $\matr{U}, \matr{U}'\in\text{St}(r,n,\CC)$ and $\matr{P}\in\CC^{r\times r}$ be Hermitian positive semidefinite.
	Suppose that the minimal singular value (also minimal eigenvalue) $\sigma_{\rm min}(\matr{P})>0$.
	Then 
	\begin{equation*}
	\langle\matr{U}-\matr{U}', \matr{U}\matr{P}\rangle_{\R}\geq\frac{1}{2}
	\sigma_{\rm min}(\matr{P})
	\|\matr{U}-\matr{U}'\|^2.
	\end{equation*}
\end{lemma}
\begin{proof}
	Let $\matr{P}=\matr{Q}^{\H}\matr{D}\matr{Q}$ be the spectral decomposition, where $\matr{D}$ is a diagonal matrix with diagonal elements $(\sigma_1,\cdots,\sigma_r)$. 
	Denote that $\matr{U}\matr{Q}^{\H}=[\vect{u}_1, \cdots, \vect{u}_r]$ and 
	$\matr{U}'\matr{Q}^{\H}=[\vect{u}'_1, \cdots, \vect{u}'_r]$.
	Then 
	\begin{equation*}
	\langle\matr{U}-\matr{U}', \matr{U}\matr{P}\rangle_{\R} = 
	\langle\matr{U}\matr{Q}^{\H}-\matr{U}'\matr{Q}^{\H}, \matr{U}\matr{Q}^{\H}\matr{D}\rangle_{\R} = \sum_{i=1}^{r} \sigma_i(1-\langle \vect{u}_{i}, \vect{u}'_{i}\rangle_{\R}).
	\end{equation*}
	Note that
	$\|\matr{U}-\matr{U}'\|^2 = 2\sum_{i=1}^{r} (1-\langle \vect{u}_{i}, \vect{u}'_{i}\rangle_{\R}).$
	The proof is complete. 
\end{proof}

\begin{lemma}\label{lemma-main-01}
	Suppose that the objective function \eqref{definition-f} is block multiconvex. 
	If there exists $\delta>0$ such that condition \eqref{main-condition-convergence} always holds in \Cref{al-general-polar},
	then 
	\begin{equation*}
	h_{(k,i)}(\matr{U}_{k}^{(i)}) - h_{(k,i)}(\matr{U}_{k-1}^{(i)}) \geq \frac{\delta}{2} \|\matr{U}_{k}^{(i)}-\matr{U}_{k-1}^{(i)}\|^{2}.
	\end{equation*}
\end{lemma}

\begin{proof}
	Let $\nabla h_{(k,i)}(\matr{U}_{k-1}^{(i)})=\matr{U}_{k}^{(i)}\matr{P}$ be the polar decomposition. 
	Then, by \eqref{eq-cost-increase-general-1-symme-general} and \Cref{lemma-inquality-01-general},
	we have 
	\begin{align*}
	h_{(k,i)}(\matr{U}_{k}^{(i)}) - h_{(k,i)}(\matr{U}_{k-1}^{(i)})
	\geq \langle \matr{U}_{k}^{(i)}-\matr{U}_{k-1}^{(i)}, \nabla h_{(k,i)}(\matr{U}_{k-1}^{(i)})\rangle_{\R}
	\geq \frac{\delta}{2}\|\matr{U}_{k}^{(i)}-\matr{U}_{k-1}^{(i)}\|^2.
	\end{align*}
	The proof is complete. 
\end{proof}

\begin{corollary}\label{corollary-main-01}
	Suppose that the objective function \eqref{definition-f} is block multiconvex. 
	In \Cref{al-general-polar},
	if we have $\sigma_{\rm min}(\nabla h_{(k,i)}(\matr{U}_{k-1}^{(i)}))>0$,
	then
	$f(\omega^{(k,i)}) = f(\omega^{(k,i-1)})$ implies that $\omega^{(k,i)}=\omega^{(k,i-1)}$.
\end{corollary}

\begin{corollary}\label{coro-big-improvement}
	Suppose that the objective function \eqref{definition-f} is block multiconvex. 
	If there exists $\delta>0$ such that condition \eqref{main-condition-convergence} always holds in \Cref{al-general-polar},
	then 
	\begin{equation*}
	f(\omega^{(k)}) - f(\omega^{(k-1)}) \geq \frac{\delta}{2} \|\omega^{(k)}-\omega^{(k-1)}\|^{2}.
	\end{equation*}
\end{corollary}

\begin{lemma}\label{lemma-inquality-02-09}
	Let $h:\text{St}(r,n,\CC)\rightarrow\RR$ be a differentiable function and $\matr{U}\in\text{St}(r,n,\CC)$. 
	Suppose that $\nabla h(\matr{U})\neq\matr{0}$ and $\nabla h(\matr{U})=\matr{U}_{*}\matr{P}$ is the polar decomposition.
	Then
	\begin{equation*}
	\|\matr{U}_{*}-\matr{U}\|\geq \frac{\|\ProjGrad{h}{\matr{U}}\|}{(r+1)\|\nabla h(\matr{U})\|}.
	\end{equation*}
\end{lemma}
\begin{proof}
	Note that $\|\nabla h(\matr{U})\|=\|\matr{P}\|$. 
	By equation \eqref{eq:Rie_grad},
	we see that
	\begin{align*}
	\|\ProjGrad{h}{\matr{U}}\| &= \frac{1}{2}\|(\nabla h(\matr{U})-\matr{U}\matr{U}^{\H}\nabla h(\matr{U}))+(\nabla h(\matr{U})-\matr{U}\nabla h(\matr{U})^{\H}\matr{U})\| \\
	&\leq \frac{1}{2}\|(\matr{U}_{*}-\matr{U}) + \matr{U}\matr{U}^{\H} (\matr{U} - \matr{U}_{*})\|\|\matr{P}\|\\  
	&\quad+\frac{1}{2}\|(\matr{U}_{*}-\matr{U})\matr{P}+\matr{U}\matr{P}(\matr{U}^{\H}-\matr{U}^{\H}_{*})\matr{U}\| \\
	&\leq (r+1)\|\nabla h(\matr{U})\|\|\matr{U}_{*}-\matr{U}\|.
	\end{align*}
	The proof is complete. 
\end{proof}

\begin{lemma}\label{lemma-098}
	Suppose that the objective function \eqref{definition-f} is block multiconvex. 
	If there exists $\delta>0$ such that condition \eqref{main-condition-convergence} always holds in \Cref{al-general-polar},
	then 
	\begin{equation*}
	\|\omega^{(k)} - \omega^{(k-1)}\|\geq 
	\frac{\|\ProjGrad{f}{\omega^{(k-1)}}\|}{\sqrt{d}(1+r_{\rm max})(\rho+\Delta_{1})}.
	\end{equation*}
\end{lemma}

\begin{proof}
	By \eqref{eq:Rie_grad}, \eqref{eq-L} and \Cref{lemma-inquality-02-09}, for any $1\leq i\leq d$, we have  
	\begin{align*}
	\|\ProjGrad{p_{(k,i)}}{\matr{U}_{k-1}^{(i)}}\| 
	&= \|\nabla p_{(k,i)}(\matr{U}_{k-1}^{(i)}) -\matr{U}_{k-1}^{(i)}\symmm{(\matr{U}_{k-1}^{(i)})^{\H}\nabla p_{(k,i)}(\matr{U}_{k-1}^{(i)})}\|\\
	&\leq\|\nabla p_{(k,i)}(\matr{U}_{k-1}^{(i)}) - \nabla h_{(k,i)}(\matr{U}_{k-1}^{(i)})\| \\
	&\ \ \ +\|\nabla h_{(k,i)}(\matr{U}_{k-1}^{(i)}) - \matr{U}_{k-1}^{(i)}\symmm{(\matr{U}_{k-1}^{(i)})^{\H}\nabla h_{(k,i)}(\matr{U}_{k-1}^{(i)})}\|\\
	&\ \ \ +
	\|\matr{U}_{k-1}^{(i)}\symmm{(\matr{U}_{k-1}^{(i)})^{\H}\nabla h_{(k,i)}(\matr{U}_{k-1}^{(i)})}\\
	&\ \ \  -\matr{U}_{k-1}^{(i)}\symmm{(\matr{U}_{k-1}^{(i)})^{\H}\nabla p_{(k,i)}(\matr{U}_{k-1}^{(i)})}\|\\
	& \leq (1+\|\matr{U}_{k-1}^{(i)}\|^2) \|\nabla_{i} f(\omega^{(k-1)}) - \nabla_{i} f(\omega^{(k,i-1)}) \|
	+\|\ProjGrad{h_{(k,i)}}{\matr{U}_{k-1}^{(i)}}\|\\
	&\leq (1+r_{\rm max})\rho\|\omega^{(k-1)} - \omega^{(k)}\|
	+ (1+r_{\rm max})\Delta_{1}\|\matr{U}_{k-1}^{(i)} - \matr{U}_{k}^{(i)}\|\\
	&\leq (1+r_{\rm max})(\rho+\Delta_{1})\|\omega^{(k-1)} - \omega^{(k)}\|.
	\end{align*}
	It follows that 
	\begin{equation*}
	\|\ProjGrad{f}{\omega^{(k-1)}} \| 
	\leq \sqrt{d}(1+r_{\rm max})(\rho+\Delta_{1})\|\omega^{(k)} - \omega^{(k-1)}\|.
	\end{equation*}
	The proof is complete. 
\end{proof}

\paragraph{Proof of \Cref{theorem-weak-conver-product}}
It follows directly from 
\Cref{lemma-main-01}, \Cref{lemma-inquality-02-09} and \Cref{remark-smaller-delta} that 
\begin{equation}\label{eq:proof_weak}
h_{(k,i)}(\matr{U}_{k}^{(i)}) - h_{(k,i)}(\matr{U}_{k-1}^{(i)}) \geq 
\frac{\delta}{2(1+r_{\rm max})^2\Delta_{1}^2}
\|\ProjGrad{h_{(k,i)}}{\matr{U}_{k-1}^{(i)}}\|^2.
\end{equation}
Let $\omega_{*}=(\matr{U}_{*}^{(1)}, \matr{U}_{*}^{(2)}, \cdots, 
\matr{U}_{*}^{(d)})$ be an accumulation point of the iterates produced by \Cref{al-general-polar}.
Assume that 
\begin{equation*}
\ProjGrad{f}{\omega_{*}} = (\ProjGrad{h_{(1)}}{\matr{U}^{(1)}_{*}},\cdots, \ProjGrad{h_{(d)}}{\matr{U}^{(d)}_{*}})\neq 0.
\end{equation*}
Suppose $\|\ProjGrad{h_{(d)}}{\matr{U}^{(d)}_{*}}\|\neq 0$ without loss of generality. 
Note that $\omega_{*}$ is an accumulation point. There exists a subsequence of the iterates $\omega^{(k,i)}$ converging to $\omega_{*}$.
In this subsequence, 
we can choose $i_0$ such that $\omega^{(k,i_0)}$ occurs in the subsequence infinite times. 
Without loss of generality, assume $i_0=1$ and denote these elements by $\omega^{(k_\ell,1)}$. 
Note that $\|\matr{U}_{k}^{(i)}-\matr{U}_{k-1}^{(i)}\|\rightarrow 0$ by \Cref{lemma-main-01}. 
We see that $\omega^{(k_\ell,d-1)}$ also converges to $\omega_{*}$. 
Therefore, we obtain  $$\|\ProjGrad{h_{(k_{\ell},d)}}{\matr{U}^{(d)}_{k_{\ell}-1}}\|\rightarrow\|\ProjGrad{h_{(d)}}{\matr{U}^{(d)}_{*}}\|\neq 0,$$
which contradicts \eqref{eq:proof_weak}.
The proof is complete.  

\paragraph{Proof of \Cref{theor-isolate-limit}}
It follows from \Cref{coro-big-improvement} and \Cref{lemma-098} that 
\begin{equation}\label{eq-condi-980}
f(\omega^{(k)}) - f(\omega^{(k-1)}) \geq 
\frac{\delta}{2\sqrt{d}(1+r_{\rm max})(\rho+\Delta_{1})}
\|\omega^{(k)} - \omega^{(k-1)}\|\|\ProjGrad{f}{\omega^{(k-1)}}\|.
\end{equation}
It is clear that $f(\omega^{(k+1)}) = f(\omega^{(k)})$ implies that $\omega^{(k+1)} = \omega^{(k)}$ by \Cref{coro-big-improvement}. 
Then, by \Cref{theorem-SU15} and \Cref{remark-condition-change}, it follows that the iterates $\omega^{(k)}$ converge to $\omega_{*}$. 
Moreover, by \Cref{theorem-weak-conver-product}, we see that $\omega_{*}$ is a stationary point. 
Note that $\|\omega^{(k,i)}-\omega^{(k,i-1)}\|\rightarrow 0$ by \Cref{lemma-main-01}. 
It is not difficult to see that the iterates $\omega^{(k,i)}$ also converge to $\omega_{*}$.

If $f$ is scale (or unitarily) invariant and $\omega_{*}$ is a scale (or unitarily) semi-nondegenerate point,
by \Cref{def-local-morse-bott} and \Cref{remMorseBott}(ii), 
we obtain that $f$ is Morse-Bott at $\omega_{*}\eqdef (\matr{U}_{*}^{(1)}, \matr{U}_{*}^{(2)}, \cdots, \matr{U}_{*}^{(d)})\in\Omega$
by letting 
$$\set{C} \eqdef\left\{(\matr{U}_{*}^{(1)}\matr{R}_{1},  \cdots, \matr{U}_{*}^{(d)}\matr{R}_{d}):
\matr{R}_{i} \in \set{DU}_{r_i}\ (\textrm{or}\ \UN{r_i}),\ 1\leq i\leq d\right\}.$$
Then the proof is complete by \Cref{cor:PLMorseBottManifold}, \Cref{lemma-098} and \Cref{theorem-SU15}.  

\subsection{Algorithm APDOI-S}\label{sec:APDOI_S}
We now propose the following shifted version of \Cref{al-general-polar} to rid the dependence of convergence results on condition \eqref{main-condition-convergence}.
This variant is called Algorithm APDOI-S.

\begin{algorithm}
	\caption{APDOI-S}\label{al-general-polar-S}
	\begin{algorithmic}[1]
		\STATE{{\bf Input:} starting point $\omega^{(0)}$, $\gamma>\Delta_{1}$.}
		\STATE{{\bf Output:} $\omega^{(k,i)}$, $k\geq 1$, $1\leq i\leq d$.}
		\FOR{$k=1,2,\cdots,$}
		\FOR{$i=1,2,\cdots,d$}
		\STATE Compute $\nabla h_{(k,i)}(\matr{U}^{(i)}_{k-1})$;
		\STATE Compute the polar decomposition of $\nabla h_{(k,i)}(\matr{U}^{(i)}_{k-1})+\gamma\matr{U}^{(i)}_{k-1}$;
		\STATE Update $\matr{U}_{k}^{(i)}$ to be the orthogonal polar factor.
		\ENDFOR
		\ENDFOR
	\end{algorithmic}
\end{algorithm}

\begin{lemma}\label{lem:sigma_delta}
	Let $\delta=\gamma-\Delta_{1}>0$. 
	Then, in \Cref{al-general-polar-S}, we always have 
	\begin{equation}\label{eq:shifited_conditon}
	\sigma_{\rm min}(\nabla h_{(k,i)}(\matr{U}_{k-1}^{(i)})+\gamma\matr{U}^{(i)}_{k-1})>\delta.
	\end{equation}
\end{lemma}

\begin{proof}
	For any $\vect{x}\in\mathbb{S}^{r-1}$, we set $\vect{v}_1=\nabla h_{(k,i)}(\matr{U}_{k-1}^{(i)})\vect{x}$ and $\vect{v}_2=\matr{U}^{(i)}_{k-1}\vect{x}$. 
	It is clear that $\|\vect{v}_1\|\leq \Delta_1$ and $\|\vect{v}_2\|=1$. 
	Then we have  
	\begin{align*}
	\vect{x}^{\H}(\nabla h_{(k,i)}(\matr{U}_{k-1}^{(i)})+\gamma\matr{U}^{(i)}_{k-1})^{\H}(\nabla h_{(k,i)}(\matr{U}_{k-1}^{(i)})+\gamma\matr{U}^{(i)}_{k-1})\vect{x}
	\geq \|\vect{v}_1\|^2+\gamma^2-2\gamma\|\vect{v}_1\|
	\geq (\gamma-\Delta_1)^2.
	\end{align*}
	The proof is complete.  
\end{proof}

For \Cref{al-general-polar-S}, based on \eqref{eq:shifited_conditon}, we can prove the following convergence result without further assumption.

\begin{theorem}\label{theor-isolate-limit-s}
	Suppose that the objective function  \eqref{definition-f} is block multiconvex and real analytic. 
	Then the iterates $\omega^{(k,i)}$ produced by \Cref{al-general-polar-S} converge to a stationary point $\omega_{*}$.
	If $f$ is scale (or unitarily) invariant and $\omega_{*}$ is a scale (or unitarily) semi-nondegenerate point,
	then the convergence rate is linear. 
\end{theorem}

Since the proof of \Cref{theor-isolate-limit-s} is very similar to the derivations in \Cref{subsec-cost-increase-general-product},
we omit the details here, and only present some important intermediate lemmas.

\begin{lemma}
	Suppose that the objective function \eqref{definition-f} is block multiconvex. 
	Then the objective value of the iterates produced by \Cref{al-general-polar-S} monotonically increases and converges. 
\end{lemma}
\begin{proof}
	By the proof of \Cref{lemma-incresease-cost-general},
	we only need to note that 
	\begin{equation*}
	\langle \matr{U}_{k}^{(i)}-\matr{U}_{k-1}^{(i)}, \nabla h_{(k,i)}(\matr{U}_{k-1}^{(i)})\rangle_{\R} 
	\geq \langle \matr{U}_{k}^{(i)}-\matr{U}_{k-1}^{(i)}, \nabla h_{(k,i)}(\matr{U}_{k-1}^{(i)})+\gamma\matr{U}^{(i)}_{k-1}\rangle_{\R}\geq 0.
	\end{equation*}
	The proof is complete.  
\end{proof}

\begin{lemma}
	Let $h:\text{St}(r,n,\CC)\rightarrow\RR$ be a differentiable function and $\matr{U}\in\text{St}(r,n,\CC)$. 
	Suppose that $\nabla h(\matr{U})+\gamma\matr{U}=\matr{U}_{*}\matr{P}$ is the polar decomposition.
	Then
	\begin{equation*}
	\|\matr{U}_{*}-\matr{U}\|\geq \frac{\|\ProjGrad{h}{\matr{U}}\|}{(r+1)\|\nabla h(\matr{U})+\gamma\matr{U}\|}.
	\end{equation*}
\end{lemma}

\begin{lemma}
	Suppose that the objective function \eqref{definition-f} is block multiconvex. 
	Let $\delta=\gamma-\Delta_{1}$. 
	In \Cref{al-general-polar-S}, for any $k\geq 1$ and $1\leq i\leq d$, we have 
	\begin{align*}
	&\text{(i)}\ h_{(k,i)}(\matr{U}_{k}^{(i)}) - h_{(k,i)}(\matr{U}_{k-1}^{(i)}) \geq \frac{\delta}{2} \|\matr{U}_{k}^{(i)}-\matr{U}_{k-1}^{(i)}\|^{2};\\
	&\text{(ii)}\ \|\omega^{(k)} - \omega^{(k-1)}\|\geq 
	\frac{\|\ProjGrad{f}{\omega^{(k-1)}}\|}{\sqrt{d}(1+r_{\rm max})(\rho+\Delta_{1}+ \gamma\sqrt{r_{\rm max}})};\\
	&\text{(iii)}\ f(\omega^{(k)}) - f(\omega^{(k-1)}) \geq 
	\frac{\delta\|\ProjGrad{f}{\omega^{(k-1)}}\|}{2\sqrt{d}(1+r_{\rm max})(\rho+\Delta_{1}+\gamma\sqrt{r_{\rm max}})}
	\|\omega^{(k)} - \omega^{(k-1)}\|.
	\end{align*}
\end{lemma}

\subsection{Special case: the product of unit spheres}\label{sec:condition_conver}

We say that the objective function $f$ in  \eqref{definition-f} is \emph{restricted $\frac{1}{\alpha}-$homogeneous}, if there is some fixed $\alpha\in (0,1)$ such that, for any $\nu^{(i)}\in\Omega^{(i)}$ and $1\leq i\leq d$, it holds that 
\begin{align}\label{eq-general-assumption-product}
h_{(i)}(\matr{U}) = \alpha\langle \matr{U}, \nabla h_{(i)}(\matr{U})\rangle_{\R},
\end{align}
for any $\matr{U}\in\text{St}(r_i,n_i,\CC)$. 
If $r_i=1$ for all $1\leq i\leq d$ in \eqref{eq:Omega}, 
then $f$ is defined on the product of $d$ unit spheres.
In this case, if $f$ is restricted $\frac{1}{\alpha}-$homogeneous, 
we have 
\begin{equation*}
\|\nabla h_{(k,i)}(\matr{U}_{k-1}^{(i)})\|\geq\langle\matr{U}_{k-1}^{(i)}, \nabla h_{(k,i)}(\matr{U}_{k-1}^{(i)})\rangle_{\R} = \frac{1}{\alpha}h_{(k,i)}(\matr{U}_{k-1}^{(i)})\geq \frac{1}{\alpha}f(\omega^{(0)})>0
\end{equation*}
by equation \eqref{eq-general-assumption-product} and \Cref{lemma-incresease-cost-general}. 
Note that $\nabla h_{(k,i)}(\matr{U}_{k-1}^{(i)})$ is a vector, and thus has only one singular value. 
Therefore, the condition \eqref{main-condition-convergence} is always satisfied with $\delta=\frac{1}{2\alpha}f(\omega^{(0)})$. 
The next result follows directly from \Cref{theor-isolate-limit}.

\begin{corollary}\label{coro:conver_sphere}
	Let $r_i=1$ for all $1\leq i\leq d$.
	Suppose that the objective function \eqref{definition-f} is restricted $\frac{1}{\alpha}-$homogeneous, block multiconvex and real analytic. 
	Then the iterates $\omega^{(k,i)}$ produced by \Cref{al-general-polar} converge to a stationary point $\omega_{*}$.
	If $f$ is scale (or unitarily) invariant and $\omega_{*}$ is a scale (or unitarily) semi-nondegenerate point,
	then the convergence rate is linear. 
\end{corollary}

\section{Algorithm PDOI on the Stiefel manifold}\label{sec-opt-manifold}

In this section, based on the matrix polar decomposition, we propose a symmetric variant of \Cref{al-general-polar} and its shifted version to solve model \eqref{cost-function-genral-h},
and establish their convergence. 
These two algorithms and convergence results will be applied to the objective function \eqref{cost-fn-general-1-s} in \Cref{sec:objec_01-s}. 

\subsection{Algorithm PDOI}

By a similar argument as in \eqref{eq-gradient-polar-decom},
if $\matr{U}_{*}$ is a maximal point of the function $g$ in \eqref{cost-function-genral-h},
then it should satisfy 
\begin{equation}\label{eq-gradient-polar-decom-general}
\nabla g(\matr{U}_{*}) = \matr{U}_{*}\symmm{\matr{U}_{*}^{\H}\nabla g(\matr{U}_{*})},
\end{equation}
which is close to a polar decomposition form of $\nabla g(\matr{U}_{*})$.
Inspired by this observation, we propose the following \emph{polar decomposition based orthogonal iteration} (PDOI) algorithm to solve model \eqref{cost-function-genral-h}.
We assume that the starting point $\matr{U}_{0}$ satisfies\footnote{This condition will be used in \Cref{subsection:spe_unit_sphere}.} that $g(\matr{U}_{0})>0$.

\begin{algorithm}
	\caption{PDOI}\label{al-general-polar-general}
	\begin{algorithmic}[1]
		\STATE{{\bf Input:} starting point $\matr{U}_{0}$.}
		\STATE{{\bf Output:} $\matr{U}_{k}$, $k\geq 1$.}
		\FOR{$k=1,2,\cdots,$}
		\STATE Compute $\nabla g(\matr{U}_{k-1})$;
		\STATE Compute the polar decomposition of $\nabla g(\matr{U}_{k-1})$;
		\STATE Update $\matr{U}_{k}$ to be the orthogonal polar factor.
		\ENDFOR
	\end{algorithmic}
\end{algorithm}

\begin{remark} \label{remark-smaller-delta-1}
	Note that the objective function \eqref{cost-function-genral-h} is smooth and $\text{St}(r,n,\CC)$ is compact. 
	We denote 
	$\Delta_{2}\eqdef\max_{\matr{U}\in\text{St}(r,n,\CC)} \|\nabla g(\matr{U})\|.$
\end{remark}

\begin{remark}
	By \Cref{lemma-equiva-no-move-general},
	we know that, in \Cref{al-general-polar-general},
	if the gradient $\nabla g(\matr{U}_{k-1})$ is of full rank and $\symmm{\matr{U}_{k-1}^{\H}\nabla g(\matr{U}_{k-1})}$ is positive semidefinite, 
	then $\matr{U}_{k}=\matr{U}_{k-1}$ if and only if 
	$\ProjGrad{g}{\matr{U}_{k-1}}=0$.
\end{remark}

\begin{remark}
	\Cref{al-general-polar-general} can be seen as a special case of \emph{generalized power method} in \cite{journee2010generalized},
	which was developed for maximizing a convex
	function on a compact set. 
	However, in \cite{journee2010generalized}, only the \emph{stepsize convergence} was proved, \emph{i.e.}, $\|\matr{U}_{k+1}-\matr{U}_k\|\rightarrow 0$ under some conditions,
	while we will further establish its global convergence and linear convergence rate in \Cref{subsec-cost-increase-general}. 
\end{remark}

\subsection{Convergence analysis}\label{subsec-cost-increase-general}

For \Cref{al-general-polar-general}, we mainly prove the following results about its weak convergence, global convergence and convergence rate. 

\begin{theorem}\label{theorem-weak-conver-single}
	Suppose that the objective function \eqref{cost-function-genral-h} is convex. 
	If, in \Cref{al-general-polar-general} there exists $\delta>0$ such that 
	\begin{equation}\label{main-condition-convergence-symme-general}
	\sigma_{\rm min}(\nabla g(\matr{U}_{k-1}))>\delta
	\end{equation}
	always holds, then every accumulation point of the iterates $\matr{U}_{k}$ is a stationary point. 
\end{theorem}

\begin{theorem}\label{theorem-main-global-convergence-symme-general}
	Suppose that the objective function  \eqref{cost-function-genral-h} is convex and real analytic.
	If, in \Cref{al-general-polar-general} there exists $\delta>0$ such that condition \eqref{main-condition-convergence-symme-general}
	always holds,
	then the iterates $\matr{U}_{k}$ converge to a stationary point $\matr{U}_{*}$. 
	If $g$ is scale (or unitarily) invariant and $\matr{U}_{*}$ is a scale (or unitarily) semi-nondegenerate point,
	then the convergence rate is linear. 
\end{theorem}

Before proving \Cref{theorem-weak-conver-single} and \Cref{theorem-main-global-convergence-symme-general}, 
we need some lemmas,
which can be shown similarly 
as in \Cref{subsec-cost-increase-general-product}. 

\begin{lemma}\label{lemma-symmetric-increasing-general}
	Suppose that the objective function \eqref{cost-function-genral-h} is convex. 
	Then, the objective value of the iterates produced by \Cref{al-general-polar-general} monotonically increases and converges. 
\end{lemma}

\begin{lemma}\label{lemma-main-01-09}
	Suppose that the objective function \eqref{cost-function-genral-h} is convex. 
	If there exists $\delta>0$ such that condition 
	\eqref{main-condition-convergence-symme-general} always holds in \Cref{al-general-polar-general}, 
	then 
	\begin{equation*}
	g(\matr{U}_{k}) - g(\matr{U}_{k-1}) \geq \frac{\delta}{2} \|\matr{U}_{k}-\matr{U}_{k-1}\|^{2}.
	\end{equation*}
\end{lemma}

\begin{corollary}\label{corollary-main-01-04}
	Suppose that the objective function \eqref{cost-function-genral-h} is convex. 
	In the $k$-th iteration of \Cref{al-general-polar-general},
	if $\sigma_{\rm min}(\nabla g(\matr{U}_{k-1}))>0$,
	then
	$g(\matr{U}_{k}) = g(\matr{U}_{k-1})$ implies that $\matr{U}_{k} = \matr{U}_{k-1}$.
\end{corollary}

\paragraph{Proof of \Cref{theorem-weak-conver-single}}
By \Cref{lemma-main-01-09} and \Cref{lemma-inquality-02-09},
we have 
\begin{equation}\label{eq-inequality-import-g-general-2}
g(\matr{U}_{k}) - g(\matr{U}_{k-1}) \geq 
\frac{\delta}{2 (r+1)^2\Delta_{2}^2}
\|\ProjGrad{g}{\matr{U}_{k-1}}\|^2.
\end{equation}
Assume that $\matr{U}_{*}$ is an accumulation point and $\|\ProjGrad{g}{\matr{U}_{*}}\|\neq 0$.
Then there exists a subsequence $\{\matr{U}_{k_\ell}\}$ converging to $\matr{U}_{*}$.
Let $\ell\rightarrow\infty$. 
We see that, in \eqref{eq-inequality-import-g-general-2},
the left side converges to 0,
while the right side converges to $\frac{\delta}{2 (r+1)^2\Delta_{2}^2}\|\ProjGrad{g}{\matr{U}_{*}}\|^2> 0$.
This is impossible, and thus $\matr{U}_{*}$ is a stationary point. 
The proof is complete.  

\paragraph{Proof of \Cref{theorem-main-global-convergence-symme-general}}
By \Cref{lemma-main-01-09} and \Cref{lemma-inquality-02-09},
we have  
\begin{equation*}
g(\matr{U}_{k}) - g(\matr{U}_{k-1}) \geq 
\frac{\delta}{2(r+1)\Delta_{2}}
\|\matr{U}_{k}-\matr{U}_{k-1}\|\|\ProjGrad{g}{\matr{U}_{k-1}}\|.
\end{equation*}
It is clear that 
$g(\matr{U}_{k+1}) = g(\matr{U}_{k})$ implies that $\matr{U}_{k+1} = \matr{U}_{k}$ by \Cref{lemma-main-01-09}. 
Then, by \Cref{theorem-SU15} and \Cref{remark-condition-change}, it follows that the iterates converge to $\matr{U}_{*}$. 
Moreover, by \Cref{theorem-weak-conver-single}, $\matr{U}_{*}$ is a stationary point. 
If $g$ is scale (or unitarily) invariant and $\matr{U}_{*}$ is a scale (or unitarily) semi-nondegenerate point,
by \Cref{def-local-morse-bott} and \Cref{remMorseBott}(ii), 
we obtain that the function $g$ is Morse-Bott at $\matr{U}_{*}$
by letting 
$\set{C} \eqdef\left\{\matr{U}_{*}\matr{R}:
\matr{R} \in\set{DU}_{r}\ (\textrm{or}\ \UN{r})\right\}.$
Then the proof is complete by \Cref{cor:PLMorseBottManifold}, \Cref{lemma-inquality-02-09} and \Cref{theorem-SU15}.  

\subsection{Algorithm PDOI-S}
As in \Cref{sec:APDOI_S},
we now propose the following shifted version of \Cref{al-general-polar-general} to rid the dependence of convergence results on condition \eqref{main-condition-convergence-symme-general}.
This variant is called Algorithm PDOI-S. 

\begin{algorithm}
	\caption{PDOI-S}\label{al-general-polar-general-s}
	\begin{algorithmic}[1]
		\STATE{{\bf Input:} starting point $\matr{U}_{0}$, $\gamma>\Delta_{2}$.}
		\STATE{{\bf Output:} $\matr{U}_{k}$, $k\geq 1$.}
		\FOR{$k=1,2,\cdots,$}
		\STATE Compute $\nabla g(\matr{U}_{k-1})$;
		\STATE Compute the polar decomposition of $\nabla g(\matr{U}_{k-1})+\gamma\matr{U}_{k-1}$;
		\STATE Update $\matr{U}_{k}$ to be the orthogonal polar factor.
		\ENDFOR
	\end{algorithmic}
\end{algorithm}

We can prove the following lemma similarly as for \Cref{lem:sigma_delta}. 

\begin{lemma}
	Let $\delta=\gamma-\Delta_{2}>0$. 
	Then, in \Cref{al-general-polar-general-s}, we always have that 
	\begin{equation}\label{eq:Shifted_conditon_S}
	\sigma_{\rm min}(\nabla g(\matr{U}_{k-1})+\gamma\matr{U}_{k-1})>\delta.
	\end{equation}
\end{lemma}

For \Cref{al-general-polar-general-s}, based on \eqref{eq:Shifted_conditon_S}, we can prove the following convergence result without further assumption.
The proof is similar to that in \Cref{subsec-cost-increase-general}. 
We omit the details here.

\begin{theorem}\label{theorem-main-global-convergence-symme-general-s}
	Suppose that the objective function  \eqref{cost-function-genral-h} is convex and real analytic.
	Then the iterates $\matr{U}_{k}$ produced by \Cref{al-general-polar-general-s} converge to a stationary point $\matr{U}_{*}$.
	If $g$ is scale (or unitarily) invariant and $\matr{U}_{*}$ is a scale (or unitarily) semi-nondegenerate point,
	then the convergence rate is linear.  
\end{theorem}

\subsection{Special case: the unit sphere}\label{subsection:spe_unit_sphere}

We say that the objective function \eqref{cost-function-genral-h} is \emph{$\frac{1}{\beta}-$homo-} \emph{geneous}, if there exists some fixed $\beta\in (0,1)$ such that 
\begin{equation}\label{eq-general-assumption}
g(\matr{U}) = \beta\langle \matr{U}, \nabla g(\matr{U})\rangle_{\R}
\end{equation}
for any $\matr{U}\in\text{St}(r,n,\CC)$.
Suppose that the objective function  \eqref{definition-f} is restricted $\frac{1}{\alpha}-$homogene- ous and symmetric on $\Omega_{s}$. 
It follows from \eqref{eq-grad-rela-g-h} that the corresponding $g$ is $\frac{1}{\beta}-$homogeneous with $\beta=\frac{\alpha}{d}$.

If $r=1$, by similar derivations as in \Cref{sec:condition_conver},
we obtain that the condition \eqref{main-condition-convergence-symme-general} is always satisfied with $\delta=\frac{1}{2\beta}g(\matr{U}_{0})$ by \eqref{eq-general-assumption} and \Cref{lemma-symmetric-increasing-general}. 
The next result follows directly from \Cref{theorem-main-global-convergence-symme-general}.

\begin{corollary}\label{coro:r1_symmetric}
	Let $r=1$. 
	Suppose that the objective function \eqref{cost-function-genral-h} is $\frac{1}{\beta}-$homogeneous, convex and real analytic.
	Then the iterates $\matr{U}_{k}$ produced by \Cref{al-general-polar-general} converge to a stationary point $\matr{U}_{*}$. 
	If $g$ is scale (or unitarily) invariant and $\matr{U}_{*}$ is a scale (or unitarily) semi-nondegenerate point,
	then the convergence rate is linear.  
\end{corollary}

\section{Simultaneous approximate tensor diagonalization}\label{sec:objec_01}

\subsection{Euclidean gradient}\label{subsec:eucli_prob_01}

In this subsection, we mainly prove that the objective function \eqref{cost-fn-general-1} is restricted $\frac{1}{\alpha}-$homogeneous with $\alpha=\frac{1}{2}$ and block multiconvex.
Then all the algorithms and convergence results in \Cref{sec-opt-manifold-product} apply to this function. 
The following simple lemma follows directly from the gradient inequality \cite[Eq. (3.2)]{boyd2004convex}.

\begin{lemma}\label{remar:conv_inequality}
	If the objective function \eqref{definition-f} is restricted $\frac{1}{\alpha}-$homogeneous, then it is block multiconvex if and only if the restricted function $h_{(i)}$ in \eqref{definition-h} satisfies
	\begin{equation*}\label{eq-grad-convex-inequality}
	\langle\matr{U}', \nabla h_{(i)}(\matr{U})\rangle_{\R} \leq (1-\alpha)\langle\matr{U}, \nabla h_{(i)}(\matr{U})\rangle_{\R} + \alpha\langle\matr{U}', \nabla h_{(i)}(\matr{U}')\rangle_{\R},
	\end{equation*}
	for any $\matr{U}, \matr{U}'\in\CC^{n_i\times r_i}$. 
\end{lemma}

\begin{lemma}\label{lemma:h_U_T-sp}
	Let $\vect{v}\in\CC^{n}$.
	Define
	\begin{equation}\label{eq:form_h_01-sp-sp}
	h(\vect{u}) = |w|^2,
	\end{equation}
	where $\vect{u}\in\CC^{n}$ and $w=\vect{u}^{\dagger}\vect{v}$. 
	Then we have  
	\begin{align}\label{eq:delta_h-sp}
	\nabla h(\vect{u}) =
	\begin{cases}
	2w^{*}\vect{v}, & \text{ if } (\cdot)^{\dagger}=(\cdot)^{\H}; \\
	2w\vect{v}^{*},  & \text{ if } (\cdot)^{\dagger}=(\cdot)^{\T}.
	\end{cases}
	\end{align}
\end{lemma}

\begin{proof}
	We only prove the case $(\cdot)^{\dagger}=(\cdot)^{\H}$. 
	The other case is similar.
	We see that 
	\begin{align*}
	h(\vect{u}) = |\vect{u}^{\H}\vect{v}|^2
	=\left(\langle\vect{v}^{\R},\vect{u}^{\R}\rangle+\langle\vect{v}^{\I},\vect{u}^{\I}\rangle\right)^2
	+\left(\langle\vect{v}^{\I},\vect{u}^{\R}\rangle-\langle\vect{v}^{\R},\vect{u}^{\I}\rangle\right)^2.
	\end{align*}
	It follows that 
	$\nabla h(\vect{u}) =  2w^{*}\vect{v}.$ 
	The proof is complete.  
\end{proof}

The following lemma for matrix case is a direct consequence of \Cref{lemma:h_U_T-sp}. 

\begin{lemma}\label{lemma:h_U_T}
	Let $\tenss{V}\in\CC^{n\times r\times\cdots r}$.
	Define
	\begin{equation}\label{eq:form_h_01}
	h(\matr{U}) = \|\diag{\tenss{W}}\|^2,
	\end{equation}
	where $\matr{U}\in\CC^{n\times r}$ and $\tenss{W}=\tenss{V} \contr{1} \matr{U}^{\dagger}$. 
	Let $\vect{v}_{p}=\tenss{V}_{:,p\cdots p}\in\CC^{n}$ for $1\leq p\leq r$.
	Then 
	\begin{align}\label{eq:delta_h}
	\nabla h(\matr{U})^{\dagger} =  2\begin{bmatrix}
	\tenselem{W}_{1\cdots 1} & & \\
	& \ddots & \\
	& & \tenselem{W}_{r\cdots r}
	\end{bmatrix}[\vect{v}_{1},\cdots,\vect{v}_{r}]^{\H}.
	\end{align}
\end{lemma}


\begin{corollary} \label{coro:h_homege}
	Let $h$ be as in \Cref{lemma:h_U_T}.
	Then we have 
	\begin{align*}
	h(\matr{U}) = \frac{1}{2}\langle \matr{U}, \nabla h(\matr{U})\rangle_{\R}.
	\end{align*}
\end{corollary}

\begin{corollary}\label{lem:conv_h}
	Let $h$ be as in \Cref{lemma:h_U_T}.
	Then $h$ is convex. 
\end{corollary}

\begin{proof}
	We only prove the case $(\cdot)^{\dagger}=(\cdot)^{\H}$. 
	The other case  is similar. By \Cref{lemma:h_U_T} and Cauchy-Schwarz inequality, we see that 
	\begin{equation*}
	\langle\matr{U}', \nabla h(\matr{U})\rangle_{\R} \leq \frac{1}{2}\langle\matr{U}, \nabla h(\matr{U})\rangle_{\R} + \frac{1}{2}\langle\matr{U}', \nabla h(\matr{U}')\rangle_{\R}.
	\end{equation*}
	The proof is complete by \Cref{remar:conv_inequality}.  
\end{proof}

\begin{proposition}\label{prop:multicon_01}
	The objective function \eqref{cost-fn-general-1} is restricted $\frac{1}{\alpha}-$homogeneous with $\alpha=\frac{1}{2}$ and block multiconvex.
\end{proposition}
\begin{proof}
	For $\nu^{(i)}\in\Omega^{(i)}$, we denote that
	\begin{equation*}
	\tenss{V}^{(\ell)}=\tenss{A}^{(\ell)} \contr{1} (\matr{U}^{(1)})^{\dagger}\cdots \contr{i-1} (\matr{U}^{(i-1)})^{\dagger} \contr{i+1} (\matr{U}^{(i+1)})^{\dagger} \cdots \contr{d} (\matr{U}^{(d)})^{\dagger}.
	\end{equation*}
	Then we see that 
	\begin{align*}
	h_{(i)}(\matr{U})=  \sum\limits_{\ell=1}^{L}
	\alpha_{\ell}\|\diag{\tenss{V}^{(\ell)} \contr{i} \matr{U}^{\dagger}}\|^2 
	\end{align*}
	is a sum of forms in \eqref{eq:form_h_01}.  
	The proof is complete by \Cref{coro:h_homege} and \Cref{lem:conv_h}.  
\end{proof}

\subsection{Riemannian Hessian}\label{subsec:eucli_prob_01-H}

In this subsection, we mainly focus on the Riemannian Hessian of objective function \eqref{cost-fn-general-1}, and prove that there exist scale semi-nondegenerate points for this function.

\begin{lemma}\label{lemma:Eu_Hessi_g_1-2}
	Let the function $h$ be as in \eqref{eq:form_h_01-sp-sp}. 
	Then, for $\vect{z}\in\CC^{n}$, we have 
	\begin{align*}
	\nabla^{2}h(\vect{u})[\vect{z}] =
	\begin{cases}
	2\vect{v}\vect{v}^{\H}\vect{z}, & \text{ if } (\cdot)^{\dagger}=(\cdot)^{\H}; \\
	2\vect{v}^{*}\vect{v}^{\T}\vect{z}, & \text{ if } (\cdot)^{\dagger}=(\cdot)^{\T}.
	\end{cases}
	\end{align*}
\end{lemma}

\begin{proof}
	We only prove the case $(\cdot)^{\dagger}=(\cdot)^{\H}$. 
	The other case is similar.
	By \eqref{eq:delta_h-sp}, we have 
	\begin{align*}
	\nabla^2_{c} h(\vect{u})= \vect{v}\vect{v}^{\H}
	\ \ \textrm{and}\ \  
	\nabla^2_{r} h(\vect{u})=\matr{0}.
	\end{align*} 
	Then, by \eqref{eq:Hess_computation}, we obtain that 
	$\nabla^2 h(\vect{u})[\vect{z}] = 2\vect{v}\vect{v}^{\H}\vect{z}.$ The proof is complete. 
\end{proof}

The following result for matrix case is a direct consequence of  \Cref{lemma:Eu_Hessi_g_1-2}. 

\begin{lemma}\label{coro:h_U_T-H}
	Let $h$ and $\vect{v}_{q}$ be defined as in \Cref{lemma:h_U_T}. 
	Then, for $\matr{Z}\in\CC^{n\times r}$, we have 
	\begin{align}\label{eq:Hess_h_1}
	\left(\nabla^{2}h(\matr{U})[\matr{Z}]\right)^{\dagger}=
	2\begin{bmatrix}
	(\tenss{V} \contr{1} \matr{Z}^{\dagger})_{1\cdots 1} & & \\
	& \ddots & \\
	& & (\tenss{V}\contr{1} \matr{Z}^{\dagger})_{r\cdots r}
	\end{bmatrix}[\vect{v}_{1},\cdots,\vect{v}_{r}]^{\H}. 
	\end{align}
\end{lemma}


\begin{lemma}\label{lemm:h_hess_neq0}
	Let $\{\tenss{D}^{(\ell)}\}_{1\leq\ell\leq L}\subseteq\CC^{n\times r\times\cdots r}$ be a set of diagonal tensors. 
	Let $\alpha_\ell \in \RR^{+}$ for $1\leq\ell\leq L$.
	Define
	\begin{equation}\label{eq:form_h_01-h-ex}
	h(\matr{U}) = \sum_{\ell=1}^{L}\alpha_\ell\|\diag{\tenss{W}^{(\ell)}}\|^2,
	\end{equation}
	where $\matr{U}\in\CC^{n\times r}$ and $\tenss{W}^{(\ell)}=\tenss{D}^{(\ell)} \contr{1} \matr{U}^{\dagger}$. 
	Suppose that the diagonal elements $\tenselem{D}^{(\ell)}_{p\cdots p}\in\CC$ satisfy:
	\begin{align}
	\sum_{\ell=1}^{L}\alpha_\ell|\tenselem{D}^{(\ell)}_{p\cdots p}|^2\neq 0, \ \textrm{for any}\ 1\leq p\leq r.\label{cond_Hess_ex_02}
	\end{align}
	Then $\matr{I}_{n\times r}$ is a scale semi-nondegenerate point of $h$. 
\end{lemma}
\begin{proof}
	We only prove the case $(\cdot)^{\dagger}=(\cdot)^{\H}$. 
	The other case is similar.
	By \eqref{eq:delta_h}, we get that
	\begin{align*}
	\nabla h(\matr{I}_{n\times r}) =  2\sum_{\ell=1}^{L}\alpha_\ell\matr{I}_{n\times r}\begin{bmatrix}
	|\tenselem{D}^{(\ell)}_{1\cdots 1}|^2 & & \\
	& \ddots & \\
	& & |\tenselem{D}^{(\ell)}_{r\cdots r}|^2
	\end{bmatrix}.
	\end{align*}
	Then, by \eqref{eq:Rie_grad}, we have 
	$\ProjGrad{h}{\matr{I}_{n\times r}} =\matr{0}$, and thus 
	${\rm Proj}^{\bot}_{\matr{I}_{n\times r}} \nabla h(\matr{I}_{n\times r})=\nabla h(\matr{I}_{n\times r})$. 
	Since $h$ is scale invariant,  
	by \Cref{lem:f_invariant_Hessian}, we know that  $\rank{\Hess{h}{\matr{I}_{n\times r}}} \le 2nr-r^2-r$. 
	Note that
	\begin{align*}
	\HessAppl{h}{\matr{I}_{n\times r}}{\matr{Z}} 
	&= \nabla^2 h(\matr{I}_{n\times r}) [\matr{Z} ] -\matr{I}_{n\times r}\symmm{\matr{I}_{n\times r}^{\H}\nabla^2 h(\matr{I}_{n\times r}) [\matr{Z}]} \notag\\
	&\quad -\matr{Z}\matr{I}_{n\times r}^{\H}\nabla h(\matr{I}_{n\times r}) -
	\matr{I}_{n\times r} \symmm{\matr{Z}^{\H} \nabla h(\matr{I}_{n\times r})} \label{Hess_h_inter}
	\end{align*}
	by \eqref{eq:Hessian_reduction}, and
	\begin{align*}
	\nabla^2 h(\matr{I}_{n\times r}) [\matr{Z} ] = 
	2\sum_{\ell=1}^{L}\alpha_\ell\matr{I}_{n\times r}\begin{bmatrix}
	|\tenselem{D}^{(\ell)}_{1\cdots 1}|^2Z_{11}& & \\
	& \ddots & \\
	& & |\tenselem{D}^{(\ell)}_{r\cdots r}|^2Z_{rr}
	\end{bmatrix}
	\end{align*}
	by \eqref{eq:Hess_h_1}. 
	To prove that $\rank{\Hess{h}{\matr{I}_{n\times r}}} = 2nr-r^2-r$, we now show that $\Hess{h}{\matr{I}_{n\times r}}[\matr{Z}]\neq 0$ in the following cases.
	\begin{itemize}
		\item[$\bullet$] $\matr{Z}=\matr{I}_{n\times r}(\vect{e}_p\vect{e}_q^{\T}+\vect{e}_q\vect{e}_p^{\T})\ui,\ 1\leq p<q\leq r.$
		\begin{align*}
		\HessAppl{h}{\matr{I}_{n\times r}}{\matr{Z}} &= -\matr{I}_{n\times r}(\vect{e}_p\vect{e}_q^{\T}+\vect{e}_q\vect{e}_p^{\T})\ui \sum_{\ell=1}^{L}\alpha_\ell\left(|\tenselem{D}^{(\ell)}_{p\cdots p}|^2+|\tenselem{D}^{(\ell)}_{q\cdots q}|^2\right)\neq0. 
		\end{align*}
		\item[$\bullet$] $\matr{Z}=\matr{I}_{n\times r}(\vect{e}_p\vect{e}_q^{\T}-\vect{e}_q\vect{e}_p^{\T}),\ 1\leq p<q\leq r.$
		\begin{align*}
		\HessAppl{h}{\matr{I}_{n\times r}}{\matr{Z}} &= -\matr{I}_{n\times r}(\vect{e}_p\vect{e}_q^{\T}-\vect{e}_q\vect{e}_p^{\T}) \sum_{\ell=1}^{L}\alpha_\ell\left(|\tenselem{D}^{(\ell)}_{p\cdots p}|^2+|\tenselem{D}^{(\ell)}_{q\cdots q}|^2\right)\neq0. 
		\end{align*}
		\item[$\bullet$] $\matr{Z}=\matr{I}^{\perp}_{n\times r}\matr{X}, \ \matr{X}\in\CC^{(n-r)\times r}.$
		\begin{align*}
		\HessAppl{h}{\matr{I}_{n\times r}}{\matr{Z}} &= -2\matr{I}^{\perp}_{n\times r}\matr{X}\sum_{\ell=1}^{L}\alpha_\ell\begin{bmatrix}
		|\tenselem{D}^{(\ell)}_{1\cdots 1}|^2& & \\
		& \ddots & \\
		& & |\tenselem{D}^{(\ell)}_{r\cdots r}|^2
		\end{bmatrix}\neq0.
		\end{align*}
	\end{itemize}
	Note that the tangent space ${\rm\bf T}_{\matr{I}_{n\times r}} \text{St}(r,n,\CC)$ in \eqref{eq:tanget_spac} is spanned by the matrices $\matr{Z}$ in above cases and $\matr{Z}_p$ (as in \Cref{lem:f_invariant_Hessian}). 
	The proof is complete. 
\end{proof}

\begin{remark}
	Under the conditions of \Cref{lemm:h_hess_neq0}, 
	since the function $h$ in \eqref{eq:form_h_01-h-ex} is scale invariant and $\matr{I}_{n\times r}$ is a scale semi-nondegenerate point, we obtain that the matrix 
	$\matr{I}_{n\times r}\matr{R}$ with $\matr{R}\in\set{DU}_{r}$ is also a a scale semi-nondegenerate point of $h$ . 
\end{remark}

The following result can be proved by a similar method as in the proof of \Cref{lemm:h_hess_neq0}. 
We omit the details here. 

\begin{proposition}\label{prop:hess_semi_stric_01}
	Let $\{\tenss{D}^{(\ell)}\}_{1\leq\ell\leq L}\subseteq\CC^{n_1\times n_2\times\cdots n_d}$ be a set of diagonal tensors
	and 
	\begin{align*}
	\tenss{A}^{(\ell)} = \tenss{D}^{(\ell)} \contr{1} (\matr{Q}_{1}^{\H})^{\dagger}  \cdots \contr{d}(\matr{Q}_{d}^{\H})^{\dagger}
	\end{align*}
	for $1\leq\ell\leq n$, where $\matr{Q}_{i}\in\UN{n_i}$ for $1\leq i\leq d$. 
	Suppose that the diagonal tensors $\{\tenss{D}^{(\ell)}\}_{1\leq\ell\leq L}$ satisfy the condition \eqref{cond_Hess_ex_02}. 
	Then $\omega_0=(\matr{Q}_{1}\matr{I}_{n_1\times r},\cdots,\matr{Q}_{d}\matr{I}_{n_d\times r})$ is a scale semi-nondegenerate point of the objective function \eqref{cost-fn-general-1}.
\end{proposition}

\subsection{Equivalent form}\label{subsec:equiva_01}
In this subsection,
we give an equivalent characterization of the objective function \eqref{cost-fn-general-1}. 
Let $\tenss{B}\in\CC^{n_1\times n_2\times\cdots n_d\times n_1\times n_2\times\cdots n_d}$ be a $2d$-th order complex tensor.
We say that $\tenss{B}$ is \emph{Hermitian} \cite{NieY19:Hermitian,ULC2019} if 
\begin{equation*}
\tenselem{B}_{p_1\cdots p_d p_{d+1} \cdots p_{2d}} = \tenselem{B}^{*}_{p_{d+1} \cdots p_{2d} p_1\cdots p_d}
\end{equation*}
for any $1\leq p_i, p_{d+i}\leq n_i$ and $1\leq i\leq d$. 
We say that $\tenss{B}$ is \emph{positive semidefinite} if $\tenss{B}$ is Hermitian and the unfolding $\tau(\tenss{B})\in\CC^{R\times R}$ ($R=n_1n_2\cdots n_d$) is positive semidefinite. 
Now we present the following lemma, which is an easy extension of \cite[Lem. 4.11]{ULC2019} and \cite[Prop. 3.5]{jiang2016characterizing}.
It can be proved by the spectral theorem for Hermitian matrices applied to the unfolding $\tau(\tenss{B})\in\CC^{R\times R}$. 

\begin{lemma}\label{lemm:equiv_B_A}
	Let $\tenss{B}\in\CC^{n_1\times n_2\times\cdots n_d\times n_1\times n_2\times\cdots n_d}$ be a $2d$-th order positive semidefinite tensor.
	Then there exist a set of complex tensors $\{\tenss{A}^{(\ell)}\}_{1\leq\ell\leq L}\subseteq \CC^{n_1\times n_2\times\cdots n_d}$ and $\alpha_{\ell}\geq 0$ such that
	\begin{equation*}\label{eq:hermitian_spectral}
	\tenss{B} = \sum\limits_{\ell=1}^{L} \alpha_\ell (\tenss{A}^{(\ell)})^* \otimes \tenss{A}^{(\ell)},\ \text{i.e.,}\ \tenselem{B}_{p_1\ldots p_dp_{d+1}\ldots p_{2d}} = \sum\limits_{\ell=1}^{L} \alpha_\ell \left(\tenselem{A}^{(\ell)}_{p_1\ldots p_d}\right)^* \tenselem{A}^{(\ell)}_{p_{d+1}\ldots p_{2d}}
	\end{equation*}
	for any $1\leq p_i, p_{d+i}\leq n_i$ and $1\leq i\leq d$. 
\end{lemma}

\begin{proposition}\label{prop:equiv_form_B}
	The function $f$ admits representation \eqref{cost-fn-general-1} if and only if there exists a $2d$-th order positive semidefinite tensor $\tenss{B}$ 
	such that 
	\begin{equation}\label{cost-fn-general-trace}
	f(\omega) = \trace{\tenss{B} \contr{1} (\matr{U}^{(1)})^{\H}  \cdots \contr{d} (\matr{U}^{(d)})^{\H}\contr{d+1} (\matr{U}^{(1)})^{\T}\cdots \contr{2d} (\matr{U}^{(d)})^{\T}}.
	\end{equation}
\end{proposition}
\begin{proof}
	Note that the objective function \eqref{cost-fn-general-1} satisfies
	\begin{equation*}\label{cost-fn-general-12}
	f(\omega) = \sum\limits_{\ell=1}^{L} \sum_{q=1}^{r}\alpha_\ell|\tenss{A}^{(\ell)} \contr{1} (\vect{u}_{q}^{(1)})^{\dagger}\cdots \contr{d} (\vect{u}_{q}^{(d)})^{\dagger}|^2,
	\end{equation*}
	and the function \eqref{cost-fn-general-trace} satisfies 
	\begin{equation*}\label{cost-fn-general-trace-2}
	f(\omega) 
	=\sum_{q=1}^{r}(\tenss{B} \contr{1} (\vect{u}_{q}^{(1)})^{\H}  \cdots \contr{d} (\vect{u}_{q}^{(d)})^{\H}\contr{d+1} (\vect{u}_{q}^{(1)})^{\T}\cdots \contr{2d} (\vect{u}_{q}^{(d)})^{\T}).
	\end{equation*}
	Then the equivalence between \eqref{cost-fn-general-1} and \eqref{cost-fn-general-trace} follows directly from \Cref{lemm:equiv_B_A}. 
	The proof is complete.  
\end{proof}

\subsection{Example: low rank orthogonal approximation}
\label{sub-pro-sec-best-rank-ort-appro-1}

Let $\tenss{A}\in\CC^{n_1\times n_2\times\cdots\times n_d}$ and $r\leq n_i$ for $1\leq i\leq d$.
The \emph{low rank orthogonal approximation} \cite{chen2009tensor,hu2019LROAT,kolda2001orthogonal,yang2019epsilon} is to find 
\begin{equation}\label{eq-problem-low-orth-rank}
\tenss{B}_{*} =\arg\min \|\tenss{A} - \tenss{B}\|,
\end{equation}
where 
$\tenss{B} = \sum_{p=1}^{r}\mu_{p}\vect{u}_{p}^{(1)}\otimes\cdots\otimes\vect{u}_{p}^{(d)}$ 
satisfies that
$\langle\vect{u}_{p_1}^{(i)}, \vect{u}_{p_2}^{(i)}\rangle=0$ for any $1\leq i\leq d$ and $p_1\neq p_2$.
Let $\matr{U}^{(i)}=[\vect{u}_{1}^{(i)},\cdots,\vect{u}_{r}^{(i)}]\in\text{St}(r,n_i,\CC)$ for $1\leq i\leq d.$
It is clear that
\begin{equation*}
\tenss{B} = \tenss{D}\contr{1}\matr{U}^{(1)}\contr{2}\matr{U}^{(2)}\contr{3}\cdots\contr{d}\matr{U}^{(d)},
\end{equation*}
where $\tenss{D}\in\CC^{r\times r\times\cdots\times r}$ is a diagonal tensor satisfying that $\diag{\tenss{D}}=(\mu_1,\cdots,\mu_r)^{\T}$. 
The following lemma is a direct extension of \cite[Prop. 5.1]{chen2009tensor}.
We omit the detailed proof here.

\begin{lemma}\label{lemm:equiv_ortho_low}
	Problem \eqref{eq-problem-low-orth-rank} is equivalent to the maximization of 
	\begin{equation}\label{eq-cost-function-general-orth}
	f(\matr{U}^{(1)}, \matr{U}^{(2)}, \cdots, \matr{U}^{(d)}) = \|\diag{\tenss{W}}\|^2,
	\end{equation}
	where $\tenss{W}=\tenss{A}\contr{1}(\matr{U}^{(1)})^\H\contr{2}(\matr{U}^{(2)})^\H\cdots\contr{d}(\matr{U}^{(d)})^\H\in\CC^{r\times r\times\cdots\times r}$. 
\end{lemma}

We now make some notations to present \Cref{al-general-polar} for objective function \eqref{eq-cost-function-general-orth}. 
Let $\matr{U}_{k}^{(i)}=[\vect{u}_{1}^{(k,i)},\cdots,\vect{u}_{r}^{(k,i)}]\in\text{St}(r,n_i,\CC)$ for $k\geq 1$ and $1\leq i\leq d$. 
Denote 
\begin{align*}
\tenss{V}^{(k,i)}&\eqdef\tenss{A}\contr{1}(\matr{U}_{k}^{(1)})^\H\cdots\contr{i-1}(\matr{U}_{k}^{(i-1)})^\H\contr{i+1}(\matr{U}_{k-1}^{(i+1)})^\H\cdots\contr{d}(\matr{U}_{k-1}^{(d)})^\H,\\
\vect{v}^{(k,i,p)}
&\eqdef\tenss{A}\contr{1}(\vect{u}_{p}^{(k,1)})^\H\cdots\contr{i-1}(\vect{u}_{p}^{(k,i-1)})^\H\contr{i+1}(\vect{u}_{p}^{(k-1,i+1)})^\H\cdots\contr{d}(\vect{u}_{p}^{(k-1,d)})^\H\in\CC^{n_{i}},
\end{align*}
for $1\leq p\leq r$.
Let $\tenss{W}^{(k,i-1)}\eqdef\tenss{V}^{(k,i)}\contr{i}(\matr{U}_{k-1}^{(i)})^\H$. 
Note that 
\begin{equation}\label{eq-delta-low-orth}
\nabla h_{(k,i)}(\matr{U}_{k-1}^{(i)}) = 2
[\vect{v}^{(k,i,1)},\cdots,\vect{v}^{(k,i,r)}]
\begin{bmatrix}
\left(\tenselem{W}^{(k,i-1)}_{1\cdots 1}\right)^{*} & & \\
& \ddots & \\
& & \left(\tenselem{W}^{(k,i-1)}_{r\cdots r}\right)^{*}
\end{bmatrix}
\end{equation}
by equation \eqref{eq:delta_h}.
Then \Cref{al-general-polar} for objective function \eqref{eq-cost-function-general-orth} specializes to \Cref{al-general-polar-LROAT}, which is the \emph{low rank orthogonal approximation of tensors} (LROAT) algorithm
in \cite[Sec. 5.4]{chen2009tensor}. 
In real case, the decomposition form in \eqref{eq-gradient-polar-decom} for objective function \eqref{eq-cost-function-general-orth} is the same as \cite[Eq. (5.16)]{chen2009tensor}. 

\begin{algorithm}
	\caption{LROAT}\label{al-general-polar-LROAT}
	\begin{algorithmic}[1]
		\STATE{{\bf Input:} starting point $\omega^{(0)}$.}
		\STATE{{\bf Output:} $\omega^{(k,i)}$, $k\geq 1$, $1\leq i\leq d$.}
		\FOR{$k=1,2,\cdots,$}
		\FOR{$i=1,2,\cdots,d$}
		\STATE Compute $\nabla h_{(k,i)}(\matr{U}^{(i)}_{k-1})$ by \eqref{eq-delta-low-orth};
		\STATE Compute the polar decomposition of $\nabla h_{(k,i)}(\matr{U}^{(i)}_{k-1})$;
		\STATE Update $\matr{U}_{k}^{(i)}$ to be the orthogonal polar factor.
		\ENDFOR
		\ENDFOR
	\end{algorithmic}
\end{algorithm}

Note that the objective function \eqref{eq-cost-function-general-orth} is restricted $\frac{1}{\alpha}-$homogeneous with $\alpha=\frac{1}{2}$ and block multiconvex by \Cref{prop:multicon_01}.
The next result follows directly from \Cref{theorem-weak-conver-product}.

\begin{corollary}\label{theorem-main-global-convergence-LROAT-G}
	If, in \Cref{al-general-polar-LROAT} there exists $\delta>0$ such that condition \eqref{main-condition-convergence}
	always holds,
	then the iterates $\omega^{(k,i)}$ converge to a stationary point $\omega_{*}$. 
	If $\omega_{*}$ is a scale semi-nondegenerate point,
	then the convergence rate is linear. 
\end{corollary}

\begin{remark}
	In \cite[Thm.  5.7]{chen2009tensor},
	the weak convergence of iterates $\omega^{(k)}$ produced by the real case of \Cref{al-general-polar-LROAT} was shown under the condition that $\nabla h_{(k,i)}(\matr{U}_{k-1}^{(i)})$ is always of full rank.
	Very recently, we discovered that an \emph{epsilon-alternating least square} method was proposed to solve a problem more general than real case of \eqref{eq-problem-low-orth-rank}, and its global convergence was established in \cite{yang2019epsilon}. 
	Meanwhile, for an \emph{improved version} of real case of  \Cref{al-general-polar-LROAT}, the global convergence was established and its convergence rate was studied in \cite{hu2019LROAT}.
\end{remark}

\subsection{Example: best rank-1 approximation}\label{sec-best-rank-1-appro-11}

For $\tenss{A}\in\CC^{n_1\times n_2\times\cdots\times n_d}$, 
the \emph{best rank-1 approximation} \cite{de1997signal,DeLathauwerLieven1995,Lathauwer00:rank-1approximation,zhang2001rank} is to find 
\begin{equation}\label{eq-problem-low-multi-rank-1}
\tenss{B}_{*} =\arg\min \|\tenss{A} - \tenss{B}\|,
\end{equation}
where $\tenss{B} = \lambda\vect{u}^{(1)}\otimes\vect{u}^{(2)}\otimes\cdots\otimes\vect{u}^{(d)}$ with $\lambda\in\RR$ and $\vect{u}^{(i)}\in\mathbb{S}^{n_{i}-1}$. 
It was shown\footnote{This equivalence is in fact a special case of \Cref{lemm:equiv_ortho_low}.} in \cite{de1997signal,Lathauwer00:rank-1approximation} that 
problem \eqref{eq-problem-low-multi-rank-1} is equivalent to the maximization of 
\begin{equation}\label{eq-cost-function-general-rank-1-best}
f(\vect{u}^{(1)}, \vect{u}^{(2)}, \cdots, \vect{u}^{(d)}) = |\tenss{A}\contr{1}(\vect{u}^{(1)})^\H\contr{2}(\vect{u}^{(2)})^\H\cdots\contr{d}(\vect{u}^{(d)})^\H|^2,
\end{equation}
where 
$\vect{u}^{(i)}\in\mathbb{S}^{n_i-1}$ for $1\leq i\leq d$.

Let $k\geq 1$ and
$$\vect{v}^{(k,i)}\eqdef\tenss{A}\contr{1}(\vect{u}_{k}^{(1)})^{\H}\cdots\contr{i-1}(\vect{u}_{k}^{(i-1)})^{\H}\contr{i+1}(\vect{u}_{k-1}^{(i+1)})^{\H}\cdots\contr{d}(\vect{u}_{k-1}^{(d)})^{\H}.$$
Note that 
\begin{equation}\label{eq-engrad-09}
\nabla h_{(k,i)}(\vect{u}^{(i)}_{k-1}) = 2\left((\vect{u}^{(i)}_{k-1})^{\T}(\vect{v}^{(k,i)})^{*}\right)\vect{v}^{(k,i)}
\end{equation}
by \eqref{eq:delta_h-sp}.
Then \Cref{al-general-polar} for objective function \eqref{eq-cost-function-general-rank-1-best} specializes to \Cref{al-general-polar-rank-1}, which is the well-known \emph{higher order power method} (HOPM) \cite{DeLathauwerLieven1995,Lathauwer00:rank-1approximation} algorithm.
The equation in \cite[Eq. (3.9)]{Lathauwer00:rank-1approximation} obtained by Lagrange multipliers is in fact a special case of the decomposition form in \eqref{eq-gradient-polar-decom}.

\begin{algorithm}
	\caption{HOPM}\label{al-general-polar-rank-1}
	\begin{algorithmic}[1]
		\STATE{{\bf Input:} starting point $\omega^{(0)} = (\vect{u}_{0}^{(1)}, \vect{u}_{0}^{(2)},\cdots,\vect{u}_{0}^{(d)})$.}
		\STATE{{\bf Output:} $\omega^{(k,i)}=(\vect{u}_{k}^{(1)},\cdots,\vect{u}_{k}^{(i)},\vect{u}_{k-1}^{(i+1)},\cdots,\vect{u}_{k-1}^{(d)})$, $k\geq 1$, $1\leq i\leq d$.}
		\FOR{$k=1,2,\cdots,$}
		\FOR{$i=1,2,\cdots,d$}
		\STATE Compute $\vect{v}^{(k,i)}$;
		\STATE Update $\vect{u}_{k}^{(i)}=\frac{\vect{v}^{(k,i)}}{\|\vect{v}^{(k,i)}\|}$.
		\ENDFOR
		\ENDFOR
	\end{algorithmic}
\end{algorithm}

Note that the objective function \eqref{eq-cost-function-general-rank-1-best} is restricted $\frac{1}{\alpha}-$homogeneous with $\alpha=\frac{1}{2}$ and block multiconvex by \Cref{prop:multicon_01}.
The next result follows directly from \Cref{coro:conver_sphere}.

\begin{corollary}\label{theorem-main-global-convergence-HOOI-G-rank-1}
	The iterates $\omega^{(k,i)}$ produced by \Cref{al-general-polar-rank-1} converge to a stationary point $\omega_{*}$. 
	If $\omega_{*}$ is a scale semi-nondegenerate point,
	then the convergence rate is linear. 
\end{corollary}

\begin{remark} 
	It was shown in \cite[Thm.  11]{Usch15:pjo} and \cite[Thm.  4.4, Cor. 5.4]{hu2018convergence} that the iterates $\omega^{(k)}$ produced by the real case of  \Cref{al-general-polar-rank-1} converge to a stationary point, and the convergence rate is linear for almost all tensors in $\RR^{n_1\times n_2\times\cdots\times n_d}$.
\end{remark}

\section{Simultaneous approximate symmetric tensor diagonalization}\label{sec:objec_01-s}

\subsection{Euclidean gradient}\label{subsec:eucli_prob_01-s}
By \Cref{prop:multicon_01} and \Cref{remar:symme_grad},
we see that the objective function \eqref{cost-fn-general-1-s} is $\frac{1}{\beta}-$homogeneous with $\beta=\frac{1}{2d}$. 
Therefore, all the algorithms and convergence results in \Cref{sec-opt-manifold} apply to this function. 

Let $\tenss{V}^{(\ell)}=\tenss{S}^{(\ell)} \contr{2} \matr{U}^\dagger\cdots \contr{d} \matr{U}^\dagger$
and $\vect{v}^{(\ell)}_{p}=\tenss{V}^{(\ell)}_{:,p\cdots p}\in\CC^{n}$ for $1\leq p\leq r$.
Let $\tenss{W}^{(\ell)}=\tenss{V}^{(\ell)} \contr{1} \matr{U}^{\dagger}$. 
By equations \eqref{eq:delta_h} and \eqref{eq-grad-rela-g-h},
we have the Euclidean gradient of objective function \eqref{cost-fn-general-1-s} as follows: 
\begin{align}\label{eq:delta_h-s}
\nabla g(\matr{U})^{\dagger} =  2d\sum_{\ell=1}^{L}\alpha_\ell\begin{bmatrix}
\tenselem{W}^{(\ell)}_{1\cdots 1} & & \\
& \ddots & \\
& & \tenselem{W}^{(\ell)}_{r\cdots r}
\end{bmatrix}[\vect{v}^{(\ell)}_{1},\cdots,\vect{v}^{(\ell)}_{r}]^{\H}.
\end{align}

\begin{remark}
	In the real case, when $r=n$,
	the equation \eqref{eq:delta_h-s} was obtained in \cite[Sec. 4.1]{LUC2017globally}. 
	In the real case, when $L=1$, the equation \eqref{eq:delta_h-s} was obtained in \cite[Sec. 2.2]{LUC2019}. 
	In the complex case, when $r=n$,
	the equation \eqref{eq:delta_h-s} was obtained in \cite[Sec. 4.5]{ULC2019}. 
\end{remark}

\subsection{Riemannian Hessian}

In this subsection, we mainly focus on the Riemannian Hessian of objective function \eqref{cost-fn-general-1-s}, and prove that there exist scale semi-nondegenerate points for this function. 

\begin{lemma}\label{lemma:Eu_Hessi_g_1-0}
	Let $L=1$, $\alpha_{1}=1$ and $r=1$ in the objective function \eqref{cost-fn-general-1-s}.
	In this case, 
	for a $d$-th order symmetric tensor $\tenss{S}$,
	the objective function is 
	\begin{align}\label{cost-fn-general-1-s-1-0}
	g(\vect{u}) &=|w|^2,
	\end{align}
	where $\vect{u}\in\textbf{S}_{n-1}$ and $w=\tenss{S} \contr{1} \vect{u}^\dagger\cdots  \contr{d} \vect{u}^\dagger$. Let $\vect{v} = \tenss{S}\contr{2} \vect{u}^\dagger\cdots  \contr{d} \vect{u}^\dagger$
	and $\matr{C} = \tenss{S}\contr{3} \vect{u}^\dagger\cdots \contr{d} \vect{u}^\dagger$.
	Then, for $\vect{z}\in\CC^{n}$, we have 
	\begin{align*}
	\nabla^2 g(\vect{u})[\vect{z}] =
	\begin{cases}
	2d^2\vect{v}\vect{v}^{\H}\vect{z} + 2d(d-1)w^{*}\matr{C}\vect{z}^{*}, & \text{ if } (\cdot)^{\dagger}=(\cdot)^{\H}; \\
	2d^2\vect{v}^{*}\vect{v}^{\T}\vect{z} + 2d(d-1)w\matr{C}^{*}\vect{z}^{*}, & \text{ if } (\cdot)^{\dagger}=(\cdot)^{\T}.
	\end{cases}
	\end{align*}
\end{lemma}

\begin{proof}
	We only prove the case $(\cdot)^{\dagger}=(\cdot)^{\H}$. 
	The other case is similar.
	By \eqref{eq:delta_h-sp}, we have 
	\begin{align*}
	\nabla^2_{c} g(\vect{u})= d^2\vect{v}\vect{v}^{\H}
	\ \ \textrm{and}\ \  
	\nabla^2_{r} g(\vect{u})=d(d-1)w\matr{C}^{*}.
	\end{align*} 
	The proof is complete by \eqref{eq:Hess_computation}. 
\end{proof}

The following lemma for matrix case is a direct consequence of \Cref{lemma:Eu_Hessi_g_1-0}.

\begin{lemma}\label{lemma:h_U_T-H-2-2}
	Let $g$, $\tenss{V}^{(\ell)}$, $\vect{v}^{(\ell)}_{p}$ and $\tenss{W}^{(\ell)}$ be defined as in \eqref{cost-fn-general-1-s} and \eqref{eq:delta_h-s} for $1\leq p\leq r$ and $1\leq \ell\leq L$. 
	Let $\matr{U}=[\vect{u}_1,\cdots,\vect{u}_{r}]$ and 
	$\matr{C}^{(\ell)}_p = \tenss{S}^{(\ell)}\contr{3} \vect{u}_{p}^\dagger\cdots \contr{d} \vect{u}_{p}^\dagger$.
	Then, for $\matr{Z}=[\vect{z}_1,\cdots,\vect{z}_{r}]\in\CC^{n\times r}$, we have
	\begin{align}\label{eq:Hess_h_1-2-2}
	\left(\nabla^{2}g(\matr{U})[\matr{Z}]\right)^{\dagger}&=
	2d^2\sum_{\ell=1}^{L}\begin{bmatrix}
	(\tenss{V}^{(\ell)} \contr{1} \matr{Z}^{\dagger})_{1\cdots 1}& & \\
	& \ddots & \\
	& & (\tenss{V}^{(\ell)}\contr{1} \matr{Z}^{\dagger})_{r\cdots r}
	\end{bmatrix}\left[\vect{v}^{(\ell)}_{1},\cdots,\vect{v}^{(\ell)}_{r}\right]^{\H}\notag\\
	&\quad +2d(d-1)\sum_{\ell=1}^{L}\begin{bmatrix}
	\tenselem{W}^{(\ell)}_{1\cdots 1} & & \\
	& \ddots & \\
	& & \tenselem{W}^{(\ell)}_{r\cdots r} 
	\end{bmatrix}\left[\matr{C}^{(\ell)}_{1}(\vect{z}^{\H}_{1})^{\dagger},\cdots,\matr{C}^{(\ell)}_{r}(\vect{z}^{\H}_{r})^{\dagger}\right]^{\H}.
	\end{align}
\end{lemma}

\begin{lemma}\label{lem:hess_symme}
	Let $\{\tenss{D}^{(\ell)}\}_{1\leq\ell\leq L}\subseteq\CC^{n\times r\times\cdots r}$ be a set of diagonal tensors. 
	Let $\alpha_\ell \in \RR^{+}$ for $1\leq\ell\leq L$. 
	Define
	\begin{equation}\label{eq:form_h_01-h-ex-sy}
	g(\matr{U}) = \sum_{\ell=1}^{L}\alpha_\ell\|\diag{\tenss{W}^{(\ell)}}\|^2,
	\end{equation}
	where $\matr{U}\in\CC^{n\times r}$ and $\tenss{W}^{(\ell)}=\tenss{D}^{(\ell)} \contr{1} \matr{U}^{\dagger}\cdots\contr{d} \matr{U}^{\dagger}$. 
	Suppose that the diagonal elements $\tenselem{D}^{(\ell)}_{p\cdots p}\in\CC$ satisfy:
	\begin{align}
	\sum_{\ell=1}^{L}\alpha_\ell|\tenselem{D}^{(\ell)}_{p\cdots p}|^2\neq 0, \ \textrm{for}\ 1\leq p\leq r.\label{cond_Hess_ex_02-0-2}
	\end{align}
	Then $\matr{I}_{n\times r}$ is a scale semi-nondegenerate point of $g$. 
\end{lemma}
\begin{proof}
	We only prove the case $(\cdot)^{\dagger}=(\cdot)^{\H}$. 
	The other case is similar. 
	By \eqref{eq:delta_h-s}, we get that
	\begin{align}
	\nabla g(\matr{I}_{n\times r}) =  2d\sum_{\ell=1}^{L}\alpha_\ell\matr{I}_{n\times r}\begin{bmatrix}
	|\tenselem{D}^{(\ell)}_{1\cdots 1}|^2 & & \\
	& \ddots & \\
	& & |\tenselem{D}^{(\ell)}_{r\cdots r}|^2
	\end{bmatrix}.
	\end{align}
	Then, by \eqref{eq:Rie_grad}, we have $\ProjGrad{h}{\matr{I}_{n\times r}} =0$. 
	Note that
	\begin{align*}
	\nabla^2 g(\matr{I}_{n\times r}) [\matr{Z} ] = 
	2\left(d^2+d(d-1)\right)\sum_{\ell=1}^{L}\alpha_\ell\matr{I}_{n\times r}\begin{bmatrix}
	|\tenselem{D}^{(\ell)}_{1\cdots 1}|^2Z_{11}& & \\
	& \ddots & \\
	& & |\tenselem{D}^{(\ell)}_{r\cdots r}|^2Z_{rr}
	\end{bmatrix}
	\end{align*}
	by \eqref{eq:Hess_h_1-2-2}. 
	The proof can be complete by a similar method as in the proof of
	\Cref{lemm:h_hess_neq0}, and we omit the details here. 
\end{proof}

\begin{remark}
	Under the conditions of \Cref{lem:hess_symme}, 
	since the objective function  \eqref{cost-fn-general-1-s}  is scale invariant and $\matr{I}_{n\times r}$ is a scale semi-nondegenerate point, we obtain that the matrix 
	$\matr{I}_{n\times r}\matr{R}$ with $\matr{R}\in\set{DU}_{r}$ is also a scale semi-nondegenerate point of $g$. 
	This fact can be also proved by a similar method as in the proof of \Cref{lem:hess_symme}. 
\end{remark}

The following result can be proved by a similar method as in the proof of \Cref{lem:hess_symme}. 
We omit the details here. 

\begin{proposition}\label{prop:hess_semi_stric_02-0}
	Let $\{\tenss{D}^{(\ell)}\}_{1\leq\ell\leq L}\subseteq\CC^{n\times n\times\cdots n}$ be a set of diagonal tensors
	and 
	\begin{align*}
	\tenss{S}^{(\ell)} = \tenss{D}^{(\ell)} \contr{1} (\matr{Q}^{\H})^{\dagger}  \cdots \contr{d}(\matr{Q}^{\H})^{\dagger}
	\end{align*}
	for $1\leq\ell\leq n$, where $\matr{Q}\in\UN{n}$. 
	Suppose that the diagonal tensors $\{\tenss{D}^{(\ell)}\}_{1\leq\ell\leq L}$ satisfy condition \eqref{cond_Hess_ex_02-0-2}. 
	Then $\matr{Q}\matr{I}_{n\times r}$ is a scale semi-nondegenerate point of the objective function \eqref{cost-fn-general-1-s}.
\end{proposition}

\begin{remark}
	In \cite[Prop. 7.6]{ULC2019}, for a single unitary matrix case, the existence of scale semi-nondegenerate points for three objective functions in ICA was proved in a different way.
	More precisely, using the Jacobi rotation functions on 2-dimensional subgroups, the Riemannian Hessian was proved to be of maximal rank $n(n-1)$.
\end{remark}

\subsection{Equivalent form}
We give an equivalent characterization of the objective function \eqref{cost-fn-general-1-s}, which is a direct consequence of \Cref{prop:equiv_form_B} (or \cite[Lem. 4.11]{ULC2019}). 
We omit the detailed proof here.

\begin{proposition}\label{prop:equiv_form_B-s-1}
	The function $g$ admits representation \eqref{cost-fn-general-1-s} if and only if there exists a $2d$-th order positive semidefinite tensor $\tenss{B}$ 
	such that 
	\begin{equation}\label{cost-fn-general-trace-s-1}
	g(\matr{U}) = \trace{\tenss{B} \contr{1} \matr{U}^{\H}  \cdots \contr{d} \matr{U}^{\H}\contr{d+1} \matr{U}^{\T}\cdots \contr{2d} \matr{U}^{\T}}.
	\end{equation}
\end{proposition}

Now we show why the objective function \eqref{cost-fn-general-1-s}, or equivalently \eqref{cost-fn-general-trace-s-1}, may be convex in many cases. 

\begin{example}\label{lemma-convex-tenso-form}
	Suppose that the $2d$-th order tensor $\tenss{B}$ in \eqref{cost-fn-general-trace-s-1} is constructed as
	\begin{equation}\label{eq-convex-tensor-form-9}
	\tenss{B}=\tenss{D}\contr{1}{\matr{H}^\H}\cdots\contr{d}{\matr{H}^\H}\contr{d+1}{\matr{H}^\T}\cdots\contr{2d}{\matr{H}^\T},
	\end{equation}
	where $\matr{H}\in\CC^{m\times n}$ and 
	$\tenss{D}\in\RR^{m\times m\times\cdots\times m}$ is a $2d$-th order diagonal tensor satisfying that 
	$\diag{\tenss{D}}=(\mu_1,\cdots,\mu_m)^{\T}$ with $\mu_{j}\geq 0$. 
	Let $\matr{X}\in\CC^{n\times r}$ and $$\tenss{W}=\tenss{B}\contr{1}{\matr{X}^\H}\cdots\contr{d}{\matr{X}^\H}\contr{d+1}{\matr{X}^\T}\cdots\contr{2d}{\matr{X}^\T}.$$ 
	Then, for any $1\leq i\leq r$, the function
	\begin{equation*}
	\tau_{i}:\CC^{n\times r}\longrightarrow \RR^{+},\ \matr{X}\longmapsto \tenselem{W}_{ii\cdots i}
	\end{equation*}
	is convex. 
	Therefore, the function \eqref{cost-fn-general-trace-s-1} is convex. 
\end{example}

\subsection{Example: low rank symmetric orthogonal approximation} 
\label{sub-pro-sec-best-rank-ort-appro-2}

Let 
$\tenss{A}$ be a $d$-th order symmetric tensor in $\text{symm}(\CC^{n\times n\times\cdots\times n})$ and $r\leq n$.
The \emph{low rank symmetric orthogonal approximation} \cite{chen2009tensor,pan2018symmetric,LUC2019} is to find 
\begin{equation}\label{eq-problem-low-orth-rank-symme}
\tenss{B}_{*} =\arg\min \|\tenss{A} - \tenss{B}\|,
\end{equation}
where 
$\tenss{B} = \sum_{p=1}^{r}\mu_{p}\vect{u}_{p}\otimes\cdots\otimes\vect{u}_{p}$ 
satisfies that
$\langle\vect{u}_{p_1}, \vect{u}_{p_2}\rangle=0$ for $p_1\neq p_2$.
It was shown\footnote{This equivalence is in fact a symmetric variant of \Cref{lemm:equiv_ortho_low}.} in \cite{chen2009tensor} that problem \eqref{eq-problem-low-orth-rank-symme} is equivalent to the maximization of 
\begin{equation}\label{eq-cost-function-general-orth-symm}
g(\matr{U}) = \|\diag{\tenss{W}}\|^2,
\end{equation}
where $\tenss{W}=\tenss{A}\contr{1}{\matr{U}^\H}\contr{2}{\matr{U}^\H}\cdots\contr{d}{\matr{U}^\H}\in\CC^{r\times r\times\cdots\times r}$. 

We now make some notations to present \Cref{al-general-polar-general} for objective function \eqref{eq-cost-function-general-orth-symm}. 
Let $k\geq 1$. Denote $\matr{U}_{k}=[\vect{u}_{1}^{(k)},\cdots,\vect{u}_{r}^{(k)}]\in\text{St}(r,n,\CC)$ and 
\begin{align*}
\tenss{V}^{(k)}&\eqdef\tenss{A}\contr{1}\matr{U}_{k-1}^\H\cdots\contr{d-1}\matr{U}_{k-1}^\H,\\
\vect{v}^{(k,p)}&\eqdef
\tenss{A}\contr{1}(\vect{u}_{p}^{(k-1)})^\H\cdots\contr{d-1}(\vect{u}_{p}^{(k-1)})^\H\in\CC^{n},
\end{align*}
for $1\leq p\leq r$. 
Let $\tenss{W}^{(k-1)}\eqdef\tenss{V}^{(k)}\contr{i}\matr{U}_{k-1}^\H$. 
Note that 
\begin{equation}\label{eq-delta-low-orth-1}
\nabla g(\matr{U}_{k-1}) = 2d
[\vect{v}^{(k,1)},\cdots,\vect{v}^{(k,r)}]
\begin{bmatrix}
\left(\tenselem{W}^{(k-1)}_{1\cdots 1}\right)^{*} & & \\
& \ddots & \\
& & \left(\tenselem{W}^{(k-1)}_{r\cdots r}\right)^{*}
\end{bmatrix}
\end{equation}
by equation 
\eqref{eq:delta_h-s}. 
In this case, \Cref{al-general-polar-general} for objective function \eqref{eq-cost-function-general-orth-symm} specializes to \Cref{al-general-polar-general-SLROAT}, which is the \emph{symmetric variant} of LROAT (S-LROAT) in \cite[Sec. 5.6]{chen2009tensor}.

\begin{algorithm}
	\caption{S-LROAT}\label{al-general-polar-general-SLROAT}
	\begin{algorithmic}[1]
		\STATE{{\bf Input:} starting point $\matr{U}_{0}$.}
		\STATE{{\bf Output:} $\matr{U}_{k}$, $k\geq 1$.}
		\FOR{$k=1,2,\cdots,$}
		\STATE Compute $\nabla g(\matr{U}_{k-1})$ by \eqref{eq-delta-low-orth-1};
		\STATE Compute the polar decomposition of $\nabla g(\matr{U}_{k-1})$;
		\STATE Update $\matr{U}_{k}$ to be the orthogonal polar factor.
		\ENDFOR
	\end{algorithmic}
\end{algorithm}

The next result follows directly from \Cref{theorem-main-global-convergence-symme-general}.

\begin{corollary}
	Suppose that the objective function \eqref{eq-cost-function-general-orth-symm} is convex.
	If, in \Cref{al-general-polar-general-SLROAT} 
	there exists $\delta>0$ such that condition 
	\eqref{main-condition-convergence-symme-general}
	always holds,
	then the iterates $\matr{U}_{k}$ converge to a stationary point $\matr{U}_{*}$. 
	If $\matr{U}_{*}$ is a scale semi-nondegenerate point,
	then the convergence rate is linear. 
\end{corollary}

\begin{remark}
	The real case of problem \eqref{eq-cost-function-general-orth-symm} was studied in \cite{pan2018symmetric,LUC2019}.
	In \cite{pan2018symmetric}, based on a decomposition form \cite[Eq. (11)]{pan2018symmetric} which is similar to \cite[Eq. (5.16)]{chen2009tensor}, an iterative algorithm was proposed, and the convergence of its shifted variant was studied as well. 
	In \cite{LUC2019},
	the real case of problem \eqref{eq-cost-function-general-orth-symm} is transformed to an equivalent problem on the orthogonal group $\ON{n}$, and the \emph{Jacobi low rank orthogonal approximation} (JLROA) algorithm was proposed, which includes the well-known Jacobi CoM2 algorithm \cite{Como94:sp,Como10:book} in ICA as a special case.
\end{remark}

\subsection{Example: best symmetric rank-1 approximation} 
\label{sub-pro-sec-best-rank-1-appro-2}
Let 
$\tenss{A}$ be a $d$-th order symmetric tensor in $\text{symm}(\CC^{n\times n\times\cdots\times n})$. 
The \emph{best symmetric rank-1 approximation} \cite{Lathauwer00:rank-1approximation,kofidis2002best,kolda2011shifted,qi2009z,zhang2012best} is to find 
\begin{equation}\label{eq-problem-low-multi-rank-1-symme}
\tenss{B}_{*} =\arg\min \|\tenss{A} - \tenss{B}\|,
\end{equation}
where $\tenss{B} = \lambda\vect{u}\otimes\vect{u}\otimes\cdots\otimes\vect{u}$ with $\lambda\in\RR$ and $\vect{u}\in\mathbb{S}^{n-1}$. 
It was shown\footnote{This equivalence is in fact a symmetric variant of the equivalence for  \eqref{eq-cost-function-general-rank-1-best}.} in \cite{de1997signal,Lathauwer00:rank-1approximation} that problem \eqref{eq-problem-low-multi-rank-1-symme} is equivalent to the maximization of 
\begin{equation}\label{eq-cost-function-general-rank-1-best-symme}
g(\vect{u})\eqdef f(\vect{u}, \vect{u}, \cdots, \vect{u}) = |\tenss{A}\contr{1}{\vect{u}^\H}\contr{2}{\vect{u}^\H}\cdots\contr{d}{\vect{u}^\H}|^2,
\end{equation}
where $\vect{u}\in\mathbb{S}^{n-1}$.

Let $k\geq 1$ and 
$\vect{v}^{(k)}\eqdef\tenss{A}\contr{1}{{\vect{u}_{k-1}^{\H}}}\cdots\contr{d-1}{{\vect{u}_{k-1}^{\H}}}.$
Note that 
\begin{equation}\label{eq-grad-g-rank-1}
\nabla g(\vect{u}_{k-1}) = 2d\left(\vect{u}_{k-1}^{\T}(\vect{v}^{(k)})^{*}\right)
\vect{v}^{(k)}
\end{equation}
by \eqref{eq:delta_h-s}. 
Then \Cref{al-general-polar-general} for objective function \eqref{eq-cost-function-general-rank-1-best-symme} specializes to \Cref{al-general-polar-rank-1-symme},
which is the well-known \emph{symmetric higher order power method} (S-HOPM) \cite{Lathauwer00:rank-1approximation,kofidis2002best}.

\begin{algorithm}
	\caption{S-HOPM}\label{al-general-polar-rank-1-symme}
	\begin{algorithmic}[1]
		\STATE{{\bf Input:} starting point $\vect{u}_{0}$.}
		\STATE{{\bf Output:} $\vect{u}_{k}$, $k\geq 1$.}
		\FOR{$k=1,2,\cdots,$}
		\STATE Compute $\vect{v}^{(k)}$;
		\STATE Update $\vect{u}_{k}=\frac{\vect{v}^{(k)}}{\|\vect{v}^{(k)}\|}$. 
		\ENDFOR
	\end{algorithmic}
\end{algorithm}

Note that the objective function \eqref{eq-cost-function-general-rank-1-best-symme} is $\frac{1}{\beta}-$homogeneous with $\beta=\frac{1}{2d}$ by \Cref{remar:symme_grad} and \Cref{prop:multicon_01}. 
The next result follows directly from \Cref{coro:r1_symmetric}.

\begin{corollary}
	If the objective function \eqref{eq-cost-function-general-rank-1-best-symme} is convex, then the iterates $\vect{u}_{k}$ produced by \Cref{al-general-polar-rank-1-symme} converge to a stationary point $\vect{u}_{*}$. 
	If $\vect{u}_{*}$ is a scale semi-nondegenerate point,
	then the convergence rate is linear. 
\end{corollary}

\begin{remark}
	In real case, the weak convergence of \Cref{al-general-polar-rank-1-symme} was established in
	\cite[Thm.  4]{kofidis2002best}. 
\end{remark}

\begin{remark}\label{exam-convex-rank-1}
	In real case, the objective function \eqref{eq-cost-function-general-rank-1-best-symme} was shown to be possibly convex by the method of \emph{square matrix unfolding} in \cite[Eq. (4.2)]{kofidis2002best}.
\end{remark}

\section{Simultaneous approximate tensor compression}\label{sec:objec_02}

\subsection{Euclidean gradient}\label{subsec:eucli_prob_03}
In this subsection, we mainly prove that the objective function \eqref{cost-fn-general-2} is restricted $\frac{1}{\alpha}-$homogeneous with $\alpha=\frac{1}{2}$ and block multiconvex.
Therefore, all the algorithms and convergence results in \Cref{sec-opt-manifold-product} apply to this function. 

\begin{lemma}\label{lemma:h_U_T-2-sp}
	Let $\tenss{V}\in\CC^{n\times r_2\times\cdots r_d}$.
	Define
	\begin{equation}\label{eq:form_h_02-sp}
	h(\vect{u}) = \|\tenss{W}\|^2,
	\end{equation}
	where $\vect{u}\in\CC^{n}$ and $\tenss{W}=\tenss{V} \contr{1} \vect{u}^{\dagger}$. 
	Let $\tenss{V}_{(1)}=[\vect{v}_1,\cdots,\vect{v}_{R_1}]\in\CC^{n\times R_1}$ be the 1-mode unfolding of $\tenss{V}$ with $R_1=r_2\cdots r_d$  
	and $\matr{V}=\sum_{j=1}^{R_1}\vect{v}_j\vect{v}_j^{\H}$. 
	Then 
	\begin{equation}\label{eq-euclidean-gradient-multilinear-2-sp}
	\nabla h(\vect{u}) 
	= \begin{cases}
	2\matr{V}\vect{u}, & \text{ if } (\cdot)^{\dagger}=(\cdot)^{\H}; \\
	2\matr{V}^{*}\vect{u}, & \text{ if } (\cdot)^{\dagger}=(\cdot)^{\T}.
	\end{cases}
	\end{equation}
\end{lemma}

\begin{proof}
	We only prove the case $(\cdot)^{\dagger}=(\cdot)^{\H}$. 
	The other case  is similar.
	We see that 
	\begin{align*}
	h(\vect{u}) &= \|\vect{u}^{\H}\tenss{V}_{(1)}\|^2
	= \sum_{j=1}^{R_1}|\vect{u}^{\H}\vect{v}_j|^2\\
	&= \sum_{j=1}^{R_1}\left(\left(\langle\vect{u}^{\R},\vect{v}^{\R}_{j}\rangle+\langle\vect{u}^{\I},\vect{v}^{\I}_{j}\rangle\right)^2
	+\left(\langle\vect{u}^{\R},\vect{v}^{\I}_{j}\rangle-\langle\vect{u}^{\I},\vect{v}^{\R}_{j}\rangle\right)^2\right).
	\end{align*}
	It follows that 
	$\nabla h(\vect{u}) =  2\sum_{j=1}^{R}\vect{v}_j\vect{v}^{\H}_j\vect{u}
	=2\matr{V}\vect{u}.$  
	The proof is complete.  
\end{proof}

The following result for matrix case is a direct consequence of  \Cref{lemma:h_U_T-2-sp}. 

\begin{lemma}\label{lemma:h_U_T-2}
	Let $\tenss{V}$ and $\matr{V}$ be as in \Cref{lemma:h_U_T-2-sp}.
	Define
	\begin{equation}\label{eq:form_h_02}
	h(\matr{U}) = \|\tenss{W}\|^2,
	\end{equation}
	where $\matr{U}\in\CC^{n\times r_1}$ and $\tenss{W}=\tenss{V} \contr{1} \matr{U}^{\dagger}$. 
	Then 
	\begin{equation}\label{eq-euclidean-gradient-multilinear-2}
	\nabla h(\matr{U}) 
	= \begin{cases}
	2\matr{V}\matr{U}, & \text{ if } (\cdot)^{\dagger}=(\cdot)^{\H}; \\
	2\matr{V}^{*}\matr{U}, & \text{ if } (\cdot)^{\dagger}=(\cdot)^{\T}.
	\end{cases}
	\end{equation}
\end{lemma}


As in \Cref{subsec:eucli_prob_01}, we can similarly obtain the following results.

\begin{corollary} \label{coro:h_homege-2}
	Let $h$ be as in \Cref{lemma:h_U_T-2}.
	Then we have  
	\begin{align*}
	h(\matr{U}) = \frac{1}{2}\langle \matr{U}, \nabla h(\matr{U})\rangle_{\R}.
	\end{align*}
\end{corollary}

\begin{corollary}\label{lem:conv_h-2}
	Let $h$ be as in \Cref{lemma:h_U_T-2}.
	Then $h$ is convex. 
\end{corollary}


\begin{proposition}\label{prop:multicon_02}
	The objective function \eqref{cost-fn-general-2} is restricted $\frac{1}{\alpha}-$homogeneous with $\alpha=\frac{1}{2}$ and block multiconvex.
\end{proposition}

\begin{remark}
	By equation \eqref{eq-euclidean-gradient-multilinear-2}, we see that the objective function \eqref{cost-fn-general-2} satisfies that $\matr{U}^{\H}\nabla h_{(i)}(\matr{U})$ is always positive semidefinite for any $\matr{U}\in\text{St}(r_i,n_i,\CC)$ and $1\leq i\leq d$.
\end{remark}

\subsection{Riemannian Hessian}\label{subsec:eucli_prob_01-H-7}
In this subsection, we mainly focus on the Riemannian Hessian of objective function \eqref{cost-fn-general-2}, and prove that there exist unitarily semi-nondegenerate points for this function. 

\begin{lemma}\label{lemma:h_U_T-H-7-sp}
	Let $h$ and $\matr{V}$ be defined as in \Cref{lemma:h_U_T-2-sp}. 
	Then, for $\vect{z}\in\CC^{n}$, we have  
	\begin{equation}\label{eq-euclidean-gradient-multilinear-2-7-sp}
	\nabla^{2}h(\vect{u})[\vect{z}]
	= \begin{cases}
	2\matr{V}\vect{z}, & \text{ if } (\cdot)^{\dagger}=(\cdot)^{\H}; \\
	2\matr{V}^{*}\vect{z}, & \text{ if } (\cdot)^{\dagger}=(\cdot)^{\T}.
	\end{cases}
	\end{equation}
\end{lemma}

\begin{proof}
	We only prove the case $(\cdot)^{\dagger}=(\cdot)^{\H}$. 
	The other case is similar.
	By \eqref{eq-euclidean-gradient-multilinear-2-sp}, we have 
	\begin{align*}
	\nabla^2_{c} h(\vect{u})= \sum_{j=1}^{R_1}\vect{v}_j\vect{v}_j^{\H}
	\ \ \textrm{and}\ \  
	\nabla^2_{r} h(\vect{u})=\matr{0}.
	\end{align*} 
	Then, by \eqref{eq:Hess_computation}, we obtain that 
	$\nabla^2 h(\vect{u})[\vect{z}] = 2\matr{V}\vect{z}.$ The proof is complete. 
\end{proof}

The following result for matrix case is a direct consequence of  \Cref{lemma:h_U_T-H-7-sp}. 

\begin{lemma}\label{lemma:h_U_T-H-7}
	Let $h$ and $\matr{V}$ be defined as in \Cref{lemma:h_U_T-2}. 
	Then, for $\matr{Z}\in\CC^{n\times r}$, we have  
	\begin{equation}\label{eq-euclidean-gradient-multilinear-2-7}
	\nabla^{2}h(\matr{U})[\matr{Z}]
	= \begin{cases}
	2\matr{V}\matr{Z}, & \text{ if } (\cdot)^{\dagger}=(\cdot)^{\H}; \\
	2\matr{V}^{*}\matr{Z}, & \text{ if } (\cdot)^{\dagger}=(\cdot)^{\T}.
	\end{cases}
	\end{equation}
\end{lemma}

\begin{lemma}\label{lemma:h_U_T-H-7-s}
	Let $h$ and $\matr{V}$ be defined as in \Cref{lemma:h_U_T-2}. 
	Suppose that $\ProjGrad{h}{\matr{U}_{*}} =0.$ 
	Then, for $\matr{Z}\in{\rm\bf T}_{\matr{U}_{*}} \text{St}(r,n,\CC)$, we have  
	\begin{equation}\label{eq-euclidean-gradient-multilinear-2-7-s}
	\HessAppl{h}{\matr{U}_{*}}{\matr{Z}} 
	= \begin{cases}
	2\left(\matr{I}_n-\matr{U}_{*}\matr{U}_{*}^{\H}\right)\matr{V}\matr{Z}-2\left(\matr{Z}\matr{U}_{*}^\H+\matr{U}_{*}\matr{Z}^\H\right)\matr{V}\matr{U}_{*}, & \text{if } (\cdot)^{\dagger}=(\cdot)^{\H}; \\
	2\left(\matr{I}_n-\matr{U}_{*}\matr{U}_{*}^{\H}\right)\matr{V}^{*}\matr{Z}-2\left(\matr{Z}\matr{U}_{*}^\H+\matr{U}_{*}\matr{Z}^\H\right)\matr{V}^{*}\matr{U}_{*}, & \text{if } (\cdot)^{\dagger}=(\cdot)^{\T}.
	\end{cases}
	\end{equation}
\end{lemma}
\begin{proof}
	We only prove the case $(\cdot)^{\dagger}=(\cdot)^{\H}$. 
	The other case is similar.
	Since $\matr{U}_{*}$ is a stationary point, we have 
	${\rm Proj}^{\bot}_{\matr{U}_{*}} \nabla h(\matr{U}_{*})=\nabla h(\matr{U}_{*})$. 
	Then, by equation \eqref{eq:Hessian_reduction}, we obtain that 
	\begin{align}
	\HessAppl{h}{\matr{U}_{*}}{\matr{Z}} 
	&= \nabla^2 h(\matr{U}_{*}) [\matr{Z} ] -\matr{U}_{*}\symmm{\matr{U}_{*}^{\H}\nabla^2 h(\matr{U}_{*}) [\matr{Z}]}\notag\\
	&\quad -\matr{Z}\matr{U}_{*}^{\H}\nabla h(\matr{U}_{*}) -
	\matr{U}_{*} \symmm{\matr{Z}^{\H} \nabla h(\matr{U}_{*})}\notag\\
	&=2\left(\matr{I}_n-\matr{U}_{*}\matr{U}_{*}^{\H}\right)\matr{V}\matr{Z}-2\left(\matr{Z}\matr{U}_{*}^\H+\matr{U}_{*}\matr{Z}^\H\right)\matr{V}\matr{U}_{*}.\label{Hess_h_inter-0}\notag
	\end{align}
	The proof is complete.  
\end{proof}

\begin{lemma}\label{lemm:h_hess_neq0-7-0}
	Let $\{\tenss{D}^{(\ell)}\}_{1\leq\ell\leq L}\subseteq\CC^{n\times r_2\times\cdots r_d}$ be a set of diagonal tensors. 
	Let $\alpha_\ell \in \RR^{+}$ for $1\leq\ell\leq L$. 
	Define
	\begin{equation}\label{eq:form_h_01-h-ex-7}
	h(\matr{U}) = \sum_{\ell=1}^{L}\alpha_\ell\|\tenss{W}^{(\ell)}\|^2,
	\end{equation}
	where $\matr{U}\in\text{St}(r_1,n,\CC)$ and $\tenss{W}^{(\ell)}=\tenss{D}^{(\ell)} \contr{1} \matr{U}^{\dagger}$. 
	Let $r\eqdef\min(n,r_2,r_3,\cdots,r_d)$. 
	Suppose that $r_1\leq r$ and the diagonal elements $\tenselem{D}^{(\ell)}_{p\cdots p}\in\CC$ satisfy:
	\begin{align}
	&\bullet\ \sum_{\ell=1}^{L}\alpha_\ell\left(|\tenselem{D}^{(\ell)}_{r_1+s}|^2-|\tenselem{D}^{(\ell)}_{t\cdots t}|^2\right)\neq 0, \ \textrm{for any}\ 1\leq s\leq r-r_1,\ 1\leq t\leq r_1;\label{cond_Hess_ex_02-7-1}\\
	&\bullet\ \sum_{\ell=1}^{L}\alpha_\ell|\tenselem{D}^{(\ell)}_{p\cdots p}|^2\neq 0, \ \textrm{for any}\ 1\leq p\leq r_1\ \textrm{(only needed when $r<n$)}.\label{cond_Hess_ex_02-7}
	\end{align}
	Then $\matr{I}_{n\times r_1}$ is a unitarily semi-nondegenerate point of $h$. 
\end{lemma}
\begin{proof}
	We only prove the case $(\cdot)^{\dagger}=(\cdot)^{\H}$. 
	The other case is similar. 
	By \eqref{eq-euclidean-gradient-multilinear-2}, we get that
	\begin{align*}
	\nabla h(\matr{I}_{n\times r_1}) =  2\sum_{\ell=1}^{L}\alpha_\ell\matr{I}_{n\times r_1}\begin{bmatrix}
	|\tenselem{D}^{(\ell)}_{1\cdots 1}|^2 & & \\
	& \ddots & \\
	& & |\tenselem{D}^{(\ell)}_{r_1\cdots r_1}|^2
	\end{bmatrix}.
	\end{align*}
	Then, by \eqref{eq:Rie_grad}, we have 
	$\ProjGrad{h}{\matr{I}_{n\times r_1}} =\matr{0}$. 
	Since $h$ is unitarily invariant,  
	by \Cref{lem:f_invariant_Hessian-2}, we know that  $\rank{\Hess{h}{\matr{I}_{n\times r_1}}} \le 2r_1(n-r_1)$. 
	Note that
	$$\nabla^2 h(\matr{I}_{n\times r_1}) [\matr{Z} ] = 
	2\sum_{\ell=1}^{L}\alpha_\ell\matr{V}^{(\ell)}\matr{Z}$$ 
	by \eqref{eq-euclidean-gradient-multilinear-2-7}. 
	To prove $\rank{\Hess{h}{\matr{I}_{n\times r_1}}} = 2r_1(n-r_1)$, we now show that $\Hess{h}{\matr{I}_{n\times r_1}}[\matr{Z}]\neq 0$ by \eqref{eq-euclidean-gradient-multilinear-2-7-s} in the following cases.
	\begin{itemize}
		\item[$\bullet$] $\matr{Z}=\matr{I}^{\perp}_{n\times r_1}\matr{B}, \ \matr{B}=z\vect{e}_s\vect{e}_t^{\T}\in\CC^{(n-r_1)\times r_1},\ z\in\CC,\ z\neq 0,\ 1\leq s\leq r-r_1,\ 1\leq t\leq r_1.$
		\begin{align*}
		\HessAppl{h}{\matr{I}_{n\times r_1}}{\matr{Z}} &=
		2\matr{I}^{\perp}_{n\times r_1}\vect{e}_{s}\vect{e}_t^{\T}z\sum_{\ell=1}^{L}\alpha_\ell\left(|\tenselem{D}^{(\ell)}_{r_1+s}|^2-|\tenselem{D}^{(\ell)}_{t\cdots t}|^2\right)\neq 0;
		\end{align*}
		\item[$\bullet$] $\matr{Z}=\matr{I}^{\perp}_{n\times r_1}\matr{B}, \ \matr{B}=z\vect{e}_s\vect{e}_t^{\T}\in\CC^{(n-r_1)\times r_1},\ z\in\CC,\ z\neq 0,\ r-r_1< s\leq n-r_1,\ 1\leq t\leq r_1.$
		\begin{align*}
		\HessAppl{h}{\matr{I}_{n\times r_1}}{\matr{Z}} &=
		-2\matr{I}^{\perp}_{n\times r_1}\vect{e}_{s}\vect{e}_t^{\T}z\sum_{\ell=1}^{L}\alpha_\ell|\tenselem{D}^{(\ell)}_{t\cdots t}|^2\neq 0.
		\end{align*}
	\end{itemize}
	Note that the tangent space ${\rm\bf T}_{\matr{I}_{n\times r}} \text{St}(r,n,\CC)$ in \eqref{eq:tanget_spac} is spanned by the matrices $\matr{Z}$ in above cases and $\matr{Z}_p$ (as in \Cref{lem:f_invariant_Hessian}) and $\matr{Z}_{p,q}, \matr{Z}_{p,q}^{'}$ (as in \Cref{lem:f_invariant_Hessian-2}). 
	The proof is complete. 
\end{proof}

\begin{remark}
	Under the conditions of \Cref{lemm:h_hess_neq0-7-0}, 
	since the function $h$ in \eqref{eq:form_h_01-h-ex-7} is unitarily invariant and $\matr{I}_{n\times r}$ is a unitarily semi-nondegenerate point, we obtain that the matrix 
	$\matr{I}_{n\times r}\matr{R}$ with $\matr{R}\in\UN{r}$ is also a a unitarily semi-nondegenerate point of $h$. 
	This fact can be also proved by a similar method as in the proof of \Cref{lemm:h_hess_neq0-7-0}. 
\end{remark}

The following result can be proved by a similar method as in the proof of \Cref{lemm:h_hess_neq0-7-0}. 
We omit the details here. 

\begin{proposition}\label{prop:countexam_02}
	Let $\{\tenss{D}^{(\ell)}\}_{1\leq\ell\leq L}\subseteq\CC^{n_1\times n_2\times\cdots n_d}$ be a set of diagonal tensors
	and 
	\begin{align*}
	\tenss{A}^{(\ell)} = \tenss{D}^{(\ell)} \contr{1} (\matr{Q}_{1}^\H)^{\dagger}  \cdots \contr{d}(\matr{Q}_{d}^\H)^{\dagger}
	\end{align*}
	for $1\leq\ell\leq L$, where $\matr{Q}_{i}\in\ON{n_i}$ for $1\leq i\leq d$. 
	Suppose that $r=r_1=\cdots = r_d$ and the diagonal tensors $\{\tenss{D}^{(\ell)}\}_{1\leq\ell\leq L}$ satisfy the conditions \eqref{cond_Hess_ex_02-7-1} and \eqref{cond_Hess_ex_02-7}. 
	Then $\omega_0=(\matr{Q}_{1}\matr{I}_{n_1\times r},\cdots,\matr{Q}_{d}\matr{I}_{n_d\times r})$ is a unitarily semi-nondegenerate point of the objective function \eqref{cost-fn-general-2}.
\end{proposition}


\subsection{Equivalent form}
By a similar method as in \Cref{subsec:equiva_01}, we could directly get the following equivalent characterization of objective function \eqref{cost-fn-general-2}. 

\begin{proposition}\label{prop:equiv_form_B-2}
	The function $f$ admits representation \eqref{cost-fn-general-2} if and only if there exists a $2d$-th order positive semidefinite tensor $\tenss{B}$ 
	such that 
	\begin{equation}\label{cost-fn-general-trace-2-1}
	f(\omega) = \sum_{1\leq p_1\leq r_1}\cdots\sum_{1\leq p_d\leq r_d}\tenselem{T}_{p_1\cdots p_dp_1\cdots p_d},
	\end{equation}
	where $\tenss{T}=\tenss{B} \contr{1} (\matr{U}^{(1)})^{\H}  \cdots \contr{d} (\matr{U}^{(d)})^{\H}\contr{d+1} (\matr{U}^{(1)})^{\T}\cdots \contr{2d} (\matr{U}^{(d)})^{\T}$.
\end{proposition}

\subsection{Example: low multilinear rank approximation}
\label{sub-pro-sec-best-rank-mul-appro} 

Let $\tenss{A}\in\CC^{n_1\times n_2\times\cdots\times n_d}$. 
We denote by $\text{rank}_{i}(\tenss{A})$ the \emph{i-rank}, that is, the rank of $i$-mode unfolding $\tenss{A}_{(i)}$. 
The vector $(\text{rank}_{1}(\tenss{A}),\cdots,\text{rank}_{d}(\tenss{A}))$ is called the \emph{multilinear rank} of $\tenss{A}$. 
Let $r_i\leq n_i$ for $1\leq i\leq d$. 
The \emph{low multilinear rank approximation} \cite{de1997signal,Lathauwer00:rank-1approximation,ishteva2009numerical,kruskal1989rank} is to find 
\begin{equation}\label{eq-problem-low-multi-rank}
\tenss{B}_{*} =\arg\min \|\tenss{A} - \tenss{B}\|,
\end{equation}
where $\text{rank}_{i}(\tenss{B})\leq r_i$ for $1\leq i\leq d$.
Note that any such $\tenss{B}$ can be decomposed as 
\begin{equation*}
\tenss{B} = \tenss{C}\contr{1}\matr{U}^{(1)}\contr{2}\matr{U}^{(2)}\contr{3}\cdots\contr{d}\matr{U}^{(d)},
\end{equation*}
with $\tenss{C}\in\CC^{r_1\times r_2\times\cdots \times r_{d}}$ and 
$\matr{U}^{(i)}\in\text{St}(r_i,n_i,\CC)$ for $1\leq i\leq d$.
Problem \eqref{eq-problem-low-multi-rank} is in fact equivalent to the following \emph{Tucker decomposition} \cite{de1997signal,Lathauwer00:rank-1approximation,tucker1966some}: 
\begin{equation}\label{eq-problem-low-multi-rank-equi}
\min\|\tenss{A} - \tenss{C}\contr{1}\matr{U}^{(1)}\contr{2}\matr{U}^{(2)}\contr{3}\cdots\contr{d}\matr{U}^{(d)}\|.
\end{equation}
The next lemma is a direct extension of \cite[Thm.  4.2]{Lathauwer00:rank-1approximation}. 
Its proof needs a complex analogue of \cite[Thm.  4.1]{Lathauwer00:rank-1approximation}, which can be obtained by the method in \cite[Prop. 5.1]{chen2009tensor}. 
We omit the detailed proof here.

\begin{lemma}
	Problem \eqref{eq-problem-low-multi-rank-equi} is equivalent to the maximization of 
	\begin{equation}\label{eq-cost-function-general}
	f(\matr{U}^{(1)}, \matr{U}^{(2)}, \cdots, \matr{U}^{(d)}) = \|\tenss{W}\|^2,
	\end{equation}
	where $\tenss{W}=\tenss{A}\contr{1}(\matr{U}^{(1)})^\H\contr{2}(\matr{U}^{(2)})^\H\cdots\contr{d}(\matr{U}^{(d)})^\H$
	and
	$\matr{U}^{(i)}\in\text{St}(r_i,n_i,\CC)$ for $1\leq i\leq d$.
\end{lemma}

We now make some notations to present \Cref{al-general-polar} for objective function \eqref{eq-cost-function-general}.  
Let $k\geq 1$ and $1\leq i\leq d$. Denote 
\begin{equation*}\label{eq-V-k-i}
\tenss{V}^{(k,i)}\eqdef\tenss{A}\contr{1}(\matr{U}_{k}^{(1)})^\H\cdots\contr{i-1}(\matr{U}_{k}^{(i-1)})^\H\contr{i+1}(\matr{U}_{k-1}^{(i+1)})^\H\cdots\contr{d}(\matr{U}_{k-1}^{(d)})^\H.
\end{equation*}
Let $\tenss{V}_{(i)}^{(k,i)}=[\vect{v}^{(k,i)}_1,\vect{v}^{(k,i)}_2,\cdots,\vect{v}^{(k,i)}_{R_i}]\in\CC^{n_{i}\times R_i}$ and $\matr{V}^{(k,i)}\eqdef\sum_{j=1}^{R_i}\vect{v}^{(k,i)}_j(\vect{v}^{(k,i)}_j)^{\H}$.
Then 
\begin{equation}\label{eq-cost-function-general-1-2-31-01}
\nabla h_{(k,i)}(\matr{U}_{k-1}^{(i)}) = 
2\matr{V}^{(k,i)}\matr{U}_{k-1}^{(i)}
\end{equation}
by equation \eqref{eq-euclidean-gradient-multilinear-2}. 
We call \Cref{al-general-polar} for objective function \eqref{eq-cost-function-general} to be the \emph{low multilinear rank approximation based on polar decomposition} (LMPD),
which is presented as \Cref{al-low-multilinear-hooiPD}. 
For objective function \eqref{eq-cost-function-general}, 
we can also derive the decomposition form in \eqref{eq-gradient-polar-decom} similarly as in \cite[Eq. (5.16)]{chen2009tensor}. 
We omit the details here. 

\begin{algorithm}
	\caption{LMPD}\label{al-low-multilinear-hooiPD}
	\begin{algorithmic}[1]
		\STATE{{\bf Input:} starting point $\omega^{(0)}$, multilinear rank $(r_1,\cdots,r_d)$.}
		\STATE{{\bf Output:} $\omega^{(k,i)}$, $k\geq 1$, $1\leq i\leq d$.}
		\FOR{$k=1,2,\cdots,$}
		\FOR{$i=1,2,\cdots,d$}
		\STATE Compute $\nabla h_{(k,i)}(\matr{U}^{(i)}_{k-1})$ by \eqref{eq-cost-function-general-1-2-31-01};
		\STATE Compute the polar decomposition of $\nabla h_{(k,i)}(\matr{U}^{(i)}_{k-1})$;
		\STATE Update $\matr{U}_{k}^{(i)}$ to be the orthogonal polar factor.
		\ENDFOR
		\ENDFOR
	\end{algorithmic}
\end{algorithm}

Note that the objective function \eqref{eq-cost-function-general} is restricted $\frac{1}{\alpha}-$homogeneous with $\alpha=\frac{1}{2}$ and block multiconvex by \Cref{prop:multicon_02}.
The next result follows directly from \Cref{theor-isolate-limit}.

\begin{corollary}\label{theorem-main-global-convergence-HOOI-G}
	If, in \Cref{al-low-multilinear-hooiPD} there exists $\delta>0$ such that condition \eqref{main-condition-convergence}
	always holds,
	then the iterates $\omega^{(k,i)}$ converge to a stationary point $\omega_{*}$. 
	If $\omega_{*}$ is a unitarily semi-nondegenerate point,
	then the convergence rate is linear. 
\end{corollary}

\subsection{Algorithm LMPD-S}\label{subsec:lmpd-s}

\Cref{al-general-polar-S} for objective function \eqref{eq-cost-function-general} is specialized as \Cref{al-low-multilinear-polar-shift-gradient}.
Because of the explicit form 
\eqref{eq-cost-function-general-1-2-31-01} and $\matr{V}^{(k,i)}$ is positive semidefinite,
we have $\sigma_{\rm min}(\nabla h_{(k,i)}(\matr{U}^{(i)}_{k-1})+\gamma\matr{U}^{(i)}_{k-1})\geq\gamma$. 
Therefore, we only need to set $\gamma>0$ in \Cref{al-low-multilinear-polar-shift-gradient},
which is better than the general case in \Cref{al-general-polar-S}.

\begin{algorithm}
	\caption{LMPD-S}\label{al-low-multilinear-polar-shift-gradient}
	\begin{algorithmic}[1]
		\STATE{{\bf Input:} starting point $\omega^{(0)}$, multilinear rank $(r_1,\cdots,r_d)$, $\gamma>0$.}
		\STATE{{\bf Output:} $\omega^{(k,i)}$, $k\geq 1$, $1\leq i\leq d$.}
		\FOR{$k=1,2,\cdots,$}
		\FOR{$i=1,2,\cdots,d$}
		\STATE Compute $\nabla h_{(k,i)}(\matr{U}^{(i)}_{k-1})$ by \eqref{eq-cost-function-general-1-2-31-01};
		\STATE Compute the polar decomposition of $\nabla h_{(k,i)}(\matr{U}^{(i)}_{k-1})+\gamma\matr{U}^{(i)}_{k-1}$;
		\STATE Update $\matr{U}_{k}^{(i)}$ to be the orthogonal polar factor.
		\ENDFOR
		\ENDFOR
	\end{algorithmic}
\end{algorithm}

The next result follows directly from \Cref{theor-isolate-limit-s}.

\begin{theorem}\label{theorem-main-global-convergence-HOOI-G-s}
	The iterates $\omega^{(k,i)}$ produced by \Cref{al-low-multilinear-polar-shift-gradient} converge to a stationary point $\omega_{*}$. 
	If $\omega_{*}$ is a unitarily semi-nondegenerate point,
	then the convergence rate is linear. 
\end{theorem}

\subsection{Experiments}\label{sec-experiment-hooi}
To solve problem \eqref{eq-problem-low-multi-rank},
various algorithms \cite{ishteva2009numerical} have been developed,
including the \emph{higher order orthogonal iteration} (HOOI) algorithm \cite{Lathauwer00:rank-1approximation,xu2018convergence}, the Riemannian trust-region algorithm \cite{ishteva2011best} and the Jacobi-type algorithm \cite{IshtAV13:simax}. 
In this subsection,
some numerical experiments are conducted to compare the performances of HOOI, LMPD and LMPD-S.

\begin{example}\label{example-2}
	We randomly generate one tensor in $\CC^{5\times 5\times 5}$, and run HOOI, LMPD and LMPD-S ($\gamma=0.01$).
	The results of objective function value \eqref{eq-cost-function-general} are shown in \Cref{figure-example-2}.
	It can be seen that LMPD and LMPD-S have comparable speed of convergence as compared with HOOI. 
\end{example}

\begin{figure}[tbhp]
	\centering
	\subfloat[($r_1,r_2,r_3$) = (1,1,2)]{\includegraphics[width=0.5\textwidth]{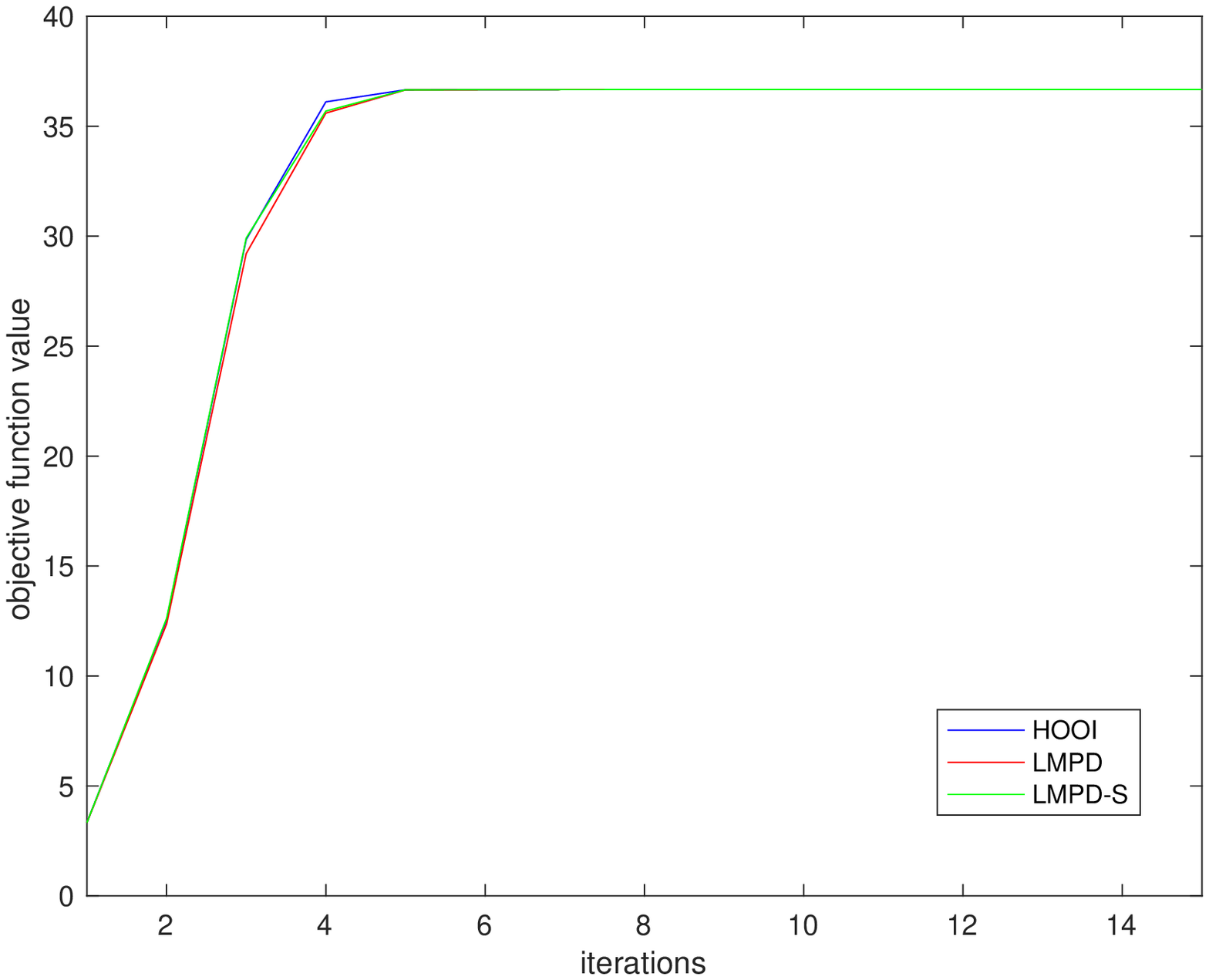}}\!\!\!
	\subfloat[$(r_1,r_2,r_3$) = (3,3,3)]{\includegraphics[width=0.5\textwidth]{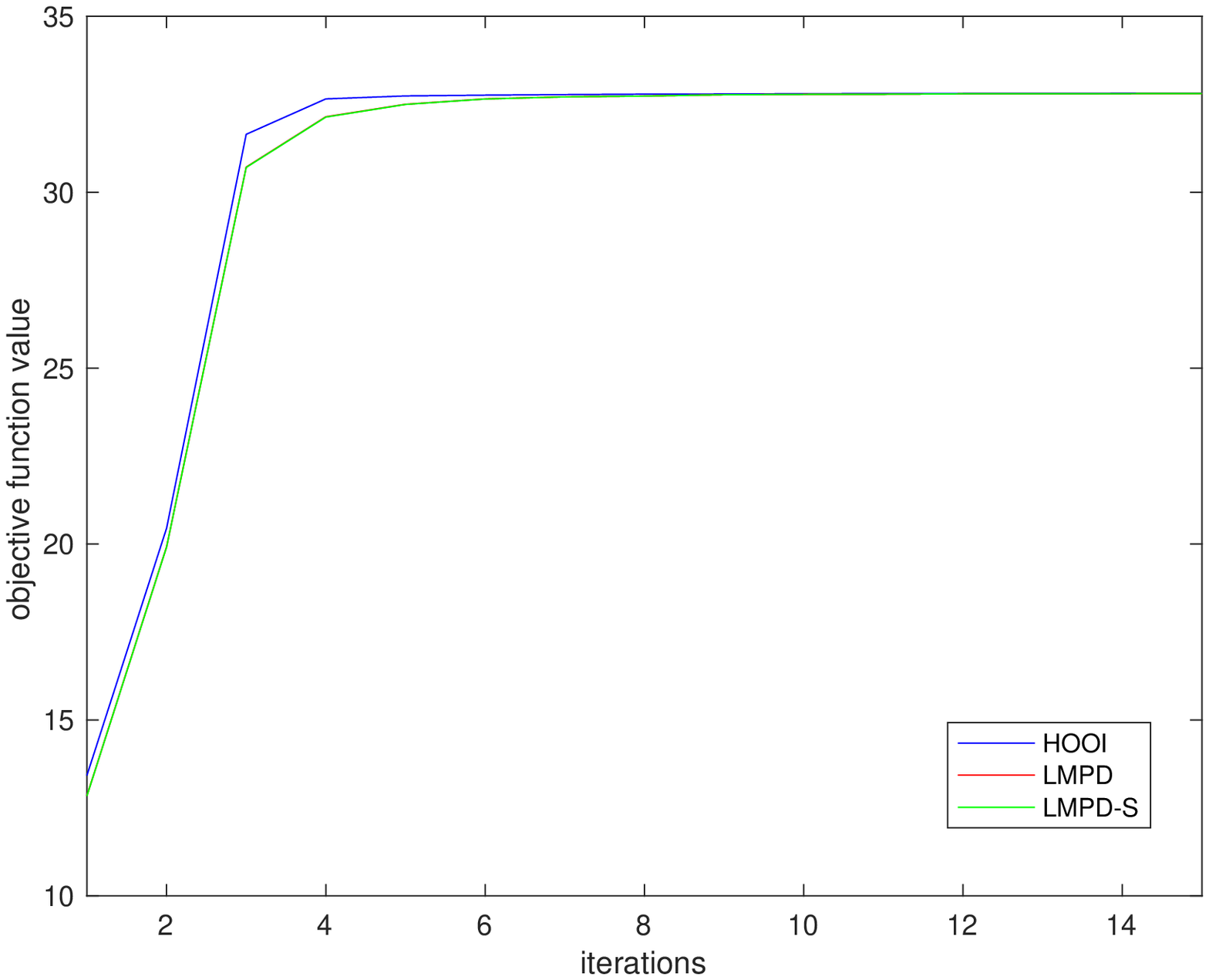}}
	\caption{Results for the tensor randomly generated in \Cref{example-2}.} 
	\label{figure-example-2}
\end{figure}

\begin{example}\label{example-3}
	We randomly generate two tensors in $\CC^{10\times 10\times 10}$, and run HOOI, LMPD and LMPD-S ($\gamma=0.01$) with ($r_1,r_2,r_3$) = (1,1,2).
	The distances of successive iterates $\|\omega^{(k,i+1)}-\omega^{(k,i)}\|$ are shown in \Cref{figure-example-3}.
	It can be seen that HOOI and LMPD fail to converge globally,
	while LMPD-S has a much better convergence performance.
\end{example}

\begin{figure}[tbhp]
	\centering
	\subfloat[first tensor]{\includegraphics[width=0.5\textwidth]{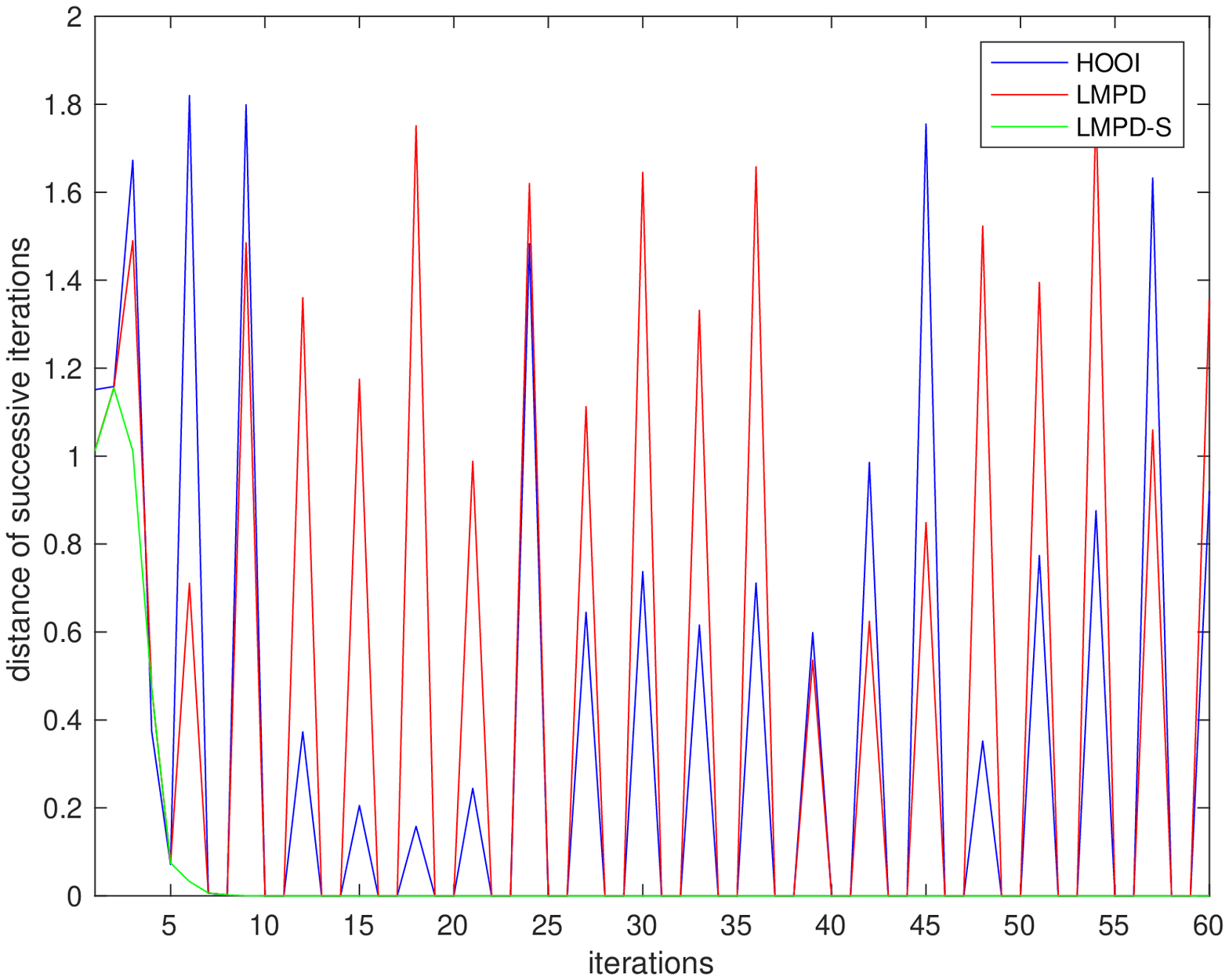}}\!\!\!
	\subfloat[second tensor]{\includegraphics[width=0.5\textwidth]{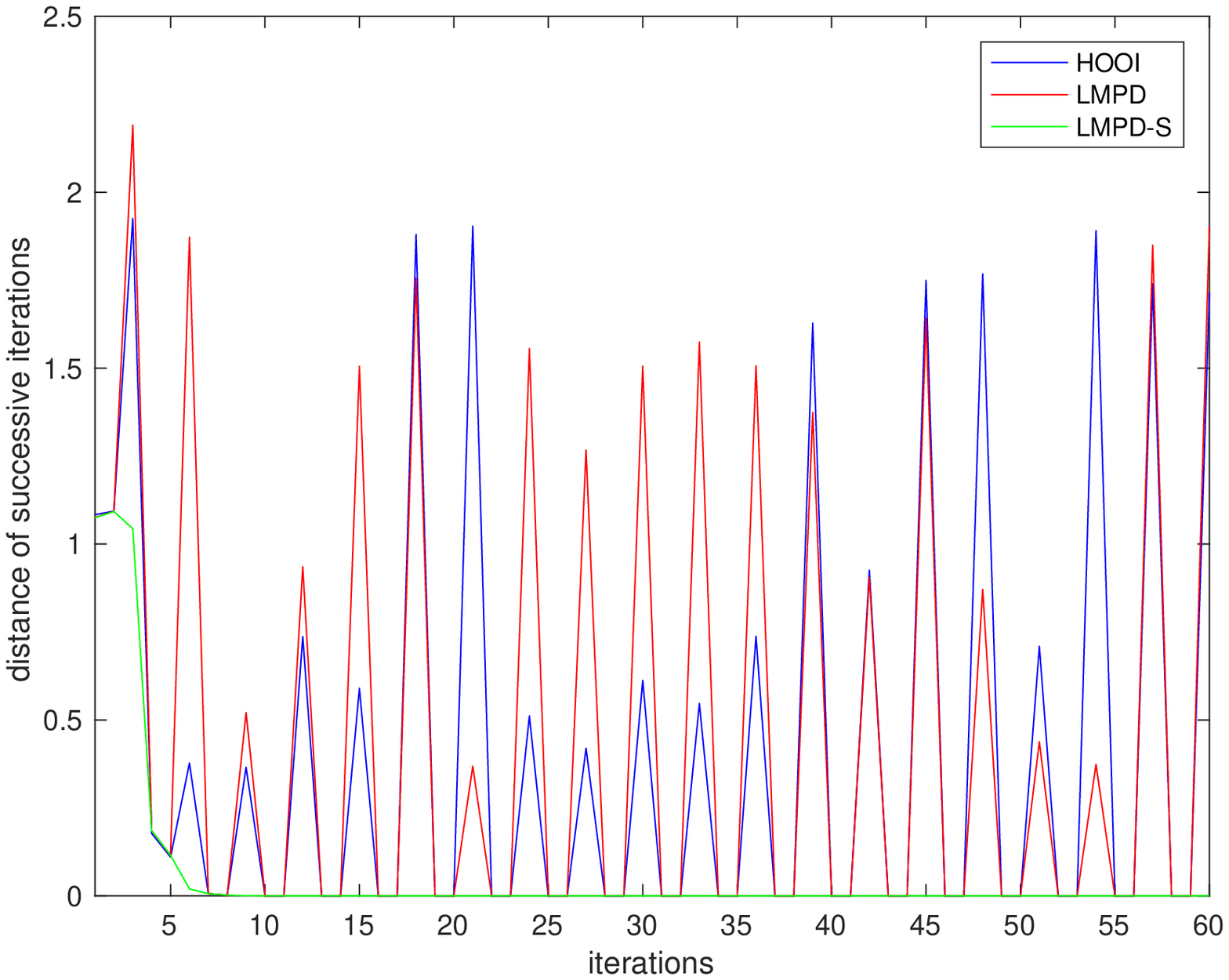}}
	\caption{Results for the tensors randomly generated in \Cref{example-3}.} 
	\label{figure-example-3}
\end{figure}

\section{Conclusions}\label{sec:conclusion}

Motivated by the objective functions \eqref{cost-fn-general-1} and \eqref{cost-fn-general-2}, which are defined on the product of complex Stiefel manifolds and have wide applications in tensor approximation, 
in the general setting, 
we propose Algorithm APDOI and its shifted version APDOI-S to solve optimization over the product of complex Stiefel manifolds based on the matrix polar decomposition. 
When the general objective function \eqref{definition-f} is multiconvex, we establish their weak convergence and global convergence based on the \L{}ojasiewicz gradient inequality. 
Then, based on the Morse-Bott property, we prove the linear convergence rate, when the limit is a scale (or unitarily) semi-nondegenerate point. In particular, we show that the objective functions \eqref{cost-fn-general-1} and \eqref{cost-fn-general-2} are both multiconvex. 

As the symmetric variant, for objective function \eqref{cost-fn-general-1-s}, which has wide applications in symmetric tensor approximation, in the general setting, we propose Algorithm PDOI and its shifted version PDOI-S to solve optimization over a single complex Stiefel manifold based on the matrix polar decomposition. 
When the general objective function \eqref{cost-function-genral-h} is convex, we similarly establish their weak convergence, global convergence and linear convergence rate. 
We present examples in \Cref{lemma-convex-tenso-form} to illustrate why the objective function \eqref{cost-fn-general-1-s} may be convex in many cases. 
It is a topic for further study to investigate if we can remove this assumption, while still guaranteeing convergence. 

\bibliographystyle{siamplain}
\bibliography{ReferenceTensor}

\end{document}